\documentclass[11pt]{amsart}
\usepackage{amsmath,amsfonts,amssymb,amsxtra,latexsym,amscd,enumerate,amsthm}

\usepackage{latexsym}
\usepackage{epsfig}
\usepackage{verbatim}

\def\r{\mathcal{R}}
\def\rr{\rho}

\def\a{\alpha}
\def\E{\mathcal{E}}

\def\p{\Phi}
\def\v{\varphi}
\def\ep{\epsilon}

\def\l{\Lambda}

\def\o{\omega}

\def\R{\mathbb{R}}

\def\c{\mathcal{C}}
\def\C{\mathbb{C}}
\def\Z{\mathbb{Z}}

\def\P{\mathbb{P}}
\def\p{\mathcal{P}}

\def\F{\mathcal{F}}
\def\J{\mathcal{J}}
\def\I{\mathcal{I}}
\def\s{\mathbf{s}}

\def\S{\mathcal{S}}
\def\H{\mathcal{H}}
\def\T{\textsf{T}}
\def\A{\mathcal{A}}
\def\B{\mathcal{B}}
\def\LL{\mathcal{L}}
\def\TT{\mathbb{T}}
\def\N{\mathbb{N}}

\def\f{\mathcal{F}}
\def\W{\mathcal{W}}
\def\Z{\mathbb{Z}}
\def\beq{\begin{equation}}
\def\eeq{\end{equation}}

\def\beq{\begin{equation}}
\def\eeq{\end{equation}}

\newtheorem{thm}{Theorem}

\newtheorem{prop}[thm]{Proposition}
\newtheorem{claim}[thm]{Claim}

\newtheorem{defn}[thm]{Definition}
\newtheorem{nt}[thm]{Notation}
\newtheorem{cor}[thm]{Corollary}
\newtheorem{obs}[thm]{Observation}

\newtheorem{d0}[thm]{Definition}
\newtheorem{o0}[thm]{Observation}
\newtheorem{c0}[thm]{Corollary}
\newtheorem{t1}[thm]{Theorem}
\newtheorem{l1}[thm]{Lemma}
\newtheorem{p1}[thm]{Proposition}

\begin{document}
\title[]{The pointwise convergence of Fourier Series (II)\\
Strong $L^1$ case for the lacunary Carleson operator}

\author{Victor Lie}

\date{\today}
\address{Department of Mathematics, Purdue, IN 46907 USA}

\email{vlie@math.purdue.edu}

\address{Institute of Mathematics of the
Romanian Academy, Bucharest, RO 70700, P.O. Box 1-764, Romania.}

\thanks{The author was supported by the National Science Foundation under Grant No. DMS-1500958. Parts of the current paper were extended and finalized while the author was in residence at the Mathematical Sciences Research Institute in Berkeley, California, during the Spring 2017 semester. This paper was submitted in October 2017 and accepted for publication in \emph{Advances in Mathematics} in February 2019.}


\maketitle

\begin{abstract}
We prove that the lacunary Carleson operator is bounded from $L \log L$ to $L^{1}$. This
result is sharp.

The proof is based on two newly introduced concepts:
\begin{enumerate}
\item the \emph{time-frequency regularization of a measurable set} and

\item the \emph{set-resolution of the time-frequency plane at $0-$frequency}.
\end{enumerate}

These two concepts  will play the central role in providing a special tile decomposition adapted to the interaction between the \emph{structure} of the lacunary Carleson operator and the corresponding  \emph{structure} of a fix measurable set.

Another key insight of our paper is that it provides for the first time a simultaneous treatment of families of tiles with \emph{distinct} mass parameters. This should be regarded as a fundamental feature/difficulty of the problem of the pointwise convergence of Fourier Series near $L^1$, context in which, unlike the standard $L^p,\:p>1$ case, \emph{no decay} in the mass parameter is possible.
\end{abstract}
$\newline$

\section{\bf Introduction}\label{intro}

In this paper we continue the author's dedicated work in investigating one of the central and oldest themes in the area of harmonic analysis that of the almost everywhere convergence of the Fourier Series.

The main result of our paper is the following:
$\newline$

\noindent\textbf{Main Theorem.}
 \emph{Let $\{n_j\}_{j\in\N}\subset\N$ be a lacunary sequence.\footnote{Recall here that a sequence $\{n_j\}_{j\in\N}\subset\N$ is called \textit{lacunary} iff $\underline{\lim}_{j\:\rightarrow\:\infty}\: \frac{n_{j+1}}{n_j}>1$.} Define the lacunary Carleson operator associated with $\{n_j\}_{j\in\N}\subset\N$ as
$$C_{lac}^{\{n_j\}_j}:\,C^{\infty}(\TT)\,\mapsto\,L^{\infty}(\TT)$$ with\footnote{Throughout this paper, for notational simplicity, we ignore the \emph{p.v.} symbol.}
\beq\label{carlacc}
C_{lac}^{\{n_j\}_j}f(x):=\sup_{j\in\N}\left|\int_{\TT}e^{2\pi\,i\,n_j\,(x-y)}\,\cot(\pi\,(x-y))\,f(y)\,dy\right|\:.
\eeq
Then the following holds:
\beq\label{carlblac}
\|C_{lac}^{\{n_j\}_j}f\|_{1}\leq C\,\|f\|_{L\log L}\,,
\eeq
where here $C=C(\{n_j\}_j)>0$ is constant independent of $f$.
$\newline$
\indent Moreover, this result is sharp.}
$\newline$

The proof of this result relies on the following key relation\footnote{For a more detailed account of this, please see the statement of the Main Theorem B at the end of this section.}:
\beq\label{carlac1aaa}
\|\chi_{F}\,C_{lac}^{*}(g)\|_{L^1}\leq C_1\,|F|\,\log\frac{4}{|F|}\,\|g\|_{\infty}\:,
\eeq
for any $F\subseteq\TT$ measurable and any function $g\in L^{\infty}(\TT)$. Here $C_{lac}^{*}$ stands for the adjoint\footnote{This should be understood as the adjoint of the \emph{linearized} Carleson operator - see \eqref{carlac11}, where the linearized function $N(\cdot)$ is regarded as an arbitrary measurable function independent of the input function.} lacunary Carleson operator and $C_1>0$ is an absolute constant.

In order to show \eqref{carlac1aaa} and thus prove the central part of our Main Theorem, we rely on three new insights within the generic problem of the pointwise convergence of the Fourier Series near $L^1$:
\begin{itemize}

\item we introduce the concept of the \emph{time-frequency regularization (\textbf{TFR}) of a set \footnote{Throughout the paper, all the sets consisting of real numbers are assumed to be measurable.}} - this is based on an algorithm designed to identify structures within the set $F$ in \eqref{carlac1aaa} and relies on embedding the level sets of the Hardy-Littlewood maximal operator applied to $F$ into a union of subsets having what we call ``uniform" structure, that translates, in particular, into the fact that each such subset can be represented as a union of same-length disjoint intervals that roughly have similar $F-$density.\footnote{If $I$ is a given interval one should think at its $F-$density (or mass) as given by $\approx \frac{|I\cap F|}{|I|}$.}

\item relying on the first item, we further introduce the fundamental concept of the \emph{set-resolution of the time-frequency plane (at $0$-frequency)} - this connects the structure of the lacunary Carleson operator, more precisely its corresponding tile decomposition, with the structure of the set $F$. Our entire proof is indissolubly tied to this type of structure-analysis/compatibility.

\item in the final/main part of our proof we develop a first approach to the \emph{simultaneous treatment} of the tile families with \emph{distinct mass-parameter} that further relies on the notion of tree foliation of the time-frequency plane introduced in Section \ref{foltrees1}. This approach to the mass-analysis of our tiles is a key aspect with which one must deal in order to get sharp results on the pointwise convergence of Fourier Series near $L^1$. Indeed, the independent treatment of the tile families having distinct mass parameters is only possible when one can get decay in the mass parameter, situation that is bound to the $L^p-$case with $p>1$.
\end{itemize}

In a nutshell, this work aims to reveal some subtle connections between the structure of the underlying set $F$ and the corresponding structure of the frequencies of the tiles appearing in the time-frequency decomposition of the (lacunary) Carleson operator with the latter further reflected in the properties of the linearizing measurable function $N$ in \eqref{carll0}. Based on our insight and current progress, we modify the classical $L\log L$-conjecture on the pointwise convergence of the Fourier Series near $L^1$ - see Section 1.2. - into two distinct, sharper\footnote{Either of these conjectures imply the $L\log L$ conjecture.} conjectures that address separately the $L^{1,\infty}$-bound and the $L^{1}$-bound of the Carleson operator - for this please see Conjecture 1 and Conjecture 2, respectively. One should regard our main result in the present paper as a support and indication for Conjecture 2.\footnote{For the same reasons as those revealed in the present paper, if true, Conjecture 2 is automatically sharp.}

Finally, we reiterate our belief that any significant progress on either Conjecture 1 or 2 must eventually involve additive combinatoric techniques together with a very careful analysis of the structural properties of the set $F$ relative to the structure of the set of frequencies of the tiles in the decomposition of the Carlson operator. This will very likely rely on and develop some of the concepts introduced in this paper and listed in the above itemization. We plan to investigate these directions in our future work.

\subsection{Historical background}

The history of this problem was initiated  by J. Fourier, who in his study on the heat propagation (\cite{Fou}), had the idea of representing a (suitable) function as a (possibly infinite) superposition of sinuses and cosines at distinct frequencies.

The question of when such (pointwise) representations make sense is naturally related with the regularity properties of the  function that is represented and was one of the main themes of research within the mathematical analysis during the 19th century and early 20th century. After successive investigations of key leading figures in mathematics such as Dirichlet, Cauchy and Riemann it came as a complete surprise when Du Bois Reymond (\cite{Bois}) showed that there are continuous periodic functions that do not admit \textit{everywhere} such a representation - that is, there are continuous functions that do not agree everywhere with their Fourier series representation.\footnote{Moreover, there exists a continuous function for which its Fourier Series is divergent at a given point (and in fact at any rational point).} After H. Lebesgue (\cite{Leb}) developed his theory of integration introducing the formalism about sets of (Lebesgue) ``measure" zero, N. Luzin (\cite{Lu}) was the first one to formulate a``reasonable" conjecture, that is: for every square-integrable function $f\in L^2(\TT)$ the correspondent Fourier Series converges Lebesgue almost everywhere to $f$. Surprisingly, his student A. Kolmogorov (\cite{Kol1}) was able to show in 1922 that there exists an $L^1(\TT)-$function whose corresponding Fourier series diverges (almost) everywhere.

From this point on the general belief was that Luzin's conjecture must be false. After more than forty years of misconceptions, L. Carleson (\cite{c1}) proved that in fact Luzin's conjecture is true. Shortly afterwards, R. Hunt (\cite{hu}) extended Carleson's result to functions $f\in L^p(\TT)$ with $1<p<\infty$. In this context, we mention that after Carleson's result two more proofs of Luzin's conjecture were provided: one due to C. Fefferman (\cite{f}) that became a landmark in the newly developing area of time-frequency analysis, and much later, another one due to M. Lacey and C. Thiele (\cite{lt3}).

\subsection{The Carleson operator: formulation of the main question; the two fundamental conjectures}

Having in mind this evolution of the topic, one naturally reaches the following general question (formally stated in \cite{lvPCFS1}):
$\newline$

\noindent\textbf{Main Question (Heuristic).} \textit{What is the behavior of the (almost everywhere) pointwise convergence of the Fourier Series between the two known cases for the Lebesgue-scale spaces $L^p(\TT)$:
\begin{itemize}
\item $p=1$, divergence of the Fourier Series (Kolmogorov)
\item $p>1$, convergence of the Fourier Series (Carleson-Hunt) ?
\end{itemize}}

Following \cite{lvPCFS1}, we now introduce

\begin{d0}\label{Cprop0} Let $Y$ be a r.i. (quasi-)Banach space. We say that
$Y$ is a $\mathcal{C}-$space iff $\:\exists\:\:C_0=C_0(Y)>0$ such that the Carleson operator defined by $$C:\,C^{\infty}(\TT)\,\mapsto\,L^{\infty}(\TT)$$ with
\beq\label{carll0}
Cf(x):=\sup_{N\in\N}\left|\int_{\TT}e^{2\pi\,i\,N\,(x-y)}\,\cot(\pi\,(x-y))\,f(y)\,dy\right|\:,
\eeq
obeys the relation\footnote{Recall that the weak-$L^1$ quasinorm is given by $\|f\|_{1,\infty}:
=\sup_{\lambda>0} \lambda\,|\{x\,|\,|f(x)|>\lambda\}|$.}
\beq\label{carlb}
\|Cf\|_{1,\infty}\leq C_0\,\|f\|_{Y}\:\:\:\:\:\forall\:\:f\in Y\:.
\eeq
\end{d0}
$\newline$
 With this definition, the Main Question above can be reformulated as follows:
$\newline$

\noindent\textbf{Main Question (Formal)} \textit{Give a satisfactory description of the Lorentz (or more general r.i. quasi-Banach) spaces $Y\subseteq L^1(\TT)$ that are also $\mathcal{C}-$spaces. If such exists, describe the maximal Lorentz $\mathcal{C}-$space $Y_0$.}
$\newline$

Once at this point we mention the classical conjecture on the near-$L^1$ behavior/convergence of Fourier Series that goes back to Carleson's work in \cite{c1}:
$\newline$

\noindent\textbf{The classical $L\log L$ Conjecture.}

\textit{The Lorentz space $L\log L$ is a $\mathcal{C}-$space.
Equivalently, we have that there exists $c>0$ absolute constant such that
\beq\label{largestt1}
\|Cf\|_{1,\infty}\leq c\,\|f\|_{L\log L}\:\:\:\:\:\:\:\:\:\forall\:f\in L\log L\,.
\eeq}
$\newline$

Returning to the formal enounce of the Main Question, we are now formulating the two fundamental conjectures that we believe are characterizing the behavior of the Carleson operator. These should be regarded as refinements of the above classical conjecture.

The first conjecture, regards \textit{per se} the problem of the largest Lorentz space where one has the almost everywhere convergence of the Fourier Series:
$\newline$

\noindent\textbf{Conjecture 1.} [\textsf{The $L^{1,\infty}$-behavior of the Carleson operator}]

\textit{The largest Lorentz space $Y_0\subseteq L^1(\TT)$ that is also a $\mathcal{C}-$space is given by
\beq\label{largest}
Y_0=L\sqrt{\log L}\,.
\eeq
In particular we have that
\beq\label{largest1}
C:\,L\sqrt{\log L}\,\mapsto\,L^{1,\infty}(\TT)\,.
\eeq}
$\newline$

\noindent\textsf{Heuristic Conjecture 1:} The elements that hint toward the formulation of this conjecture rely on the nature of the cancelations appearing among the tiles that belong to a multi-tower (see Section 4 in \cite{lvPCFS1} for the definition). Due to space limitation we will not provide here more details on this topic, but instead refer our reader to a result of T. Tao and J. Wright (\cite{TW}) that can serve as a very basic/rudimentary model for understanding the behavior of the Carleson operator restricted to a tower: if $T$ is a convolution operator arising from a Marcinkiewicz multiplier then $T:\:L\sqrt{\log L}\:\rightarrow\:L^{1,\infty}$ sharply.
\medskip

In contrast with the above, one may also ask about the \textit{strong} bounds for the Carleson operator. More precisely,
by a symmetric reasoning with the one in the formulation of the Main Question above, we can ask: ``\textit{what is the largest Lorentz space $Y_1\subseteq L^1(\TT)$ such that
\beq\label{largestL1}
C:\,Y_1\,\mapsto\,L^{1}(\TT)?
\eeq}

\noindent\textsf{Heuristic Conjecture 2:} The Carleson operator encompasses both the behavior of the Hardy-Littlewood maximal operator and that of the Hilbert transform. However, both these operators map \textit{sharply} $L\log L$ into $L^1$. As a consequence, one naturally arises to the following
$\newline$

\noindent\textbf{Conjecture 2.} [\textsf{The $L^{1}$-behavior of the Carleson operator}]

\textit{The largest Lorentz space $Y_1\subseteq L^1(\TT)$ that obeys \eqref{largestL1}
is given by
\beq\label{largestl1}
Y_1=L\log L\,.
\eeq}
$\newline$

Once at this point, let us briefly comment on the up to date progress on the above conjectures:

\begin{itemize}
\item regarding the $L^{1,\infty}-$behavior of the Carleson operator:
\begin{itemize}
\item on the \textit{positive} side, the best known result belongs to N. Antonov (\cite{An}) who showed that $Y=L\log L\log \log\log L$ is a Lorentz $\mathcal{C}-$space. A bit later, Arias-de-Reyna (\cite{Ar}) proved that $Y$ can be enlarged to a rearrangement invariant quasi-Banach space - denoted $QA$ - that strictly contains $L \log L \log\log\log L$. However, it is worth saying that this result is only apparently stronger, since, in \cite{MMR}, the authors show that under mild conditions on the fundamental function $\v$ the correspondent\footnote{ Throughout the paper, if $\v:\,\TT\,\rightarrow\,[0,\infty)$ is a function with the following properties: $\v(0)=0$ with $\v$ nondecreasing and concave we define the \textit{Lorentz space} $\l_{\v}$ as the r.i. Banach space of all the measurable functions $f\in L^{0}(\TT)$ such that $\|f\|_{\l_{\v}}:=\int_{0}^{1} f^{*}(s)\,d(\v(s)) <\infty$. The function $\v$ becomes now the fundamental function of $\l_{\v}$.} largest Lorentz space $\Lambda_{\v}$
contained in $QA$ is precisely given by Antonov's space $\Lambda_{\v}=L\log L\log\log\log L$.

Previous near-$L^1$ results were obtained by Sj\"olin (\cite{sj3}) and F. Soria (\cite{So1}, \cite{So2}).

\item on the \textit{negative} side,  Konyagin (\cite{koDivf}, \cite{koDivff}) proved that if $\phi(u)=o(u\sqrt{\frac{\log u}{\log\log u}})$ as $u\rightarrow \infty$ then the space $\phi(L)=\Lambda_{\bar{\phi}}$ is not a Lorentz $\mathcal{C}-$space, where here $\bar{\phi}(t):=\int_{0}^{t} s\,\phi(\frac{1}{s})\,ds$.

    Previous negative results were obtained by Chen (\cite{ch}), Prohorenko (\cite{P}) and K\"orner (\cite{Kor}).
\end{itemize}
\item regarding the $L^{1}-$behavior of the Carleson operator:
\begin{itemize}
\item the best and only known results are due to  Sj\"olin (\cite{sj3}) for the Walsh-Fourier case and to the author (\cite{lvCarl1}) for the case of Fourier Series. These results amount to saying that \eqref{largestL1} holds for
$Y_1=L (\log L)^2$.
\end{itemize}
\end{itemize}

We end the commentary on the behavior of the full Carleson operator by mentioning that all the (best) positive results mentioned above were based on extrapolation techniques. Few years ago, by using an approach that relied entirely on time-frequency methods, (\cite{lvCarl1}), the author re-proved all these best known positive results by a unified approach.

\subsection{The model problem - the lacunary Carleson operator: formulation of the main question; the two fundamental conjectures - analogues}

Given the deepness and difficulty of these questions it is natural to search for relevant \textit{model} problems. One of the most natural and prominent such models was born in the early 20th century once that the Littlewood-Paley theory of block-dyadic summation of the Fourier Series developed. In such a context, appeared naturally the question about the pointwise convergence of the partial Fourier Sums along lacunary subsequences.

Now, by analogy with our previous section presentation, we first introduce the following definition (initially developed in \cite{lvPCFS1})

\begin{d0}\label{Cprop} Let $Z$ be a r.i. (quasi-)Banach space.

\noindent i)  Assume $\{n_j\}_{j\in\N}\subset\N$ is a lacunary sequence. We say that
$Z$ is a $\mathcal{C}_{L}^{\{n_j\}_j}-$space iff $\:\exists\:C_1=C_1(Z,\,\{n_j\}_j)>0$ such that the $\{n_j\}_{j\in\N}$ - lacunary Carleson operator defined by
$$C_{lac}^{\{n_j\}_j}:\,C^{\infty}(\TT)\,\mapsto\,L^{\infty}(\TT)$$ with
\beq\label{carlac}
C_{lac}^{\{n_j\}_j}f(x):=\sup_{j\in\N}\left|\int_{\TT}e^{2\pi\,i\,n_j\,(x-y)}\,\cot(\pi\,(x-y))\,f(y)\,dy\right|\:,
\eeq
obeys the relation
\beq\label{carlblac}
\|C_{lac}^{\{n_j\}_j}f\|_{1,\infty}\leq C_1\,\|f\|_{Z}\:\:\:\:\:\forall\:\:f\in Z\:.
\eeq

\noindent ii) We say that $Z$ is a $\mathcal{C}_L-$space iff $Z$ is a $\mathcal{C}_{L}^{\{n_j\}_j}-$space for any lacunary sequence $\{n_j\}_{j\in\N}$.

Moreover, trough out the paper, if $Z$ is a $\mathcal{C}_L-$space, we will (often) express this as\footnote{Given $A,\,B>0$, throughout the paper,  we will use the notations $A\lesssim B$ and $B\gtrsim A$ to specify that there exists $C>0$ such that $A\leq C\,B$ and $B\leq C\, A$ respectively.}

\beq\label{carlblacuni}
\|C_{lac}f\|_{1,\infty}\lesssim \|f\|_{Z}\:\:\:\:\:\forall\:\:f\in Z\:,
\eeq
where here $C_{lac}$ stands for ``the generic" lacunary Carleson operator. \footnote{In \eqref{carlblacuni}, the implicit constant is allowed to depend on the specific choice of the lacunary sequence and on the space $Z$ but not on the function $f\in Z$.}
\end{d0}

We can now formulate the analogue of the Main Question raised in the previous section:
$\newline$

\noindent\textbf{Main Model Question (Formal).}
\textit{Give a satisfactory description of the Lorentz (or more general r.i. quasi-Banach) spaces $Z$ that are also $\mathcal{C}_L-$spaces. If such exists, describe the maximal Lorentz $\mathcal{C}_L-$space $Z_0$.}
$\newline$

Following now line by line the analogy with the full Carleson operator case, one can formulate the analogues of the fundamental Conjectures 1 and 2. Thus, in the context of the largest Lorentz space for which one has the almost everywhere convergence for any lacunary subsequence we have:
$\newline$

\noindent\textbf{Model Conjecture $1$.} [\textsf{The $L^{1,\infty}$-behavior of the lacunary Carleson operator}]
(S. Konyagin, ICM, Madrid 2006, \cite{ko1}.)

\textit{The largest Lorentz space $Z_0\subseteq L^1(\TT)$ that is also a $\mathcal{C}_{L}-$space is given by
\beq\label{largestlac}
Z_0=L\log\log L\,.
\eeq}
$\newline$

\noindent\textsf{Heuristic - Model Conjecture 1}\footnote{In this very succinct heuristic we will make use of the notations and definitions introduced in Sections \ref{discr1} and \ref{tilediscr1}.}:  The first important observation is that if $\{n_j\}_{j\in\N}$ is a lacunary sequence then the corresponding characters $\{e^{2\pi i n_j \cdot}\}_{j}$ behave similarly to a sequence of i.i.d random variables. As a consequence one has the ``trigonometric analogue" of Hincin's inequality under the name of Zygmund's inequality\footnote{In the Appendix we will provide a nice short proof of Zygmund's inequality based on the two newly introduced concepts of the time-frequency regularization of a set and of the set resolution of the time-frequency plane at a fix frequency.}: if $\{a_j\}_{j\in\N}\in l^2(\N)$, $\{n_j\}_{j\in\N}$ lacunary, and $F\subseteq\T$ measurable then
\beq\label{zygg}
\|\sum_{j=1}^{\infty} a_j\,e^{2\pi i n_j x}\|_{L^1(F)}\lesssim |F|\,(\log \frac{4}{|F|})^{\frac{1}{2}}\,\|\{a_j\}_{j}\|_{l^2(\N)}\:.
\eeq
Fixing now $n,\,N\in\N$ and assuming $|F|\approx 2^{-N}$, one has from \eqref{zygg}
$$\|\sum_{j=1}^{2^n} a_j\,e^{2\pi i n_j x}\|_{L^1(F)}\lesssim |F|\, N^{\frac{1}{2}}\,(\sum_{j=1}^{2^n}|a_j|^2)^{\frac{1}{2}}\:,$$
while a trivial $L^1$-summation argument gives
$$\|\sum_{j=1}^{2^n} a_j\,e^{2\pi i n_j x}\|_{L^1(F)}\lesssim |F|\,(\sum_{j=1}^{2^n}|a_j|)\:.$$
From this we deduce that if $|a_j|\approx 2^{-n}$ for $1\leq j\leq 2^n$, then
$$\|\sum_{j=1}^{2^n} a_j\,e^{2\pi i n_j x}\|_{L^1(F)}\lesssim |F|\,\min\{1,\,N^{\frac{1}{2}}\,2^{-\frac{n}{2}}\}\:.$$
Conclude that if we were in an ideal setting with all the tiles $P$ appearing in the time-frequency portrait of the lacunary Carleson operator $C_{lac}\equiv T$ having uniform $F-$mass, then
$$\|\sum_{n}\sum_{P\in\P_{n}} T_{P}^{*}g\|_{L^1(F)}\lesssim |F|\,\sum_{n}\min\{1,\,N^{\frac{1}{2}}\,2^{-\frac{n}{2}}\}\|g\|_{\infty}$$
$$\lesssim |F|\,\log\log\frac{4}{|F|}\,\|g\|_{\infty}\:,$$
which would justify at least the hope for the restricted weak-type form of the Model Conjecture 1, that is
\beq\label{restr}
\|C_{lac}(h\,\chi_F)\|_{1,\infty}\lesssim |F|\,\log\log\frac{4}{|F|}\,\|h\|_{\infty}\:.
\eeq

\medskip
Turning now towards the strong $L^1$ bounds for the lacunary Carleson operator, by analogy with \eqref{largestL1}, one
can ask: ``\textit{what is the largest Lorentz space $Z_1\subseteq L^1(\TT)$ such that
\beq\label{largestL1lac}
C_{lac}:\,Z_1\,\mapsto\,L^{1}(\TT)?
\eeq}

\noindent\textsf{Heuristic - Model Conjecture 2}: Notice that the same arguments served as a motivation for the full Carleson operator remain valid in the current setting: indeed, the lacunary Carleson operator $C_{lac}$ can also be thought of as subsuming the behavior of both the Hardy-Littewood maximal operator and that of the Hilbert transform.
\medskip

With these, we have
$\newline$

\noindent\textbf{Model Conjecture $2$.} [\textsf{The $L^{1}$-behavior of the lacunary Carleson operator}]

\textit{The largest Lorentz space $Z_1\subseteq L^1(\TT)$ that obeys \eqref{largestL1lac}
is given by
\beq\label{largestl1lac}
Z_1=L\log L\,.
\eeq}
$\newline$

\subsection{Resolution of the model problem and of the corresponding conjectures. Main results.}

Once at this point, we have the following remarkable fact: we can fully answer both of the above model conjectures, and moreover, provide the answer to the Main Model Question, thus completely solving the lacunary Carleson operator case.

\subsubsection{Resolution of the Main Model Question and Model Conjecture $1$}

From author's previous work we first have:
$\newline$

\noindent\textbf{Answer Main Model Question (\cite{lvKony1}, \cite{lvPCFS1}).}:

\textit{Define $\v_0:\:[0,1]\,\rightarrow\,\R_{+}$ as $\v_0(s):=s\,\log\log\frac{17}{s}\,\log\log\log\log\frac{17}{s}$.}

\textit{Let now $\v:\:[0,1]\,\rightarrow\,\R_{+}$ be a non-decreasing concave function with $\v(0)=0$. Then we have:}

i) \textit{If $\:\:\:\underline{\lim}_{{s\rightarrow 0}\atop{s>0}}\:\frac{\v(s)}{\v_0(s)}>0\:\:\:$ then the Lorentz space $\l_{\v}$ is
a $\mathcal{C}_L-$space.}

ii) \textit{If $\:\:\:\overline{\lim}_{{s\rightarrow 0}\atop{s>0}}\:\frac{\v(s)}{\v_0(s)}=0\:\:\:$ then the Lorentz space $\l_{\v}$ is
not a $\mathcal{C}_L-$space.}

iii) \textit{If $\:\:\:\underline{\lim}_{{s\rightarrow 0}\atop{s>0}}\:\frac{\v(s)}{\v_0(s)}=0<\overline{\lim}_{{s\rightarrow 0}\atop{s>0}}\:\frac{\v(s)}{\v_0(s)}\:\:\:$ then both scenarios are possible. More precisely, one can choose a $\v$ such that
$\l_{\v}$ is a $\mathcal{C}_L-$space while for another proper choice of $\v$ one has that $\l_{\v}$ is not a $\mathcal{C}_L-$space.
$\newline$
\indent Moreover, letting $\W$ be the quasi-Banach space defined by\footnote{Throughout the paper we will use the following convention: $\log k$ stands for $\log_{2} k$.}:
$$\W:=\{f:\:\TT\mapsto C\,|\,f\:\textrm{measurable},\:\|f\|_{\W}<\infty \}\:,$$
where
$$\|f\|_{\W}:=\inf\left\{\sum_{j=1}^{\infty}(1+\log j)\|f_j\|_1\,\log\log\frac{4\,\|f_j\|_{\infty}}{\|f_j\|_1}\:\:\bigg|\:\:
\begin{array}{cl}
f=\sum_{j=1}^{\infty}f_j,\:\\
\sum_{j=1}^{\infty}|f_j|<\infty\:\textrm{a.e.}\\
f_j\in L^{\infty}(\TT)
\end{array}
\right\}\;.$$
we have \beq\label{K1}
\|\,C_{lac}(f)\,\|_{1,\infty}\lesssim \|f\|_{\W}\:.
\eeq
and thus $Z=\W$ is a $\mathcal{C}_L-$space.}

From this, we immediately deduce:

\begin{c0}\label{Orlicz}[\textsf{Maximal characterization}] (\cite{lvPCFS1})

Let $\v:\:[0,1]\,\rightarrow\,\R_{+}$ be a non-decreasing concave function with $\v(0)=0$.
Assume that there exists
\beq\label{las}
\lim_{{s\rightarrow 0}\atop{s>0}}\:\frac{\v(s)}{\v_0(s)}\in [0,\,\infty]\:.
\eeq
Then the largest Lorentz $\mathcal{C}_L$-space $\l_{\v}$ for which $\v$ obeys \eqref{las} is given by
$$Z_0=L\,\log\log L\,\log\log\log\log L\:.$$
In particular, taking in \eqref{las} the function $\v(s)=s\,\log\log \frac{4}{s}$, one further deduces
that the Model Conjecture $1$ is false.
\end{c0}

Finally, we record

\begin{c0}\label{Halo}[\textsf{Restricted weak-type does not imply weak-type}]

The lacunary Carleson operator obeys the following:
\begin{itemize}
\item $C_{lac}$ is a sublinear, translation invariant operator;
\item $C_{lac}$ is of restricted weak type $(L\log\log L, \,L^1)$ and hence \eqref{restr} holds validating our heuristic from before.
\item  $C_{lac}$ is not of weak type $(L\log\log L, \,L^1)$.
\end{itemize}
 \end{c0}

This last corollary disproved the so-called ``generalized Halo Conjecture'', conjecturing the equivalence of weak-type and restricted--weak-type boundedness for translation-invariant sublinear operators on Lorentz (Orlicz) spaces.
\medskip

We mention here that previous results on the $L^{1,\infty}$-behavior of the lacunary Carleson operator (or its Walsh analogue) were obtained chronologically in \cite{Zyg}, \cite{ko2}, \cite{An1}, \cite{ko1}, \cite{LaDo}, \cite{lvKony1} and  \cite{dp} (for more details on this, please consult \cite{lvPCFS1}).

\subsubsection{Resolution of the Model Conjecture $2$. Main results.}

Regarding the Model Conjecture $2$, no previous results were known besides the partial answer for the stronger case of the full Carleson operator provided in \cite{sj3} (for the Fourier-Walsh case) and \cite{lvCarl1} (for the Fourier case). This however, is now settled through the present paper. Indeed, we have:
$\newline$

\noindent\textbf{Main Theorem A.} [\textsf{Sharp $L^1$-strong bounds for $C_{lac}$}]

\textit{Model Conjecture $2$ is true.}
$\newline$

Our Main Theorem A above is a consequence of the following result:
$\newline$

\noindent\textbf{Main Theorem B.} [\textsf{Restricted type version}]

\textit{Let $\{n_j\}_{j\in\N}\subset\N$ be any given lacunary sequence. Then the following are true:
$\newline$
i) There exists $C_1>0$ such that for any $F\subseteq\TT$ measurable and any function $g\in L^{\infty}(\TT)$ the \underline{adjoint} lacunary Carleson operator $C_{lac}^{*}$ obeys
\beq\label{carlac1aa}
\|\chi_{F}\,C_{lac}^{*}(g)\|_{L^1}\leq C_1\,|F|\,\log\frac{4}{|F|}\,\|g\|_{\infty}\:.
\eeq
$\newline$
ii) There exist $C_2>0$ such that for any $\eta\in(0,1]$ there exists $F(\eta)=F\subseteq\TT$ with $|F|=\eta$ such that
\beq\label{carlac1bb}
\|C_{lac}(\chi_F)\|_{L^1}\geq C_2\,|F|\,\log\frac{4}{|F|}\:\:\:\:\;\:\:\:\:\:.
\eeq}

\subsection{Structure of the paper}

In this final (sub)section of the introduction we detail the structure of the paper:

\begin{itemize}

\item following the discretization in \cite{f}, Section \ref{discr1} takes care of the standard decomposition of the (lacunary) Carleson operator into a superposition of operators that are well-localized in the time-frequency plane (\textit{i.e.} adapted to tiles of area one).

\item in Section \ref{reduct} we show that Main Theorem B implies Main Theorem A and also prove the easy part of Main Theorem B, that is part ii).

\item Section \ref{resol} introduces the key concept of the time-frequency regularization (\textbf{TFR}) of a set.

\item Section \ref{tilediscr1} unravels the first chapter of the tile discretization discussing the notions of a mass and $F-$mass of a tile as well as the ordering relation among tiles. The section ends with an algorithm describing the so called tree $*-$foliation of the time-frequency plane.

\item in Section \ref{tffolF} we introduce the new key concept of the set resolution of the time frequency plane at a fix frequency which connects the structure of the set $F$ with that of the lacunary Carleson operator. Our reasonings  relies on the \textbf{TFR} concept introduced in Section \ref{resol}. This is the second (and final) chapter of our tile discretization.

\item in Section \ref{prep} we split the proof of the Main Theorem B part i) into four theorems\footnote{Theorem \ref{clust} turns out it can be reduced to the classical statement that the maximal Hilbert transform acts boundedly from $L\log L$ to $L^1$.} - Theorems \ref{clust}, \ref{ortoF1}, \ref{noosc} and \ref{ortog}. In the same section we state and prove the Main Lemma.

\item in Section \ref{pfth1} we present the proof of Theorem \ref{ortoF1} that can be reduced to the proofs of the $L^2$- boundedness of the Carleson operator and of the Main Lemma.

\item in Section \ref{thL1} we prove Theorem \ref{noosc} by closely following the spirit of our approach in \cite{lvKony1}.

\item Section \ref{thL2} is the most technical component of our paper. It seeks to reduce the proof of Theorem \ref{ortog} to that of Proposition \ref{mainpropred}. In this section we will use the entire machinery developed by the author in the previous papers that concern the Carleson operator as well as the tools introduced in the present paper with a key emphasis on the properties of the (second) tile descritization developed in Section \ref{tffolF}.

\item in Section \ref{Errterm} we treat the error term appearing in Proposition \ref{mainpropred}.

\item Section \ref{mainthm} deals with the main term in Proposition \ref{mainpropred}. This is the part where we develop a new approach in understanding the subtle interaction between the structure of the linearizing function $N$ encapsulated in the properties of set $E:=\{E(P)\}_{P\in\P}$ and the corresponding structure of the measurable set $F$ that helps us to treat simultaneously the family of tiles with distinct mass parameter as well as those with distinct $F-$mass parameter.

\item Section \ref{FR} incorporates several final remarks.

\item Section \ref{Zygmund} constitutes the Appendix of our paper and provides some light into the motivation and relevance that are hidden behind the time-frequency regularization of a set, concept that, as we will see, reveals interesting connections between additive combinatorics and time-frequency analysis areas.
\end{itemize}

\section{Operator discretization - First stage}\label{discr1}

Since the maximal operator under discussion is nothing else than a \textit{lacunary} version of the Carleson operator,
as usual in such context, we will use time-frequency methods to analyze it.

Now, the study of our operator
\beq\label{carlac}
C_{lac}^{\{n_j\}_j}f(x)\approx\sup_{j}|S_{n_j}f(x)|\:\:\:\textrm{with}\:f\in C^1(\TT)\,,
\eeq
may be canonically reduced to the analysis of
\beq\label{carlac1}
Tf(x):=\sup_{j\in\N}\,\left|\int_{\TT} \frac{1}{x-y}\,e^{2\pi\,i\,n_j\,(x-y)}\,f(y)\,dy\right| \:,
\eeq
where here $\{n_j\}_j$ is a prescribed lacunary sequence of positive integers.

Applying Fefferman's approach, (\cite{f}), we perform the following steps:
\begin{itemize}
\item firstly, linearizing our operator, we construct $N:\:\TT\rightarrow\:\{n_j\}_j\;$ measurable function such that\footnote{For technical reasons we will erase the term $N(x)\,x$ in the phase of the exponential, as later
in the proof this will simplify the structure of the adjoint operators $T_P^{*}$.}
\beq\label{carlac11}
Tf(x)=\int_{\TT} \frac{1}{x-y}\,e^{-2 \pi\,i\,N(x)\,y}\,f(y)\,dy\,.
\eeq

\item next, using the dilation symmetry of the kernel, we decompose
$$\frac{1}{y}=\sum_{k\geq 0} \psi_k(y)\:\:\:\:\:\:\:\:\:\forall\:\:0<|y|<1\:,$$
where  $\psi_k(y):=2^{k}\psi(2^{k}y)$ (with $k\in \N$) and $\psi$ an odd
$C^{\infty}$ function such that
$\operatorname{supp}\:\psi\subseteq\left\{y\in \R\:|\:2<|y|<8\right\}$.

\item deduce now that in the newly created context the following holds
\beq\label{carllac1}
Tf(x)=\sum_{k\geq 0}\int_{\TT}e^{-2\pi\,i\,N(x)\,y}\,\psi_{k}(x-y)\,f(y)\,dy\:.
\eeq

\item given $k\in\N$, we partition the time-frequency plane in tiles\footnote{Rectangles of area one.} of the form $$P=[\o,I]\,,$$
where $\o,\,I$ are dyadic intervals\footnote{With respect to the canonical dyadic grids on $\R$ and respectively $\TT$.} such that $|\o|=|I|^{-1}=2^{k}$.
With this done, we let $\P(k)$ be the collection of tiles at scale $k$ and define the collection of all the tiles be
\beq\label{tiless}
\P=\bigcup_{k\in\N}\P(k)\,.
\eeq

\item to each tile $P=[\o,I]\in\P$ we associate a set encapsulating the ``amount" of the graph of $N$ contained in $P$, that is
$$E(P):=\left\{x\in I\:|\:N(x)\in
\o\right\}\,.$$

\item finally, for $P=[\o,I]\in\P(k)$, we define the operators
$$T_{P}f(x)=\left\{\int_{\TT}e^{-2\pi\,i\,N(x)\,y}\,\psi_{k}(x-y)\,f(y)\,dy\right\}\chi_{E(P)}(x)\:,$$
and conclude that
\beq\label{discret}
Tf(x)=\sum_{P\in\P} T_P f(x)\,.
\eeq
\end{itemize}

Notice that if we think at $N:\:\TT\rightarrow\:\{n_j\}_j\;$ as a predefined measurable function, then, the above decomposition does not depend on the function $f$. Using this perspective will be enough to show that the bounds on $T$ do not depend on $N$.

\begin{obs}\label{tileadj}
Since our procedure will involve the support of the adjoint operators $\{T_P\}_{P\in\P}$ we first isolate an elementary piece $T_P$ (and the corresponding $T_P^*$) and briefly introduce several notations that we will use in our later reasonings:

For $P=[\o_P,I_P]\in\P$ we set $c(I_P)$ the center of the interval $I_P$ and define $$I_{P^*}=[c(I_P)-\frac{17}{2}|I_P|,\,c(I_P)-\frac{3}{2}|I_P|]\cup[c(I_P)+\frac{3}{2}|I_P|,\,c(I_P)+\frac{17}{2}|I_P|].$$
We then have the following properties:
\beq\label{support}
\textrm{supp}\,T_{P}\subseteq I_P\:\:\:\textrm{and}\:\:\:\textrm{supp}\,T_{P}^{*}\subseteq I_{P^*}\:.
\eeq
Notice that we can express the set containing the support of $T_{P}^{*}$ as
\beq\label{support1}
I_{P^*}=\bigcup_{r=1}^{14}I_{P*}^r\,,
\eeq
with each $I_{P^*}^r$ a dyadic interval of length $|I_P|$.

We now set
\beq\label{support2}
\tilde{I}_{P}=[c(I_P)-\frac{17}{2}|I_P|,\,c(I_P)+\frac{17}{2}|I_P|]\,,
\eeq
and notice that $ I_P\cup I_{P^*}\subset \tilde{I}_{P}=17 I_{P}\:.$

Finally, using a standard reasoning, we can reduce our analysis to the following situation that will apply from now on throughout the entire paper:
\beq\label{supportinttt}
\textrm{if}\:P_1,\,P_2\in\P\:\textrm{s.t.}\:|I_{P_1}|\not=|I_{P_2}|\:\Rightarrow\:|I_{P_1}|\leq 2^{-10}\,|I_{P_2}|\:\textrm{or}\:|I_{P_2}|\leq 2^{-10}\,|I_{P_1}|\,.
\eeq
As a consequence, we deduce that if $P_1,\,P_2\subset \P$ such that $I_{P_1}\subsetneq I_{P_2}$  then
\beq\label{supportint}
\textrm{supp}\,T_{P_1}^{*}\cap\textrm{supp}\,T_{P_2}^{*}=\emptyset\:.
\eeq
\end{obs}
$\newline$

\section{Reducing our Main Theorem A to the Main Theorem B. Proof of Main Theorem B part ii)}\label{reduct}

\subsection{Main Theorem B $\:\Rightarrow\:$  Main Theorem A}

Assuming for the moment that Main Theorem B holds we have:

Part ii) trivially implies that no larger space than $L\log L$ can satisfy \eqref{largestL1lac} while
 part i) proves that $L\log L$ indeed obeys \eqref{largestL1lac}. To see this last fact, we proceed as follows:

For each $l\in\Z$ define $$F_{l}:=\{x\in\TT\,|\,|f(x)|\in [2^{l},\,2^{l+1})\}\:.$$
Then notice that we have
\beq\label{lloglloglogl}
\|f\|_{L\log L}\approx\sum_{l\in\Z} 2^{l}\,|F_l|\, \log \frac{4}{|F_l|}\:.
\eeq
After the linearization of the lacunary Carleson operator, we deduce that for a suitable $g\in L^{\infty}$ with $\|g\|_{\infty}=1$, one has
\beq\label{carlac1bc}
\eeq
$$\|C_{lac}^{\{n_j\}_j}(f)\|_{L^1}=\int (C_{lac}^{\{n_j\}_j})^{*}(g)\,f\leq\sum_{l\in Z}
\left|\int_{F_l} (C_{lac}^{\{n_j\}_j})^{*}(g)\,(\chi_{F_l}\,f)\right|$$
$$\leq \sum_{l\in Z}\|\chi_{F_l}\,f\|_{\infty}\,\int_{F_l} |(C_{lac}^{\{n_j\}_j})^{*}(g)|\,
\lesssim\,\sum_{l\in Z} 2^{l}\,|F_l|\,\log\frac{4}{|F_l|}\approx\|f\|_{L\,\log L}\:.$$

\subsection{Proof of the Main Theorem B, part ii)}

Fix $\eta\in(0,\frac{1}{2}]$.  Take now $F=[\frac{1}{2}-\eta,\,\frac{1}{2}]$. We will show that\footnote{For the sake of the philosophy behind the cancelation properties of the kernel $\frac{1}{y}$, we want to present a proof that avoids the precise computation of the Hilbert transform of a characteristic function of an interval.}
\beq\label{carlac1bb1}
\|C_{lac}^{\{n_j\}_j}(\chi_F)\|_{L^1}\gtrsim\,|F|\,\log\frac{4}{|F|}\:\:\:\:\;\:\:\:\:\:.
\eeq

Choosing now in \eqref{carlac11} the linearizing function $N(x)\equiv0$ we deduce that $|C_{lac}^{\{n_j\}_j}(\chi_F)|\geq |H(\chi_{F})|$ where here $H$ stands for the Hilbert transform.

Consequently, using standard duality, we have that
$$\|C_{lac}^{\{n_j\}_j}(\chi_F)\|_1\geq \int_{\TT}|H(\chi_{F})|\geq |\int_{\TT}\chi_{F}\,H^{*}(\chi_{[\frac{1}{2},1]})|
=|\int_{\TT}\chi_{F}\,H(\chi_{[\frac{1}{2},1]})|\,.$$
In the same spirit with \eqref{carllac1}, we write
\beq\label{h}
H(\chi_{[\frac{1}{2},1]})(x)=\sum_{k\geq 0}H_{k}(\chi_{[\frac{1}{2},1]})(x):=\sum_{k\geq 0}\int_{\TT}\psi_{k}(x-y)\,(\chi_{[\frac{1}{2},1]})(y)\,dy\:.
\eeq
Choosing now $k_{F}\in\N$ such that $2^{-k_{F}}\approx |F|$ and setting $H_{\leq k_{F}}:=\sum_{k\leq k_{F}} H_{k}$ and
$H_{>k_{F}}:=\sum_{k> k_{F}} H_{k}$ we further have
$$|\int_{\TT}\chi_{F}\,H(\chi_{[\frac{1}{2},1]})|\geq |\int_{\TT}\chi_{F}\,H_{\leq k_{F}}(\chi_{[\frac{1}{2},1]})|-|\int_{\TT}\chi_{F}\,H_{>k_{F}}(\chi_{[\frac{1}{2},1]})|$$
$$\gtrsim \sum_{k\leq k_{F}} \int_{\TT}\chi_{F}\,|H_{\leq k_{F}}(\chi_{[\frac{1}{2},1]})|- |F|^{\frac{1}{2}}(\int_{F}|H_{>k_{F}}(\chi_{[\frac{1}{2},1]})|^2)^{\frac{1}{2}}$$
$$\gtrsim k_{F}\,|F|-c\,|F|\approx |F|\log\frac{4}{|F|}\:,$$
where in the last line $c>0$ is a suitable absolute constant derived from the operator norm $\|\cdot\|_{L^2\rightarrow L^2}$ of the maximal Hilbert transform.

\section{The time-frequency regularization of a set (at a fix frequency)}\label{resol}

In this section we introduce a new concept that stays at the foundation of our approach in the present paper and has the potential of becoming a useful tool in other related problems. For more about the usefulness of this concept and for a suggestive application of it, please see the Appendix.

Let $F\subseteq\TT$ be a measurable set. In what follows we present \textbf{the time-frequency regularization (TFR) of the set $F$} (relative to the $0$-frequency).

\subsection{The family of $0$-frequency tiles.} In our construction we need to consider special structured family of tiles at $0$-frequency arising from the properties of the level sets of the Hardy-Littlewood maximal operator associated with the characteristic function of $F$. The tiles involved in these structures will be selected from within a particular family of tiles called from now on \textit{$0$-frequency tiles}. Indeed, $R=[\o_{R},\,I_{R}]\in\P$ is called a \textit{$0$-frequency tile} iff $0\in\o_{R}$. Notice that the family of all $0$-frequency tiles - denoted throughout the paper with $\mathfrak{R}$ - consists from precisely those tiles which are ``sitting" on the real axis. In what follows, we will keep the letter $R$ for elements within the set $\mathfrak{R}$, while $\r$ will be reserved for subfamilies of $\mathfrak{R}$.

Observe from the very definition of a $0$-frequency tile that any such $R\in\mathfrak{R}$ is uniquely determined by $I_{R}$. We will from now on identify these two objects (in particular any dyadic time interval will give rise to a corresponding $0-$frequency tile). Thus given $I\subseteq\TT$ dyadic interval, we will often refer to the corresponding $0-$frequency tile as $R(I)$.

\subsection{Spacial decomposition according to the level sets of the Hardy-Littlewood Maximal function $M(\chi_{F})$.} Define $k_F:=[\log\frac{1}{|F|}]+1$ and for each $k\in\N$ with $k\leq k_F$  we let $\I_k$ be the collection of maximal dyadic intervals $I$ such that
\beq\label{lset}
\frac{|F\cap I|}{|I|}>2^{-k}\:,
\eeq
and set
\beq\label{lsetun}
\bar{\I}_k=\bigcup_{I\in\I_k}\,I\,.
\eeq
Notice that we have the following natural inclusion relation: for any $1<k<k_{F}$ and $I\in\I_{k-1}$ there exists $J\in\I_{k}$ such that $I\subsetneq J$ and hence $\bar{\I}_{k-1}\subsetneq \bar{\I}_{k}$.

\subsection{TFR of $F$: the algorithm.}\label{TiFeFol}
In what follows we present the algorithm of the \textbf{time-frequency $k-$regularization of $F$}. This new concept introduced here will prove fundamental in the tile-discretization of our lacunary Carleson operator performed in Section \ref{tffolF}.

Fix from now on $k\in\N$, with $1<k\leq k_{F}$; also fix $I\in\I_{k}$. The \textit{central iterative body} of the algorithm is given by:
\begin{itemize}
 \item \underline{\textbf{Input:}} we are given a non-empty collection\footnote{This will be specified at each step of our algorithm, see below.} of dyadic intervals $\A \subseteq\I_{k-1}(I)$ .

\item \underline{\textbf{Output:}}

- a threshold frequency $\a_{\A}$;

- three sets\footnote{It is possible for some of them to be empty. However if all the three sets are empty then the algorithm will stop at that step.} of dyadic intervals: $\B$, $\B^{U}$ and $\B^{L}$.

\item \underline{\textbf{Properties of the output:}}

1) \textit{Selection  of} $\a_{\A}$: From the set $\{|J|\}_{J\in\A}$ pick the smallest possible size $|J_0|$ such that the following saturation condition is satisfied:
\beq\label{[pcondkey}
\sum_{{J\in\A}\atop{|J|\leq|J_0|}} |J|\geq\frac{1}{2}\sum_{J\in\A} |J|\:.
\eeq
Set $\a_{\A}:=|J_0|^{-1}$.

2) \textit{Selection  of} $\B$:

\beq\label{defB}
\B:=\left\{\begin{array}{cl}
J\subset I\\
J\:\textrm{dyadic}
\end{array}\:\bigg|\:\begin{array}{cl}
|J|^{-1}=\a_{A}\\
\exists\:J'\in \A\:\:\textrm{s.t.}\:\:J'\subset J
\end{array}\right\}\:.
\eeq

3) \textit{Selection  of} $\B^{U}$:

\beq\label{defBU}
\B^U:=\{J\in\A\,|\,|J|^{-1}>\a_{A}\}\:.
\eeq

4) \textit{Selection  of} $\B^{L}$:

\beq\label{defBU}
\B^L:=\{J\in\A\,|\,|J|^{-1}<\a_{A}\}\:.
\eeq
\end{itemize}

With this done, we iterate the above central body of the algorithm as follows:

\begin{itemize}
\item \textbf{Step 1.} Initialize $\A:=\I_{k-1}(I)$ and denote the output with:

- $\a:=\a_{\A}$;

- $\c(I):=\B$;

- ${\S}^{U}(I):=\B^{U}$;

- ${\S}^{L}(I):=\B^{L}$.

\item \textbf{Step 2.}  Initialize $\A:={\S}^{U}(I)$ and denote the output:

- $\a^{U}:=\a_{\A}$;

- $\c^{U}(I):=\B$;

- ${\S}^{UU}(I):=\B^{U}$;

- ${\S}^{UL}(I):=\B^{L}$.

Then, we initialize  $\A:={\S}^{L}(I)$ and denote the output:

- $\a^{L}:=\a_{\A}$;

- $\c^{L}(I):=\B$;

- ${\S}^{LU}(I):=\B^{U}$;

- ${\S}^{LL}(I):=\B^{L}$.

\item \textbf{Step r, $r\geq 3$.} We continue inductively. From step $r-1$, for $s_j\in\{U,\,L\}$ with $j\in\{1,\ldots, r-1\}$, we have

- $2^{r-2}$ threshold frequencies of the form
$$\a^{s_1\ldots s_{r-2}}\,;$$

- $2^{r-2}$ sets of the form
$$\c^{s_1\ldots s_{r-2}}(I)\,;$$

- $2^{r-2}$ sets of the form
$$\S^{s_1\ldots s_{r-2} U}(I)\,;$$

- $2^{r-2}$ sets of the form
$$\S^{s_1\ldots s_{r-2} L}(I)\,.$$

Now, we are ready to apply the main body of our algorithm:

Initialize $\A:=\S^{s_1\ldots s_{r-1}}(I)$ and denote the output with:

- $\a^{s_1\ldots s_{r-1}}:=\a_{\A}$;

- $\c^{s_1\ldots s_{r-1}}(I):=\B$;

- ${\S}^{s_1\ldots s_{r-1}U}(I):=\B^{U}$;

- ${\S}^{s_1\ldots s_{r-1}L}(I):=\B^{L}$.

We will run this algorithm until first $r$ for which all the sets $\c^{s_1\ldots s_{r-1}}(I)$ are empty; this last fact is guaranteed by the harmless assumption that the original set $\I_{k-1}(I)$ can be represented by a finite union of dyadic intervals.
\end{itemize}

With these done, we define the following sets:
\begin{itemize}
\item for $r=0$, we simply define
$$\r_k^{0}[I]:=R(I)\,,$$
and set
\beq\label{k20}
\r_{k}^{F,0}:=\bigcup_{I\in\I_{k}}\r_k^{0}[I]\:.
\eeq

\item for $r=1$, we let
$$\r_k^{1}[I]:=\{R(J)\}_{J\in \c(I)}\,,$$
and set
\beq\label{k2}
\r_{k}^{F,1}:=\bigcup_{I\in\I_{k}}\r_k^{1}[I]\:.
\eeq

\item for general $r\geq 2$, we let
$$\r_k^{s_1\ldots s_{r-1}}[I]:=\bigcup_{J\in\c^{s_1\ldots s_{r-1}}(I)} R(J)\,,$$

$$\r_k^{r}[I]:=\bigcup_{j=1}^{r-1}\bigcup_{s_j\in\{U,L\}} \r_k^{s_1\ldots s_{r-1}}[I]\,,$$
and finally
\beq\label{kn}
\r_{k}^{F}[I]:=\bigcup_{r\geq 0}\r_k^r[I]\:.
\eeq
\item for $1<k\leq k_{F}$, we define the \textbf{time-frequency $k-$regularization of $F$} as the collection of $\mathfrak{R}-$tiles given by:

\beq\label{ktree}
\r_{k}^{F}:=\bigcup_{I\in\I_k}\r_{k}^{F}[I]\:.
\eeq

\item for $k=1$, we define the \textbf{time-frequency $1-$regularization of $F$} as
\beq\label{1tree}
\r_{1}^{F}:=\bigcup_{I\in\I_1} R(I)\:.
\eeq
\end{itemize}

Finally the (global) \textbf{time-frequency regularization (TFR) of $F$} is defined as the subcollection of tiles in
$\mathfrak{R}$ given by

\beq\label{totalfol}
\r^{F}:=\bigcup_{1\leq k\leq k_{F}}\r_{k}^{F}\:.
\eeq

\subsection{Key properties of TFR of a set}\label{propTFR}

In this subsection we analyze several of the important properties that characterize the above construction.

For notational simplicity we will drop from now on the super-index $F$ from the definition of the sets $\r$ above. Also,
from now on we will refer to $\r_{k}\setminus\r_{k}^{0}$ as simply $\r_{k}^{>0}$.

We now present several definitions:

As before, fix $k\in\N$ with $1\leq k\leq k_{F}$ and further fix $I\in\I_{k}$.

\begin{defn}\label{Rb}[\textsf{Base of $R$}]

Assume $R\in \r_{k}^{>0}$. We define the \textit{base of $R$}, denoted by
\beq\label{baser0}
\underline{R}\,,
\eeq
the unique tile $R'\in \r_{k}$ such that
\beq\label{baser}
R'\:\:\textrm{is minimal relative to the inclusion}\:\: I_R\subsetneq I_{R'}\:.
\eeq
\end{defn}

\begin{d0}\label{bsup}[\textsf{Children of} $\underline{R}$.]

For $\underline{R}\in\r_k$ we set
\beq\label{defchi}
\r_k^{chi}(\underline{R}):=\{R'\in \r_{k}\,|\,\underline{R'}=\underline{R}\}\:.
\eeq
\end{d0}

\begin{d0}\label{bsup}[\textsf{Base-support} $\underline{\r}_k^{s_1\ldots s_{r}}[I]$] \footnote{Throughout the paper we adopt the following convention: when setting the parameter $r\in\N$ we allow the value $r=0$, and, in this case, we simply set
\beq\label{conv0}
\underline{\r}_k^{s_1\ldots s_{r}}[I]:=\r_k^{0}[I]\:\:\textrm{and}\:\:\r_k^{s_1\ldots s_{r}}[I]:=\r_k^{1}[I]\:.
\eeq}

We let the \textit{base-support} of $\r_k^{s_1\ldots s_{r}}[I]$ be
\beq\label{bsupp}
\underline{\r}_k^{s_1\ldots s_{r}}[I]:=\{\underline{R}\,|\,\exists\:R'\in\r_k^{s_1\ldots s_{r}}[I]\:\:s.t.\:\:\underline{R}=\underline{R}'\}\,.
\eeq
\end{d0}

\begin{defn} Let $\a$ be a given frequency and $R\in \r$. We say
\begin{itemize}
\item $\a\in R$ iff  $\a\in \o_{R}$;

\item $\a\overline{\in} R$ iff $\a\in R$ and $\a\notin\underline{R}$.
\end{itemize}
\end{defn}

Notice now that the following hold:

\begin{itemize}
\item given $\a$ frequency and $R\in \r_k^{s_1\ldots s_{r}}[I]$ such that $\a\overline{\in} R$ then
$\a\overline{\in} R'$ for any $R'\in \r_k^{s_1\ldots s_{r}}[I]$; this is a direct consequence of the relation
\beq\label{eqdist}
\forall\:R,\,R'\in \r_k^{s_1\ldots s_{r}}[I]\:\:\textrm{we have}\:\:|I_{R}|=|I_{R'}|\:.
\eeq

\item if $\F$ is a given finite set of frequencies we let
$$\F[R]:=\{\a\in\F\,|\,\a\overline{\in} R\}\:.$$
Then, based on the previous item, we have that
$$\F[R]=\F[R']\:\:\:\:\forall\:R,\,R'\in \r_k^{s_1\ldots s_{r}}[I]\:.$$
\end{itemize}

\begin{defn} Let $\F$ be a finite set of frequencies. From the above items, we notice that it makes sense to define
\beq\label{freqr}
\F(\r_k^{s_1\ldots s_{r}}[I]):=\F[R]\,,
\eeq
for some $R\in \r_k^{s_1\ldots s_{r}}[I]$.
\end{defn}

With this, let us list several key properties of our construction:

\begin{itemize}

\item The final products of the algorithm above, \emph{i.e.} \eqref{ktree}, \eqref{1tree} and \eqref{totalfol}, are in fact \textit{generalized (i.e. non-convex) trees} of $0$-frequency tiles relative to the standard order relation $``\leq"$. \footnote{For the specific definitions of $``\leq"$ and of a (convex) tree see Definitions \ref{orderl} and \ref{tree} in Section \ref{foltree}.}

\item Let $R\in \r_k^{s_1\ldots s_{r}}[I]$ and $R'\in \r_k^{s'_1\ldots s'_{r'}}[I]$ with $r<r'$. Assume there exists
$\a$ frequency such that $\a\overline{\in} R$ and $\a\overline{\in} R'$. Then, we must have that
\begin{itemize}
\item $s'_1=s_1,\:\ldots,\,s'_{r}=s_{r}$;

\item $s'_{r+1}=L$.
\end{itemize}

\item Given $r<r'$ and assuming
\beq\label{underr}
\underline{\r}_k^{s'_1\ldots s'_{r'}}[I]\cap\underline{\r}_k^{s_1\ldots s_{r}}[I]\not=\emptyset\,,
\eeq
one must have
\begin{itemize}
\item $s'_1=s_1,\:\ldots,\,s'_{r}=s_{r}$;

\item $s'_{j}=L\:\:\:\forall\:j\in\{r+1,\ldots, r'\}$.
\end{itemize}

\item Based on the previous two items we deduce that if $\F$ is a given set of frequencies, then
\beq\label{freqcontr}
\F(\r_k^{s_1\ldots s_{r}}[I])\supseteq\bigcup_{{{s'_1=s_1,\ldots,\, s'_r=s_r}\atop{s'_{r+1}=L}}\atop{r'> r}}\bigcup_{R\in \r_k^{s'_1\ldots s'_{r'}}[I]}\F[R]\,.
\eeq

\item Similarly, if $R,\,R'$ as in the second item, then assuming that there exists $x\in I_{R}$ and $x\in I_{R'}$ we must have
\begin{itemize}
\item $s_1=s'_1,\:\ldots,\,s_{r}=s'_{r}$;

\item $s'_{r+1}=U$.
\end{itemize}

\item The following geometric decay holds: if $r'\geq r+2$ and $\s=(s_1,\ldots, s_{r})$, $\s'=(s'_1,\ldots, s'_{r'})$ given such that $s'_j=s_j$ for any $j\in\{1,\ldots, r\}$, then
\beq\label{CM}
\sum_{R'\in \r_k^{s'_1\ldots s'_{r'}}[I]}|I_{R'}|\leq\frac{1}{2}\,\sum_{R\in \r_k^{s_1\ldots s_{r}}[I]}|I_{R}|\:.
\eeq
\end{itemize}

\section{Tile Discretization I: mass, $F-$mass decompositions; tree foliations}\label{tilediscr1}

In this section we proceed with the decomposition of the family of tiles $\{P=[\o_P, I_P]\}_{P\in\P}$ according to\footnote{The explanations provided at each of the items in the listing below, are only to provide the general picture for each of the criteria. The more precise descriptions will follow immediately afterwards.}
\begin{itemize}

\item \textsf{the \textbf{$F-$mass} of a tile} - this depends on how much of the information carried by $F$ lies within the spacial support of a tile, \textit{i.e.} given $P=[\o_P, I_P]\in\P$ we classify our tiles depending on the quantity $\A_{F}(P)\approx\frac{|F\cap I_P|}{|I_P|}$.

\item \textsf{the \textbf{mass} ($F-$smooth adapted version) of a tile} - this is a refined version of the original concept of a mass of a tile introduced by C. Fefferman in \cite{f}. It combines the structural properties of the set $F$ with those of the assigned measurable function $N(x)$, with the latter component being further reflected into the properties of the sets $\{E(P)\}_{P\in\P}$. Roughly, the mass of a tile $P$ should be imagined as represented by $A(P)\approx\frac{|E(P)|}{|I_P|}$.

\item \textsf{the \textbf{tree $*$-foliation} of the time-frequency plane} - based on the standard partial order relation among tiles $``\leq"$  and on the concepts introduced in the items above, we will partition our collection of tiles into suitable sub-collections - called \emph{descending} trees - that have uniform mass and $F-$mass parameters.
 \end{itemize}

\subsection{$F-$mass decomposition}\label{Fmassdec}

Before presenting our first decomposition, we take advantage of the fact that we are in the lacunary case and thus that the frequencies $\{n_j\}_j\subseteq\N$ have the property that there exists a constant $\bar{C}>0$ depending only on our choice of the lacunary sequence such that\footnote{In fact relation \eqref{lacadv} turns out to be essentially equivalent with the lacunarity requirement; more precisely, we have that any lacunary sequence $\{n_j\}_j\subseteq\N$ obeys \eqref{lacadv} and, conversely, any (increasing) sequence of natural numbers obeying \eqref{lacadv} can be decomposed as union of no more than  $\bar{\bar{C}}>0$ (depending only on $\bar{C}$) lacunary (sub)sequences.}
\beq\label{lacadv}
\sum_{j=1}^{k} n_j< \bar{C}\,n_{k+1}\:\:\:\:\textrm{for any}\:k\in\N\,.
\eeq
Now relation \eqref{lacadv} implies in particular two key facts:
\begin{enumerate}
\item the $0$ frequency plays a special role in the time-frequency decomposition of our operator;

\item the family of all tiles $\P$ can be decomposed into three subsets
\beq\label{partP}
\P=\P(0)\cup \P_{cluster}\cup\P_{sep}\,,
\eeq
such that:
\begin{itemize}
\item $\P(0)$ does not contribute at all to the time-frequency decomposition of $T$ (all the $P\in\P(0)$ have the property that $E(P)=\emptyset$);
\item $\P_{cluster}$ consists of tiles that dilated by a fix amount (depending only on $\bar{C}$ in \eqref{lacadv}) intersect the real axis;
\item $\P_{sep}$ is given by a union of well separated trees, with each tile $P=[\o,I]\in \P_{sep}$ having the property that there exists $j\in\N$ such that $n_j\in \o$.
\end{itemize}
 \end{enumerate}

We will now make our decomposition in \eqref{partP} precise. For simplicity, assuming without loss of generality\footnote{The sequence $\{n_k\}_k$ lacunary implies $\lim \textrm{inf}_{k\rightarrow\infty}\frac{n_{k+1}}{n_k}=\a>1$.} that
\beq\label{clustertileee}
n_k= \a^k\:\:\textrm{with}\:\:\a\in\R_{+},\:\a>1\,,
\eeq
we define\footnote{Throughout this paper we will use the following standard notation: if $I$ is an (open) interval having the center $c$, then for any $b>0$ we set $b\,I:=(c-\frac{b\,|I|}{2},\,c+\frac{b\,|I|}{2})$. The constant $c(\a)$ in \eqref{clustertile} can be taken as $10(1+\lfloor\frac{1}{\a-1}\rfloor)$ where here $\lfloor x\rfloor$ stands for the enire part of $x\in\R$.}
\beq\label{clustertile}
\P_{cluster}:=\{P=[\o_P, I_P]\in\P\,|\,0\in c(\a)\,\o_{P}\}\:,
\eeq
\beq\label{septile}
\P_{sep}:=\{P=[\o_P, I_P]\in\P\setminus\P_{cluster}\,|\,\{n_j\}_j\cap \o_{P}\not=\emptyset\}\,.
\eeq
and
\beq\label{0tile}
\P(0):=\P\setminus(\P_{cluster}\cup\P_{sep})\,.
\eeq
In what follows we will focus on the only truly relevant set of tiles, that is $\P_{sep}$.

The tile discretization algorithm based on the time-frequency locations of the tiles in $\P_{sep}$ relative to the information carried by the set $F$ is related with the analysis in \cite{lvCarl1} and is now presented below:

\begin{itemize}
\item we set
\beq\label{p1}
\P^{1}:=\{P=[\o_P, I_P]\in\P_{sep}\,|\,\exists\:\:I\in\I_{1}\:\:\textrm{s.t.}\:\:5\tilde{I}_{P}\subset 200 I\:\&\:|I_P|\leq |I|\}\:.
\eeq
\item proceed by induction and assume we have constructed the set $\P^{k-1}$ for a suitable $k-1<k_{F}$;
let
$$\P^{>k-1}:=\P_{sep}\setminus \bigcup_{l\leq k-1}\P^{l}\,,$$
and define
\beq\label{pk}
\P^{k}:=\{P=[\o_P, I_P]\in\P^{>k-1}\,|\,\exists\:\:I\in\I_{k}\:\:\textrm{s.t.}\:\:5 \tilde{I}_{P}\subset 200 I\:\&\:|I_P|\leq |I|\}\:.
\eeq
\item notice now that we have the partition
\beq\label{puni}
\P_{sep}=\bigcup_{1\leq k\leq k_{F}}\P^{k}\:.
\eeq
\end{itemize}

Finally, we record that from \eqref{septile} and \eqref{puni}, one has
\beq\label{puniv}
\P=\P(0)\cup\P_{cluster}\cup\bigcup_{1\leq k\leq k_{F}}\P^{k}\:.
\eeq

We end this subsection with the following useful

\begin{d0}\label{Fmass}(\textsf{ $F-$mass of a tile})

Let $P\in\P^k\subset\P_{sep}$ where here $k\in\{1,\ldots, k_{F}\}$.

We then define the {\bf $F$-mass} of $P$ as
\beq\label{vF1}
\A_{F}(P):=2^{-k}\,.
\eeq
\end{d0}

\subsection{Partial ordering; Mass decomposition - an $F$ smooth version; Trees.}\label{massdec}

Our main focus in this subsection will be on performing a mass decomposition of our family of tiles that is adapted to the structural properties of the lacunary Carleson operator and ``behaves smoothly" relative to $F$. In doing this we transform Fefferman's mass into an adaptative concept relative to carefully chosen subfamilies of tiles. Though of different flavor, this modification shares some common features with the approach embraced by the author in defining the mass of a tile in \cite{lvPolynCar}.

We start by introducing the standard partial ordering among tiles (see \cite{f}):

\begin{defn}\label{orderl}[\textsf{Partial ordering of the tiles}]

Given any two $P=[\o,\, I],\,P'=[\o',\, I']$ in $\P$ we say that
\beq\label{ord}
P\leq P'\,,
\eeq
iff
$$I\subseteq I'\:\:\textrm{and}\:\:\o\supseteq \o'\;.$$
\end{defn}

We are now ready to introduce the main definition of this subsection:

\begin{d0}\label{mass}[\textsf{Mass of a tile ($F-$smooth version)}]

Let $P=[\o,I]\in\P_{sep}$. From \eqref{puni} we know that there exists a unique $k\in\{1,\ldots, k_{F}\}$ such that $P\in\P^{k}$.

We then define the {\bf mass} of $P$ as

\beq\label{v1} A(P):=\sup_{{P'=[\o',I']\in\:\P^k}\atop{P\leq P'}}\frac{|E(P')|}{|I'|}\,. \eeq
\end{d0}

With this, we define
\beq\label{Pn}
\P_n:=\{P\in\P_{sep}\,|\,A(P)\in(2^{-n-1}, 2^{-n}]\}\:,
\eeq
and deduce based on \eqref{puni} that
\beq\label{Ptot}
\P=\P(0)\cup\P_{cluster}\cup\bigcup_{n\in\N}\P_n\:.
\eeq

Now, once we have constructed $\P_n$ and $\P^k$, it will be useful to consider for later the families
\beq\label{Pnk}
\P_n^k:=\P_n\cap \P^k\,.
\eeq
With this, we notice the refinement of \eqref{Ptot} in the form
\beq\label{Ptotref}
\P=\P(0)\cup\P_{cluster}\cup\bigcup_{n\in\N}\bigcup_{k=1}^{k_F}\P_n^k\:.
\eeq

The next definition addresses the concept of a tree that should be regarded as encoding the time-frequency representation of a modulated, scaled (maximal) Hilbert transform:

\begin{d0}\label{tree} [\textsf{Tree}]

 We say that a set of tiles $\p\subset\P$ is a \textbf{tree} with \emph{top} $P_0$ if
the following conditions hold:
$\newline 1)\:\:\:\:\:\forall\:\:P\in\p\:\:\:\Rightarrow\:\:\:\:P\leq P_0$;
$\newline 2)\:\:\:\:\:$if $P_1,\:P_2\: \in\p$ and $P_1\leq P \leq P_2$ then $P\in\p\:.$
\end{d0}

We will need one more definition, as given by

\begin{d0}\label{tree} [\textsf{Row}]

We say that a set of tiles $\p\subset\P_{sep}$ is a \textbf{row} if there exists a collection of families of tiles $\{\p_j\}_{j\geq 1}$ such that
\begin{itemize}
\item each $\p_j$ is a tree with top $P_j=[\o_j,\,I_j]$;

\item for any $j\not=k$ we have that $I_j\cap I_k=\emptyset$;

\item the set $\p$ can be represented as
\beq\label{decrow}
\p=\bigcup_{j}\p_j\;.
\eeq
\end{itemize}
\end{d0}

We end this section by recording several useful facts:

\begin{o0}\label{mascont} 1) Remark that the simultaneous mass and $F-$mass partition of $\P_{sep}$ conserves the convexity property of the trees on which their $L^p$-boundedness is heavily relying. More precisely, for $1<k<k_F$, we have that
\beq\label{conv}
\textrm{if}\:\:P_1<P_2<P_3\:\:\textrm{such that}\:\:P_1,\,P_3\in\P_n^{k}\:\:\textrm{then}\:\:P_2\in\P_n^{k}\;.
\eeq

2) Let $\p\in\P_n^{k}$ be a tree. Recalling the notations from Observation \ref{tileadj}, we define $\I_{\p,min}^{*}$ to be the collection of minimum (relative to inclusion) time intervals $\{I_{P*}^r\}_{{P\in\p}\atop{r\in\{1,\ldots, 14\}}}$
and set $CZ(\I_{\p,min}^{*})[0,1]$ the Calderon-Zygmund decomposition of the interval $[0,1]$ with respect to $\I_{\p,min}^{*}$. \footnote{We recall here that given a collection $\A$ of dyadic intervals in $[0,1]$ we say that $CZ(\A)[0,1]$ is the Calderon-Zygmund decomposition of the interval $[0,1]$ with respect to $\A$ iff $CZ(\A)[0,1]$ can be written as $\A\cup \B$ with $\B$ a collection of dyadic intervals in $[0,1]$ such that: 1) $\A\cup \B$ forms a partition of $[0,1]$; 2) for any $I\in \A$ and $J\in\B$ one has $2I\nsupseteqq J$ and $2J\nsupseteqq I$; 3) $\B$ is a collection of maximal dyadic intervals obeying 1) and 2).}

Then, from \eqref{pk} we deduce the following key property:
\beq\label{cz}
\frac{|I\cap F|}{|I|}< 2^{-k+10}\:\:\:\:\:\:\:\:\forall\:I\in CZ(\I_{\p,min}^{*})[0,1]\:.
\eeq

In particular, deduce that if $P=[\o_P,I_P]\in\P_n^{k}$ then
\beq\label{suppPstar}
\frac{|\tilde{I}_{P}\cap F|}{|I_P|}< 2^{-k+10}\:.
\eeq

3) Let $\p\subset\P_n$ be \emph{any} given collection of tiles. Then, following the spirit of \cite{f} and - closer to our definition of mass - that of \cite{lvPolynCar}, one can use a relatively complex combinatorial procedure involving suitable tree selections partitioning $\p$ to show that\footnote{Further refinements of this decomposition were introduced in Section 5 of \cite{lvPolynCar}, though they are not necessary in our present situation. In particular, each family $\P_n$ can be reduced to a $BMO-$forest of $n^{th}$ generation - for the definition see Section 4 in \cite{lvPolynCar}.}
\beq\label{l2dec}
\|T^{\p}f\|_2\lesssim 2^{-\frac{n}{2}}\,\|f\|_2\,.
\eeq
Since these type of estimates - often referred as $L^2-$mass control - were treated extensively in \cite{lvPolynCar} in a more difficult context, we will not give supplementary details here and take \eqref{l2dec} as granted throughout the entire present paper.
\end{o0}

\subsection{Tree $*-$foliation of the time-frequency plane depending on the mass and $F$-mass parameters}\label{foltree}

In this section we intend to organize our family of tiles $\P_{sep}$ into maximal trees having uniform mass and $F-$mass parameters according to a ``descending foliation pattern" that focuses on the \textit{adjoint} support of these trees.

We start now the description of our tree selection:

Since our focus is on the family of tiles $\P_{sep}$ from our hypothesis and conventions we deduce that one can split the total family of separated tiles into maximal trees
\beq \label{ltdec}
\P_{sep}=\bigcup_{l\in\N}\p_{(l)}\:,
\eeq
with each tree $\p_{(l)}$ living at a given frequency $\o_{l}$ such that the sequence $\{\o_{l}\}_{l}$ is a strictly increasing lacunary sequence with $\o_{l+1}\geq \a\,\o_l$ where $\a>1$ - recall the assumption \eqref{clustertileee}.

Next, we decompose each tree $\p_{(l)}$ into maximal \emph{rows} with uniform $F-$mass parameter, i.e.
\beq \label{lndectree}
\p_{(l)}=\bigcup_{m=1}^{k_F}\p_{(l)}^{(m)}\:,
\eeq
where each row
\beq \label{lndectreeroq}
\p_{(l)}^{(m)}:=\P^{m}\cap \p_{(l)}\:,
\eeq
represents a union of disjoint maximal trees of uniform $F-$mass $\approx 2^{-m}$ all living at the same frequency $\o_l$.

With this done, we take each uniform $F-$mass row $\p_{(l)}^{(m)}$ and decompose it into maximal (disjoint) trees
\beq \label{maxlndectree}
\p_{(l)}^{(m)}=\bigcup_{a\geq 1} \p_{(l),a}^{(m)}\:,
\eeq
with each $\p_{(l),a}^{(m)}$ being a maximal tree of uniform $F-$mass $\approx 2^{-m}$ and living at the given frequency $\o_l$.

Now, for each\footnote{Finitely many such values of $a\in\N$.} $a\geq 1$, we isolate $\p_{(l),a}^{(m)}$ - assuming that this set of tiles is non-void - and decompose our tree in maximal sub-rows of uniform mass from the lowest to the largest value of the mass; to be more precise, we write:
\beq \label{Fmasdectree}
\p_{(l),a}^{(m)}=\bigcup_{n} \p_{(l),a}^{n,(m)}\:.
\eeq
Notice that under the assumption that $\P_{sep}$ is a finite family of tiles we have that given $m\in\N$ such that the LHS in \eqref{Fmasdectree} is nonempty, we have that the set $\{n\,|\,\p_{(l),a}^{n,(m)}\not=\emptyset\}$ is a finite convex subset of the set of natural numbers.

Next, we take each non-empty maximal row $\p_{(l),a}^{n,(m)}$ and decompose it into maximal (disjoint) trees
\beq \label{Fmasdectree1}
\p_{(l),a}^{n,(m)}=\bigcup_{b} \p_{(l),a,b}^{n,(m)}\:,
\eeq
so that each $\p_{(l),a,b}^{n,(m)}$ is maximal tree with uniform $F-$mass $\approx 2^{-m}$, uniform mass $\approx 2^{-n}$ and living at the given frequency $\o_l$.

Of key importance for us in what follows, will be the collection of tiles having a fix frequency and a given mass parameter, that is
\beq \label{Fmasdectree1n}
\p_{(l)}^{n}:=\bigcup_{a,b,m} \p_{(l),a,b}^{n,(m)}\:.
\eeq

Given now $l$ and $n$, the last stage of our decomposition regroups our constructed trees $\{\p_{(l),a,b}^{n,(m)}\}_{l,n,m,a,b}$ inside $\p_{(l)}^{n}$ following
a ``descending $*-$foliation pattern",  \emph{i.e.} at the heuristic level - we iterate over maximal families of maximal trees so that inside each family, the trees within it in have their corresponding adjoint support pairwise disjoint:

\begin{itemize}
\item let $P_{(l),a,b}^{n,(m)}$ be the top of the tree $\p_{(l),a,b}^{n, (m)}$. Let $\p_{(l)}^{n,tot}$ be the collection of all the trees $\{\p_{(l),a,b}^{n,(m)}\}_{m,a,b}$ and \textbf{$\T_l^n$} be the collection of their tops. We now define the first tree-layer of our decomposition as
\beq \label{lay1}
\p_{(l),*}^{n,[1]}\,,
\eeq
according to the following rule: a tree $\p_{(l),a,b}^{n, (m)}\in \p_{(l),*}^{n,[1]}$ iff its top $P:=P_{(l),a,b}^{n,(m)}\in \textbf{$\T_l^n$}$ has the property that there exists $x\in \TT$ obeying
\beq \label{layy1}
\eeq
\begin{itemize}
\item  $x\in \tilde{I}_P$;

\item $P$ is maximal with this property, that is, if there exists $P_1\in\T_l^n$ with $x\in \tilde{I}_{P_1}$ then we must have $|I_{P_1}|\leq |I_P|$.
\end{itemize}

Let \textbf{$\T_{l,*}^n[1]$} be the collection of all the tops of the trees inside $\p_{(l),*}^{n,[1]}$.

Notice that we do have the uniform intersection property
\beq \label{uiprop}
\sum_{P\in\textbf{$\T_{l,*}^n[1]$}} \chi_{\tilde{I}_P}(x)\leq 100\,.
\eeq
Indeed, to see this, it is enough to prove that there do not exist
\beq \label{cla0}
P_1, P_2, P_3\in \textbf{$\T_{l,*}^n[1]$}\:\:\textrm{with}\:\:|I_{P_1}|>|I_{P_2}|>|I_{P_3}|\,,
\eeq
such that
\beq \label{cla}
\tilde{I}_{P_1}\cap\tilde{I}_{P_2}\cap\tilde{I}_{P_3}\not=\emptyset\:\:\,.
\eeq
Assuming by contradiction that \eqref{cla} is violated and hence that there exists $x_0\in \tilde{I}_{P_1}\cap\tilde{I}_{P_2}\cap\tilde{I}_{P_3}$ with $P_1,\,P_2,\,P_3$ obeying \eqref{cla0} we notice that based on convention \eqref{supportinttt} we must have that either $\tilde{I}_{P_1}\supset\tilde{I}_{P_2}$ or $\tilde{I}_{P_1}\cup\tilde{I}_{P_2}\supset\tilde{I}_{P_3}$. However none of these situations are consistent with the requirement $P_1, P_2, P_3\in \textbf{$\T_{l,*}^n[1]$}$.

Assume now that $\p_{(l),a,b}^{n, (m)}\in \p_{(l),*}^{n,[1]}$. Then we define
\beq \label{compllevv1}
\bar{\p}_{(l),a,b}^{n, (m),[1]}:=\left\{\p_{(l),a',b'}^{n, (m)}\in \p_{(l)}^{n,tot}\,\large|\,
\begin{array}{cl}
\tilde{I}_{P_{(l),a,b}^{n,(m)}}\cap \tilde{I}_{P_{(l),a',b'}^{n,(m)}}=\textrm{nontrivial interval}\\
|I_{P_{(l),a',b'}^{n,(m)}}|\leq |I_{P_{(l),a,b}^{n,(m)}}|
\end{array}
\right\}
\eeq
We define now
\beq \label{compllev1}
\bar{\p}_{(l),*}^{n,[1]}:=\bigcup_{\p_{(l),a,b}^{n, (m)}\in \p_{(l),*}^{n,[1]}}\bar{\p}_{(l),a,b}^{n, (m),[1]}\,,
\eeq
and symmetrically
\beq \label{topcompllev1}
\textbf{$\bar{\T}_{l,*}^n[1]$}:=\{P_{(l),a,b}^{n, (m)}\,|\,\p_{(l),a,b}^{n, (m)}\in \bar{\p}_{(l),*}^{n,[1]}\}
\eeq

\item Next, we update our collections $\p_{(l)}^{n,tot}:= \p_{(l)}^{n,tot}[old] \setminus\bar{\p}_{(l),*}^{n,[1]}$ and
$\textbf{$\T_l^n$}:=\textbf{$\T_l^n$}[old]\setminus \textbf{$\bar{\T}_{l,*}^n[1]$}$.

We now have two possible situations: either $\p_{(l)}^{n,tot}=\textbf{$\T_l^n$}=\emptyset$ case in which our algorithm stops or $\p_{(l)}^{n,tot},\,\textbf{$\T_l^n$}\not=\emptyset$ case in which we repeat our first step and construct the corresponding sets
\beq \label{topcompllevvv1}
\p_{(l),*}^{n,[2]},\:\textbf{$\T_{l,*}^n[2]$},\:\bar{\p}_{(l),a,b}^{n, (m),[2]},\: \bar{\p}_{(l),*}^{n,[2]}\:\:\textrm{and}\:\:\textbf{$\bar{\T}_{l,*}^n[2]$}\:.
\eeq

\item We continue inductively this process so that - assuming that our algorithm didn't stop earlier - at step $p$ we construct
\beq \label{topcompllevv1}
\p_{(l),*}^{n,[p]},\:\textbf{$\T_{l,*}^n[p]$},\:\bar{\p}_{(l),a,b}^{n, (m),[p]},\: \bar{\p}_{(l),*}^{n,[p]}\:\:\textrm{and}\:\:\textbf{$\bar{\T}_{l,*}^n[p]$}\:.
\eeq

\item Notice that from our construction, this algorithm ends after $p_0\leq k_F+1$ steps. Indeed, this is an immediate consequence of the following observation: given $p_1\leq p_0$, $x\in\TT$ and any $\p_{(l),a_1',b_1'}^{n, (m_1)}\in \bar{\p}_{(l),a_1,b_1}^{n, (m_1),[p_1]}$ (non-void) such that $x\in \tilde{I}_{P_{(l),a_1,b_1}^{n, (m_1)}}$ then for any $1\leq p<p_1$ we have that there exists $\p_{(l),a',b'}^{n, (m)}\in \bar{\p}_{(l),a,b}^{n, (m),[p]}$ such that $x\in \tilde{I}_{P_{(l),a,b}^{n, (m)}}$ and $m_1<m$.
\end{itemize}

We end this section with a useful notation: in order to identify the layer to which a given tree $\p_{(l),a,b}^{n,(m)}\in \p_{(l)}^{n,tot}$ belongs we write
\beq \label{mnoptree}
\p_{(l),a,b}^{n,(m),[p]}:=\p_{(l),a,b}^{n,(m)}\:\:\textrm{iff}\:\:\p_{(l),a,b}^{n,(m)}\in\p_{(l),*}^{n,[p]}\,.
\eeq

Notice that this is a well defined notion and that for each $\p_{(l),a,b}^{n,(m)}$ there exists a unique $p$ for which \eqref{mnoptree} holds.

\section{Tile Discretization II: the set resolution of the time-frequency plane at a fix frequency}\label{tffolF}

In this section we will perform a final decomposition of our tiles depending on the deep geometric and additive combinatoric properties of the set $F$ encapsulated in the time-frequency regularization of the set $F$ developed in Section \ref{TiFeFol}. The procedure described here is a major improvement over the initial approach in \cite{lvKony1} that was designed to treat the $L^{1,\infty}$ behavior of the lacunary Carleson operator.\footnote{The origin of the tile decomposition here traces back to the so called \textbf{$(f,\lambda)-$lacunary decomposition} introduces by the author in \cite{lvKony1}.
If properly modified/extended this tile-decomposition has the potential of becoming a very useful tool in approaching other time-frequency questions asking about end-points behavior of various modulation-invariant operators. We plan to make this explicit in a future work.}

\begin{itemize}

\item Fix $1<k<k_{F}$ and $I\in\I_{k}$. Let $R\in \r_{k}^{>0}$ such that $R<R(I)$.

\medskip
Then, we define the $(R,2)-$family of tiles as
\beq\label{lkfam}
 \p_{k}^2[R]:=\left\{\begin{array}{cl}P=[\o_{P}, I_{P}]\\P\in\P^{>k}\end{array}\:\big|\:\begin{array}{cl}\tilde{I}_{P}\cap I\not=\emptyset\\
  \o_{P}\cap\o_{\underline{R}}=\emptyset\:\:\&\:\:\o_{P}\cap\o_{R}\not=\emptyset\end{array}\right\}\:.
\eeq

\begin{o0}\label{rchifoliation}
Let $\underline{R}\in\r_k$ with $I_{\underline{R}}\subseteq I$ and $I\in\I_{k}$. Set
\beq \label{rchisb}
\r_{k,chi}^{s_1\ldots s_{r}}[\underline{R}]:=\r_k^{chi}(\underline{R})\cap \r_k^{s_1\ldots s_{r}}[I]\:.
\eeq
Notice that if $R,\,R'\in \r_{k,chi}^{s_1\ldots s_{r}}[\underline{R}]$
then $|I_R|=|I_{R'}|$.

Deduce from this that
\beq \label{keypropfol}
\p_{k}^{2}[R]=\p_{k}^{2}[R']\,.
\eeq
\end{o0}

\item Let $1<k<k_{F}$ and $I\in\I_{k}$. Assume $R\in\r_{k}^{0}$ with $R=R(I)$.

\medskip
We then define the $(R,1)-$family of tiles as
\beq\label{l1fam}
\eeq
$$\p_{k}^1[R]:=\left\{P=[\o_{P}, I_{P}]\in\P^{k+1}\:\big|\:\begin{array}{cl}I_{R}\cap \tilde{I}_{P}\not=\emptyset\\
  \o_{P}\cap\o_{R}\not=\emptyset\end{array}\right\}\:.$$

\item for the limiting case $k=k_{F}$ and any $R\in\r_{k_{F}}^{F}$ we set $$\p_{k_F}^1[R]\equiv\p_{k_F}^2[R]=\emptyset\,.$$

\item we will also introduce two extra sets of tiles for the limiting case $k=1$, corresponding to the rectangles $R\in\{R(I)\}_{I\in\I_{1}}$:
\beq\label{l0fam2}
 \p_{1}^2[R]:=\left\{P=[\o_{P}, I_{P}]\in\P_{sep}\:\big|\:\tilde{I}_{P}\cap I_{R}\not=\emptyset\:\:\&\:\: \o_{P}\cap\o_{R}=\emptyset\right\}\:.
\eeq
and
\beq\label{l0fam1}
 \p_{1}^1[R]:=\left\{P=[\o_{P}, I_{P}]\in\P^{1}\cup\P^{2}\:\big|\:\tilde{I}_{P}\cap I_{R}\not=\emptyset\:\:\&\:\: \o_{P}\cap\o_{R}\not=\emptyset\,\right\}\:.
\eeq

\item finally, we let $\P[F,0]$ be the collection of tiles
\beq\label{rest}
P=[\o_P,\:I_P]\in\P_{sep}\:\;\textrm{such that}\:\:|\tilde{I}_{P}\cap F|=0\:.
\eeq
\end{itemize}

Define now the following
\begin{itemize}
\item for $1<k<k_{F}$ we let
\beq\label{rk}
\P[\r_k]:=\bigcup_{R\in\r_{k}^{0}}\p_{k}^1[R]\,\cup\,\bigcup_{R\in\r_{k}^{>0}}\p_{k}^2[R]\:.
\eeq
\item for $k=1$ we let
\beq\label{r1}
\P[\r_1]:=\bigcup_{R\in \r_1}\p_{1}^1[R]\cup\p_{1}^2[R]\:.
\eeq
\item for $k=k_{F}$ we simply set
\beq\label{r1}
\P[\r_{k_{F}}]:=\emptyset\:.
\eeq
\end{itemize}

Notice that with these done we have that
\beq\label{tot}
 \P=\P(0)\cup\P_{cluster}\cup\P[F,0]\cup\bigcup_{1\leq k<k_{F}}\P[\r_k]\:.
\eeq

As in \cite{lvKony1}, we notice that \eqref{tot} \textit{does not} express $\P$ as a disjoint union (partition) of sets.

However, as we will see later, when transferred into operator language with a further spacial localization, our tile decomposition behaves as good as a partition.

We end this section with an adaptation of our construction in Section \ref{foltree} to a generic $(R,2)-$family $\p_{k}^2[R]$.

For this, we first fix $1\leq k<k_{F}$ and assume throughout what follows that $l$ is fixed and $i\in\{1,2\}$. Setting now
\beq\label{kl}
 \p_{k,(l)}^i[R]:= \p_{k}^i[R]\cap \p_{(l)}\:,
\eeq
we remark that
\beq\label{ukl}
 \p_{k}^i[R]:=  \bigcup_{l}\p_{k,(l)}^i[R]\:.
\eeq
With this, we define the following families of tiles:
\begin{itemize}
\item using \eqref{lndectreeroq}, we let
\beq\label{klm}
 \p_{k,(l)}^{i,(m)}[R]:=  \p_{k,(l)}^i[R]\cap \p_{(l)}^{(m)}\:;
\eeq
\item using \eqref{Fmasdectree1n}, we let
\beq\label{klmtreen}
\p_{k,(l)}^{i,n}[R]:=  \p_{k,(l)}^i[R]\cap \p_{(l)}^{n}\:.
\eeq
\end{itemize}

Now for the specific case $i=2$ we will need a more refined analysis of our families of tiles. For this reason, we need to define the following
\begin{itemize}
\item using \eqref{Fmasdectree1}, we let
\beq\label{klmtree}
\p_{k,(l),a,b}^{2,n,(m)}[R]:=  \p_{k,(l)}^2[R]\cap \p_{k,(l),a,b}^{2,n,(m)}\:;
\eeq
\item using \eqref{mnoptree} and then the analogue of \eqref{compllev1}, we let
\beq\label{klmtreenpt}
\p_{k,(l),a,b}^{2,n,(m),[p]}[R]:=  \p_{k,(l)}^2[R]\cap \p_{(l),a,b}^{n,(m),[p]}\:.
\eeq
and
\beq\label{kklmtreenpt}
\bar{\p}_{k,(l),a,b}^{2,n,(m),[p]}[R]:=  \p_{k,(l)}^2[R]\cap \bar{\p}_{(l),a,b}^{n,(m),[p]}\:.
\eeq
\item finally, using \eqref{topcompllevv1}, we let
\beq\label{klmtreenptst1}
\p_{k,(l)}^{2,n,[p]}[R]:=  \p_{k,(l)}^2[R]\cap \p_{(l),*}^{n,[p]}\:\:\textrm{and}\:\:\p_{k}^{2,n,[p]}[R]:=  \bigcup_{l}\p_{k,(l)}^{2,n,[p]}[R]\:,
\eeq
and
\beq\label{klmtreenptst2}
\bar{\p}_{k,(l)}^{2,n,[p]}[R]:=  \p_{k,(l)}^2[R]\cap \bar{\p}_{(l),*}^{n,[p]}\:\:\textrm{and}\:\:\bar{\p}_{k}^{2,n,[p]}[R]:=  \bigcup_{l}\bar{\p}_{k,(l)}^{2,n,[p]}[R]\:.
\eeq
\end{itemize}

With this we set $\p_{k}^{i,n}[R]:=\p_{k}^{i}[R]\cap\P_n$ and notice that
\beq\label{klmtreenn0}
\p_{k}^{i,n}[R]= \bigcup_{l} \p_{k,(l)}^{i,n}[R]\:.
\eeq
In the case $i=2$, we further have
\beq\label{klmtreenn}
\p_{k}^{2,n}[R]= \bigcup_{l,p}\bar{\p}_{k,(l)}^{2,n,[p]}[R]= \bigcup_{l,m,p}\bar{\p}_{k,(l),a,b}^{2,n,(m),[p]}[R]\:.
\eeq

Throughout the paper, we only retain those trees  $\p_{k,(l),a,b}^{2,n,(m),[p]}[R]$ and families $\bar{\p}_{k,(l)}^{2,n,[p]}[R]$, $\bar{\p}_{k,(l),a,b}^{2,n,(m),[p]}[R]$  that are \emph{non-empty}.

Finally, we set\footnote{For $i=1$ the summation in \eqref{tottree} ranges over $R\in\r_{k}^{0}$ while for $i=2$ the summation is over $R\in\r_{k}^{>0}$.}
\beq\label{tottree}
\p_{k,(l)}^{i,n}:=\bigcup_{R\in\r_{k}^{(\cdot)}} \p_{k,(l)}^{i,n}[R]\,,
\eeq
and notice that all the tiles in $\p_{k,(l)}^{1,n}$ have (roughly) uniform $F$-mass $\approx 2^{-k}$.

\section{Preparatives; subdiving the Main Theorem B, part i)}\label{prep}

\subsection{Operator Discretization - The second stage}

With these facts, making use of one more observation:
$$\forall\:I\:\textrm{interval}\:I\subset \bar{\I}_{k+1}\:\textrm{and}\:I \cap\bar {\I}_{k}=\emptyset\:\:\Rightarrow\:\:|I\cap F|=0\:,$$
we conclude that for any $g\in L^{\infty}(\TT)$ with $\|g\|_{\infty}=1$ one has

\beq\label{partition}
\begin{array}{rl}
\|\sum_{P\in\P} T_P^{*}(g)\|_{L^1(F)}&\leq \|\sum_{P\in\P_{cluster}} T_P^{*}(g)\|_{L^1(F)}\\\\
&+\,\|\sum_{R\in\r_1}\chi_{I_{R}}{T^{\p_{1}^1[R]\cup\p_{1}^2[R]}}^{*}(g)\|_{L^1(F)}\\\\
&+\,\|\sum_{1<k\leq k_{F}}\sum_{R\in\r_{k}^{0}}\chi_{I_{R}}{T^{\p_k^1[R]}}^{*}(g)\|_{L^1(F)}\\\\
&+\,\|\sum_{1<k\leq k_{F}}\sum_{R\in\r_{k}^{>0}}\chi_{I_{R}}{T^{\p_k^2[R]}}^{*}(g)\|_{L^1(F)}\;.
\end{array}
\eeq

\subsection{Subdiving the Main Theorem B, part i): the key statements}

The proof of our Main Theorem B i) follows immediately from the following four theorems:

\begin{t1}\label{clust} Let $F\subseteq\TT$ be a measurable set and $g\in L^{\infty}(\TT)$. Then
\beq\label{clusts}
\|(T^{\P_{cluster}})^{*}(g)\|_{L^1(F)}\lesssim |F|\,\log(\frac{4}{|F|})\,\|g\|_{\infty}\:.
\eeq
\end{t1}

The proof of Theorem \ref{clust} is trivial if one uses the key observation that $\P_{cluster}$ is just a $c(\a)-$dilation of a tree where here $c(\a)$ is the same as the one appearing in \eqref{clustertile}. Indeed, heuristically, the proof reduces to the fact that the (maximal) Hilbert transform is bounded from $L\log L$ to $L^{1}$. We leave the details for the reader.

\begin{t1}\label{ortoF1} Let $F\subseteq\TT$ be a measurable set and $g\in L^{\infty}(\TT)$. Then
\beq\label{ortogs}
\left\|\sum_{R\in\r_1}\chi_{I_{R}}{T^{\p_{1}^1[R]\cup\p_{1}^2[R]}}^{*}(g)\right\|_{L^1(F)}\lesssim |F|\,\log(\frac{4}{|F|})\,\|g\|_{\infty}\:.
\eeq
\end{t1}

\begin{t1}\label{noosc} Let $F\subseteq\TT$ be a measurable set and $g\in L^{\infty}(\TT)$. Then
\beq\label{nooscs}
\left\|\sum_{1<k\leq k_{F}}\sum_{R\in\r_{k}^{0}}\chi_{I_{R}}{T^{\p_k^1[R]}}^{*}(g)\right\|_{L^1(F)}\lesssim |F|\,\log(\frac{4}{|F|})\,\|g\|_{\infty}\:.
\eeq
\end{t1}

\begin{t1}\label{ortog} Let $F\subseteq\TT$ be a measurable set and $g\in L^{\infty}(\TT)$. Then
\beq\label{ortogs}
\left\|\sum_{1<k\leq k_{F}}\sum_{R\in\r_{k}^{>0}}\chi_{I_{R}}{T^{\p_k^2[R]}}^{*}(g)\right\|_{L^1(F)}\lesssim |F|\,\log(\frac{4}{|F|})\,\|g\|_{\infty}\:.
\eeq
\end{t1}

\subsection{The Main Lemma}

In order to prove the above theorems we will rely on the following key result:
$\newline$

\noindent\textbf{Main Lemma}\label{Key} \textit{Let $n\in\N$ and $k\in\{1,\ldots, k_{F}\}$ be fixed parameters. Assume we are given a tree $\p$ with top $\bar{P}$. Further, let $I_{\p}:=\tilde{I}_{\bar{P}}$ and let $\I_{\p}$ be a collection of disjoint dyadic intervals satisfying the following properties:
\begin{itemize}
\item for any $P\in\p$ we have a uniform control on the mass parameter:
\beq\label{masss}
\frac{|E(P)|}{|I_{P}|}\leq 2^{-n}\,.
\eeq
\item let
\beq\label{ipmin}
\I_{\p,min}^{*}\,,
\eeq
be the collection of minimum time intervals inside the set $\{I_{P*}^r\}_{{P\in\p}\atop{r\in\{1,\ldots,14\}}}\,.$
If we denote with $CZ^{*}(I_{\p})$ the Calderon-Zygmund decomposition of the set $I_{\p}$ relative to the set $\I_{\p,min}^{*}$ the following hold:
\begin{itemize}
\item for any $I\in \I_{\p}$ and any $J\in CZ^{*}(I_{\p})$
\beq\label{cond}
\textrm{either}\:I\cap J=\emptyset\:\:\textrm{or}\:\:I\subseteq J\,;
\eeq
\item for each $I\in \I_{\p}$ there exists $k(I)\in \{1,\ldots, k_{F}\}$ such that for any $J\in CZ^{*}(I_{\p})$
one has
\beq\label{carl}
\sum_{{I\in\I_{\p}}\atop{I\subseteq J}} 2^{-k(I)}\,|I|\lesssim 2^{-k} |J|\,.
\eeq
\end{itemize}
\end{itemize}
Then, the following holds:
\beq\label{key}
\sum_{I\in\I_{\p}}\,2^{-k(I)}\,\int_{I} |T^{\p^*}(g)|^2\lesssim 2^{-k}\, 2^{-2n}\,|I_{\p}|\,\|g\|^2_{\infty}\,.
\eeq}

\begin{proof}

We start by noticing that without loss of generality by shifting the tree at zero frequency, we can assume $\o_{Top(\p)}=0$.

Define now the function
\beq\label{h}
h:=\sum_{I\in\I_{\p}}\,2^{-k(I)}\,\chi_{I}\,.
\eeq

Then, relation \eqref{key} is equivalent with showing that for any $f,\,g$ real valued functions with $f\in L^2(\TT)$ and  $g\in L^{\infty}(\TT)$, one has
\beq\label{equiv}
|\int f\,h^{\frac{1}{2}}\,T^{\p^*} g\,|\lesssim 2^{-\frac{k}{2}}\,2^{-n}\,|I_{\p}|^{\frac{1}{2}}\,\|g\|_{\infty}\,\|f\|_2\,.
\eeq

Given a generic $u\in L^1(\TT)$ and $I\subseteq\TT$ dyadic interval we set
\beq\label{proj}
\L_{I}(u):=\frac{\int_{I}u(s)\,ds}{|I|}\chi_I\;.
\eeq
Using now \eqref{proj} we define
\beq\label{lfunct}
\LL_{\p}(u):=\sum_{J\in CZ^{*}(I_{\p})}\L_{J}(u)\,.
\eeq
With these definitions we now have
\beq\label{V}
\eeq
$$V_{\p}:=\int f\,h^{\frac{1}{2}}\,T^{\p^*} g=\int T^{\p}(f\,h^{\frac{1}{2}})\,g$$
$$=\int T^{\p} (f\,h^{\frac{1}{2}}-\LL_{\p}(f\,h^{\frac{1}{2}}))\,g\,+\,\int T^{\p} (\LL_{\p}(f\,h^{\frac{1}{2}}))\,g
=:\,A\,+\,B\:.$$

Before passing to the proof we record the following

\begin{obs}\label{totmass}
If $\p$ is a collection of tiles, we set
\beq\label{tms}
E(\p):=\bigcup_{P\in\p}E(P)\,.
\eeq
If $\p\subset\P_n$ is a tree with top $\bar{P}$ and $I_{\p}:=\tilde{I}_{\bar{P}}$,  then one has
\beq\label{tmss}
|E(\p)|\lesssim 2^{-n}\,|I_{\p}|\,.
\eeq
\end{obs}

We are now ready to pass to our proof; we will start with the analysis of the last term:
\beq\label{carll1}
\eeq
$$|B|\lesssim \|g\|_{\infty}\,|E(\p)|^{\frac{1}{2}}\,\|T^{\p} (\LL_{\p}(f\,h^{\frac{1}{2}}))\|_{2}$$
$$\lesssim \|g\|_{\infty}\,|E(\p)|^{\frac{1}{2}}\,2^{-\frac{n}{2}}\,\|\LL_{\p}(f\,h^{\frac{1}{2}}))\|_2 $$
$$\lesssim \|g\|_{\infty}\,2^{-n}\,|I_{\p}|^{\frac{1}{2}}\,\left(\sum_{J\in CZ(\p^*)}\,\left(\frac{\int_J f\,h^{\frac{1}{2}}}{|J|}\right)^2\,|J|\right)^{\frac{1}{2}}$$
$$\lesssim_{CS} \|g\|_{\infty}\,2^{-n}\,|I_{\p}|^{\frac{1}{2}}\,\left(\sum_{J\in CZ(\p^*)}\,\frac{\int_J |f|^2\,
\int_{J} h}{|J|}\right)^{\frac{1}{2}}\;.$$
At this point we notice that \eqref{carl} implies
\beq\label{carl1}
\frac{\int_{J} h}{|J|}\lesssim 2^{-k}\:\:\:\:\:\textrm{for any}\:J\in CZ(\p^*)\:.
\eeq
Inserting now \eqref{carl1} in the last inequality  and using the mutual disjointness of the intervals in $ CZ(\p^*)$ we conclude
\beq\label{B1}
B\lesssim 2^{-n}\,2^{-\frac{k}{2}}\,|I_{\p}|^{\frac{1}{2}}\,\|g\|_{\infty}\,\|f\|_2\:.
\eeq

We pass now to the treatment of $A$.

Following now \cite{lvCarl1} and \cite{lvKony1}, we use the following key relation
\beq\label{meanzero}
\int_{J}[u-\LL_{\p}(u)]=0\:\:\:\:\:\:\:\forall\:J\in CZ(\p^*)\,.
\eeq
Using this, we deduce that
\beq\label{A1}
A= \int (f\,h^{\frac{1}{2}}-\LL_{\p}(f\,h^{\frac{1}{2}}))\,(T^{\p^*}g-\LL_{\p}(T^{\p^*}g))\:.
\eeq
Once in this point, we recall the following property proved in \cite{lvKony1}: for any $J\in CZ(\p^*)$ and $x\in J$ fixed we have
\beq \label{est}
\left|{T^{\p}}^{*}g(x)-\frac{1}{|J|}\int_{J}{T^{\p}}^{*}g(s)ds\right|\,\chi_{J}
\lesssim
\|g\|_{\infty}\,\frac{\chi_{J}}{|J|}\sum_{I_{P^*}\supseteq J}\frac{|J|^2}{|I_P|^2}\, |E(P)|\;.
\eeq
For completeness, we recall the proof of \eqref{est}: assuming $x\in J$ we have
$$\left|{T^{\p}}^{*}g(x)-\frac{1}{|J|}\int_{J}{T^{\p}}^{*}g(s)ds\right|=$$
$$\left|\frac{1}{|J|}\int_{J}\left\{\sum_{{P\in\p}\atop{|I_P|\geq
|J|}}\int_{\TT}\left[\v_k(x-y)-\v_k(s-y)\right]g(y)\chi_{E(P)}(y)dy\right\}ds\right|$$
$$\lesssim\frac{1}{|J|}\int_{J}\left\{\sum_{I_{P^*}\supseteq J} |I_P|^{-1}\,|J|\frac{\int_{E(P)}|g|}{|I_P|}\right\}ds\:.$$
Returning now to \eqref{A1} and making use of \eqref{est} we have
\beq \label{A2}
\eeq
$$|A|\lesssim \|g\|_{\infty}\,\sum_{J\in CZ(\p^*)} \frac{1}{|J|}\,
\int_{J}|f\,h^{\frac{1}{2}}-\LL_{\p}(f\,h^{\frac{1}{2}})|\,\sum_{I_{P^*}\supseteq J}\frac{|J|^2}{|I_P|^2}\,|E(P)|$$
$$\lesssim \|g\|_{\infty}\,2^{-n}\,\sum_{J\in CZ(\p^*)} |J|^{\frac{1}{2}}\,(\int_J |f|^2)^{\frac{1}{2}}\,
(\frac{\int_{J}h}{|J|})^{\frac{1}{2}}\,\sum_{I_{P^*}\supseteq J}\frac{|J|}{|I_{P}|}$$
$$\lesssim \|g\|_{\infty}\,2^{-n}\,2^{-\frac{k}{2}}\,\sum_{J\in CZ(\p^*)} |J|^{\frac{1}{2}}\,(\int_J |f|^2)^{\frac{1}{2}}
\lesssim_{CS} 2^{-n}\,2^{-\frac{k}{2}}\,|I_{\p}|^{\frac{1}{2}}\,\|g\|_{\infty}\,\|f\|_2\:.$$

This ends the proof of our Main Lemma.
\end{proof}

\begin{cor}\label{rowcont} Let $\bar{\p}_{(l),a,b}^{n, (m),[p]}$ as given by \eqref{topcompllevv1} via \eqref{compllevv1}.

Then, denoting with $I_{\p_{(l),a,b}^{n,(m),[p]}}$ the spacial interval of the top of the tree $\p_{(l),a,b}^{n, (m),[p]}$, we have that
\beq \label{bdrow}
\|{T^{\bar{\p}_{(l),a,b}^{n, (m),[p]}}}^{*}(g)\|_{L^2(F)}\lesssim 2^{-n}\,(2^{-m}\,|I_{\p_{(l),a,b}^{n,(m),[p]}}|)^{\frac{1}{2}}\,\|g\|_{\infty}\:.
\eeq
\end{cor}
\begin{proof}
First of all, let us notice that from the very definition of $\bar{\p}_{(l),a,b}^{n, (m),[p]}$ we have that this family of tiles can be written as
\beq \label{bdroww}
\bar{\p}_{(l),a,b}^{n, (m),[p]}=\bigcup_{\p_{l,c}^{n, (m)}\in\bar{\p}_{(l),a,b}^{n, (m),[p]}} \p_{l,c}^{n, (m)}\:,
\eeq
with $\{\p_{l,c}^{n, (m)}\}$ maximal disjoint tress of uniform $F-$mass $\approx 2^{-m}$ and mass $\approx 2^{-n}$ and such that the corresponding extending tops obey the relation
\beq \label{bdroww1}
\sum_{\p_{l,c}^{n, (m)}\in\bar{\p}_{(l),a,b}^{n, (m),[p]}} |\tilde{I}_{\p_{l,c}^{n, (m)}}|\lesssim |I_{\p_{(l),a,b}^{n,(m),[p]}}|\:.
\eeq

As in the proof of the Main Lemma, we will use duality theory. Thus, taking $f\in L^2(\TT)$ with $\|f\|_2=1$, we deduce that
\eqref{bdrow} is equivalent with
\beq \label{bdrow1}
\eeq
$$\left|\,\int \left(\sum_{\p_{l,c}^{n, (m)}\in\bar{\p}_{(l),a,b}^{n, (m),[p]}} T^{\p_{l,c}^{n, (m)}}(f\,\chi_{F})\right)\,g\,\right|$$
$$\lesssim 2^{-n}\,\left(2^{-m}\,|\tilde{I}_{\p_{(l),a,b}^{n,(m),[p]}}|\right)^{\frac{1}{2}}\,\|f\|_2\,\|g\|_{\infty}\:.$$
Let
$$\I_{min}^{*}\,,$$
 be the collection of minimum time intervals inside the set $\{I_{P*}^r\}_{{P\in\bar{\p}_{(l),a,b}^{n, (m),[p]}}\atop{r\in\{1,\ldots,14\}}}\,,$
and let $\tilde{I}_{row}=\bigcup_{P\in \bar{\p}_{(l),a,b}^{n, (m),[p]}}\tilde{I}_P$. Define now
\beq \label{bdrow11}
CZ^{*}(\tilde{I}_{row})\,,
\eeq
be the Calderon-Zygmund decomposition of $\tilde{I}_{row}$ relative to the collection of intervals $\I_{min}^{*}$.

We now define as in the Main Lemma (see \eqref{lfunct}), the function
\beq\label{lfunctt}
\LL_{row}(u):=\sum_{J\in CZ^{*}(\tilde{I}_{row})}\L_{J}(u)\,.
\eeq
Following the same circle of ideas as before, we have the analogue of \eqref{V} given by

\beq\label{Vv}
\eeq
$$V_{row}:=\int T^{\bar{\p}_{(l),a,b}^{n, (m),[p]}}(f\,\chi_{F})\,g$$
$$=\int T^{\bar{\p}_{(l),a,b}^{n, (m),[p]}} (f\,\chi_{F}-\LL_{row}(f\,\chi_F))\,g\,+\,\int T^{\bar{\p}_{(l),a,b}^{n, (m),[p]}} (\LL_{row}(f\,\chi_{F}))\,g$$
$$=A\,+\,B\:.$$

Using now \eqref{bdroww} and \eqref{bdroww1} together with the mass single tree estimate that can be deduced from Observation \ref{mascont} point 3) we deduce that
\beq\label{Vvs}
\|T^{\bar{\p}_{(l),a,b}^{n, (m),[p]}}\|_{2\rightarrow 2}\lesssim 2^{-\frac{n}{2}}\,.
\eeq
Also, notice that one has a natural correspondent to the key estimate \eqref{carl1} in the following
\beq\label{carlgl1}
\|\LL_{row}(f\,\chi_{F})\|_{L^2(J)}\lesssim 2^{-\frac{m}{2}}\,\|f\|_{L^2(J)}\:\:\:\:\:\textrm{for any}\:J\in CZ(\p^*)\:.
\eeq

With this, one can now follow the same strategy as in the Main Lemma for the corresponding terms $A$ and $B$, see \eqref{A2} and  \eqref{carll1}, respectively.

\end{proof}

We now start the proof our Main Theorem B i) by approaching our theorems stated at the beginning of this section in their increasing order of difficulty. For simplicity, wlog we will assume throughout the entire paper that $\|g\|_{\infty}=1$.
$\newline$

\section{Proof of Theorem \ref{ortoF1}.}\label{pfth1}

This result follows easily by exploiting
\begin{itemize}
\item the almost orthogonality of the tiles within the set $\{\p_{1}^1[R]\}_{R\in\r_1}$ once that the integration is extended to the full torus spacial support $[0,1]$;

\item the almost orthogonality of the family of tiles within the set $\{\p_{1}^2[R]\}_{R\in\r_1}$ this time exploiting the well separated frequencies relative to the spacial support represented by the set $\I_1$;
\end{itemize}
Define now
\beq\label{ortogs1}
\textbf{I}_{1,F}^{1}:=\|\sum_{R\in\r_1}\chi_{I_{R}}{T^{\p_{1}^1[R]}}^{*}(g)\|_{L^1(F)}\,,
\eeq
and
\beq\label{ortogs2}
\textbf{I}_{1,F}^{2}:=\|\sum_{R\in\r_1}\chi_{I_{R}}{T^{\p_{1}^2[R]}}^{*}(g)\|_{L^1(F)}\,,
\eeq
and recall that our goal is to show that:
\beq\label{ortogs3}
\textbf{I}_{1,F}:=\textbf{I}_{1,F}^{1}+\textbf{I}_{1,F}^{2}\lesssim |F|\,\log(\frac{4}{|F|})\,\|g\|_{\infty}\:.
\eeq

Using now notation \eqref{klmtreenn} and Cauchy-Schwarz we have:

$$\textbf{I}_{1,F}=\textbf{I}_{1,F}^{1}+\textbf{I}_{1,F}^{2}$$
$$\leq \sum_{n\in\N}\|\sum_{R\in\r_1}\chi_{I_{R}}{T^{\p_1^{1,n}[R]}}^{*}(g)\|_{L^1(F)}+\sum_{n\in\N}\|\sum_{R\in\r_1}\chi_{I_{R}}{T^{\p_1^{2,n}[R]}}^{*}(g)\|_{L^1(F)}$$ $$\lesssim|F|^{\frac{1}{2}}\,\sum_{n\in\N}\|\sum_{R\in\r_1}\chi_{I_{R}}{T^{\p_1^{1,n}[R]}}^{*}(g)\|_{L^2}
+|F|^{\frac{1}{2}}\,\sum_{n\in\N}\|\sum_{R\in\r_1}\chi_{I_{R}}{T^{\p_1^{2,n}[R]}}^{*}(g)\|_{L^2}$$

Next, appealing to \eqref{klmtree}, \eqref{klmtreen} and the definition of $\p_1^{1,n}[R]$ for $R\in\r_1$, we apply Cauchy-Schwarz and the standard $L^2$-bound for uniform mass tree of order $\approx 2^{-n}$ following from \eqref{l2dec}, and deduce
$$\textbf{I}_{1,F}^{1,n}:=\|\sum_{R\in\r_1}\chi_{I_{R}}{T^{\p_1^{1,n}[R]}}^{*}(g)\|_{L^2}$$
$$\leq\|\sum_{{l,a,b}\atop{m\in\{1,2\}}} {T^{\p_{1,(l),a,b}^{1,n,(m)}}}^{*}(g)\|_{L^2(\bar{\I}_2)}
\lesssim \|(\sum_{{l,a,b}\atop{m\in\{1,2\}}}|{T^{\p_{1,(l),a,b}^{1,n,(m)}}}^{*}(g)|^2)^{\frac{1}{2}}\|_{L^2(\bar{\I}_2)}\lesssim 2^{-\frac{n}{2}}\,|\bar{\I}_2|^{\frac{1}{2}}\;.$$
Thus we deduce that
\beq \label{term1es}
\textbf{I}_{1,F}^{1}\lesssim |F|^{\frac{1}{2}}\sum_{n\in\N}2^{-\frac{n}{2}}\,|\bar{\I}_2|^{\frac{1}{2}}\lesssim |F|\:.
\eeq
In a similar fashion, using \eqref{klmtreenpt} - \eqref{klmtreenptst2}, for the second term we have:
$$\textbf{I}_{1,F}^{2,n}:=\|\sum_{R\in\r_1}\chi_{I_{R}}{T^{\p_1^{2,n}[R]}}^{*}(g)\|_{L^2}$$
$$\lesssim\sum_{p}\|\sum_{R\in\r_1}\chi_{I_{R}}{T^{\bar{\p}_1^{2,n,[p]}[R]}}^{*}(g)\|_{L^2}$$
$$\lesssim \sum_{p}\left(\sum_{R\in\r_1}\|\sum_{l}|{T^{\bar{\p}_{1,(l)}^{2,n,[p]}[R]}}^{*}(g)|^2\|_{L^1(I_R)}\right)^{\frac{1}{2}}$$
$$\lesssim \sum_{p}\|(\sum_{l,m,a,b}|{T^{\bar{\p}_{(l),a,b}^{n,(m),[p]}}}^{*}(g)|^2)^{\frac{1}{2}}\|_{L^2(\bar{\I}_1)}\;.$$

Applying now Main Lemma and Cauchy-Schwarz, we have
\beq \label{term1es1}
\eeq
$$\textbf{I}_{1,F}^{2}\lesssim |F|^{\frac{1}{2}}\,\sum_{n\in\N} \sum_{p} \left(\sum_{l,m,a,b}\, 2^{-2n}\,2^{-m}\,|\tilde{I}_{\p_{(l),a,b}^{n,(m),[p]}}|\right)^{\frac{1}{2}}$$
$$\lesssim |F|^{\frac{1}{2}}\,k_{F}^{\frac{1}{2}}\, \sum_{n\in\N} 2^{-\frac{n}{2}}\,\left(\sum_{l,a,b}\sum_{m\leq k_{F}}\, 2^{-n}\,2^{-m}\,|I_{\p_{(l),a,b}^{n,[m]}}|\right)^{\frac{1}{2}}$$
$$\lesssim |F|^{\frac{1}{2}}\,k_{F}^{\frac{1}{2}}\,\left(\sum_{n\in\N}2^{-\frac{n}{2}}\right)\,k_{F}^{\frac{1}{2}}\,|F|^{\frac{1}{2}}\lesssim |F|\,\log\frac{4}{|F|}\:.$$

Combining now \eqref{term1es} and \eqref{term1es1}, we conclude that \eqref{ortogs3} holds.

\section{Proof of Theorem \ref{noosc}.}\label{thL1}

Our approach here is similar in spirit with the one embraced for proving Proposition 3 in \cite{lvKony1}. For convenience we will provide all the required details within this section.

Recall that our intention is to prove that
\beq\label{nooscsrep}
\|\sum_{1<k\leq k_{F}}\sum_{R\in\r_{k}^{0}}\chi_{I_{R}}{T^{\p_k^1[R]}}^{*}(g)\|_{L^1(F)}\lesssim |F|\,\log(\frac{4}{|F|})\,\|g\|_{\infty}\:,
\eeq
for $F\subseteq\TT$ measurable set and $g\in L^{\infty}(\TT)$ with $\|g\|_{\infty}=1$.

\medskip
We start by making precise the following heuristic: \emph{each term of the form ${T^{\p_k^1[R]}}^{*}(g)$ is essentially a sum of constant coefficients times suitable imaginary exponentials acting on the set $F\cap I_{R}$}.

\medskip
For this, we need to introduce first some notations:

Fix $R\in\r_{k}^{0}$ and using \eqref{ukl} we decompose
$$\p_k^1[R]=\bigcup_{l}\p_{k,(l)}^1[R]$$
into maximal trees at distinct frequencies $\{\o_l\}_l$.

Recalling now
\eqref{proj}, we define
\beq \label{almostconst1}
{T_{c}^{\p_k^1[R]}}^{*}g(\cdot):=\sum_{l} e^{2\pi i\,\o_l\cdot}\: \L_{I_{R}}({T^{\p_{k,(l)}^1[R,0]}}^{*}g)(\cdot)\:,
\eeq
where here we denote with $\o_l$ the frequency corresponding to the tree $\p_{k,(l)}^1[R]$ and $\p_{k,(l)}^1[R,0]$ designates the shift of $\p_{k,(l)}^1[R]$ at the origin.

\begin{l1}\label{treecutp1}
Let $R\in\r_{k}^{0}$. Then the following holds
\beq \label{forestexpl21}
\|{T^{\p_k^1[R]}}^{*}(g)-{T_{c}^{\p_k^1[R]}}^{*}(g)\|_{L^{\infty} (I_{R})}\lesssim
\sum_{P\in\p_k^1[R]} \frac{|I_{R}|}{|I_P|}\frac{\int_{E(P)}|g|}{|I_P|}\;.
\eeq
 \end{l1}

\begin{proof}

Set for notational simplicity
 $T_l={T^{\p_{k,(l)}^1[R,0]}}^{*}(g)-\L_{I_{R}}({T^{\p_{k,(l)}^1[R,0]}}^{*}(g))$.

 For fixed $l$ and $x\in I_{R} $ we have
$$|T_l(x)|=\left|{T^{\p_{k,(l)}^1[R,0]}}^{*}(g)(x)-\frac{1}{|I_{R}|}\int_{I_{R}}{T^{\p_{k,(l)}^1[R,0]}}^{*}(g)(s)ds\right|=$$
$$\left|\frac{1}{|I_{R}|}\int_{I_{R}}\left\{\sum_{{P\in\p_{k,(l)}^1[R,0]}\atop{2^{-j}=|I_P|\geq
|I_{R}|}}\int_{\TT}\left[\v_j(x-y)-\v_j(s-y)\right]g(y)\chi_{E(P)}(y)dy\right\}ds\right|$$
$$\lesssim\sum_{{P\in\p_{k,(l)}^1[R,0]}\atop{I_{P^*}\supseteq I_{R}}} \frac{|I_{R}|}{|I_P|}\frac{\int_{E(P)}|g|}{|I_P|}\:.$$

Summing now in $l$ we deduce that \eqref{forestexpl21} holds.
\end{proof}

\begin{l1}\label{difcont1}
With the previous notations, the following holds
\beq \label{difconts1}
\eeq
$$\|\sum_{1<k\leq k_{F}}\sum_{R\in\r_{k}^{0}}\chi_{I_{R}}{T^{\p_k^1[R]}}^{*}(g)\|_{L^1(F)}$$
$$\lesssim \|\sum_{1<k\leq k_{F}}\sum_{R\in\r_{k}^{0}}\chi_{I_{R}}{T_c^{\p_k^1[R]}}^{*}(g)\|_{L^1(F)}\,+\,|F|\,\log(\frac{4}{|F|})\,\|g\|_{\infty}\:.$$
\end{l1}

\begin{proof}

Appealing to the statement of Lemma \ref{treecutp1} we deduce
\beq \label{difconts1}
\eeq
$$\|\sum_{1<k\leq k_{F}}\sum_{R\in\r_{k}^{0}}\chi_{I_{R}}{T^{\p_k^1[R]}}^{*}(g)\|_{L^1(F)}$$
$$\lesssim \|\sum_{1<k\leq k_{F}}\sum_{R\in\r_{k}^{0}}\chi_{I_{R}}{T_c^{\p_k^1[R]}}^{*}(g)\|_{L^1(F)}$$
$$+\,\sum_{k\leq k_{F}}\sum_{R\in\r_{k}^{0}}\sum_{P\in\p_k^1[R]}
 \frac{|I_{R}|}{|I_P|}\frac{\int_{E(P)}|g|}{|I_P|}\,|I_{R}\cap F|\:.$$

Now denoting with $U$ the last term above, we deduce that
$$U\lesssim \|g\|_{\infty}\,\sum_{k\leq k_{F}} 2^{-k} \sum_{R\in\r_{k}^{0}}\sum_{P\in\p_k^1[R]} \frac{|I_{R}|^2}{|I_P|^2}\,|E(P)|\,.$$
Thus, for proving our lemma it is enough to show that:
\beq \label{U}
U\lesssim \,|F|\,\log(\frac{4}{|F|})\,\|g\|_{\infty}\:.
\eeq

Fix $k\in\{1,\ldots, k_{F}\}$. Then \eqref{U} follows
immediately once we show that
\beq \label{cont0}
U_{k}:=\sum_{R\in\r_{k}^{0}}\sum_{P\in\p_k^1[R]}\frac{|I_{R}|^2}{|I_P|^2}\,|E(P)|
\lesssim |\tilde{\I}_{k}|\:.
\eeq
Set $\I(\p_k^1[R]):=\{I_{P}\,|\,P\in \p_k^1[R]\}$. Then, \eqref{cont0} is a consequence of
\beq \label{cont1}
\eeq
$$U_{k}\lesssim \sum_{R\in\r_{k}^{0}}\sum_{I_P\in \I(\p_k^1[R])}\frac{|I_{R}|^2}{|I_P|^2}\,|I_P|\lesssim \sum_{R\in \r_{k}^{0}}|I_{R}|\lesssim |\tilde{\I}_{k}|\:.$$
\end{proof}

Based on the above two lemmas, we can justify now the following \textit{main heuristic for families of the type $\p_k^1[R]$}:

\medskip
\beq \label{heur}
\textsf{On the interval $I_{R}$ our operator ${T^{\p_k^{1}[R]}}^{*}(g)$ is morally constant}.
\eeq

\medskip
In order to prove the above heuristic, we start as before by decomposing
\beq \label{dec11}
\p_k^1[R]=\bigcup_{l}\p_{k,(l)}^1[R]\,,
\eeq
into maximal trees living at distinct (dyadic) frequencies $\{\o_l\}$. Once at this point, recalling the definition of $\p_k^{1}[R]$ and \eqref{lacadv} we deduce that for all $l$ and $R$ as in decomposition \eqref{dec11} we have
\beq \label{smallosc}
\sum_{l}\o_l\leq \bar{C}\,|\o_{R}|\;.
\eeq
Next, taking any $y_R\in I_{R}$, we notice that the Taylor expansion
\beq \label{Tay}
e^{2\pi i\,\o_l\,(x-y_{R})}=\sum_{k\geq 0} (2\pi i)^k\,\frac{[\o_l\,(x-y_{R})]^k}{k!}\,,
\eeq
is absolutely and uniformly convergent for any $x,\,y_{R}\in I_{R}$.

Recall now \eqref{almostconst1} and notice that
\beq \label{exp}
{T^{\p_k^1[R]}}^{*}=\sum_{l} e^{2\pi i\,\o_l\cdot} {T^{\p_{k,(l)}^1[R,0]}}^{*}\:.
\eeq
From \eqref{Tay} and \eqref{exp} we now deduce that
\beq \label{const1}
\begin{array}{rl}
\|{T_c^{\p_k^{1}[R]}}^{*}(g)\|_{L^1(F)}&\lesssim |I_{R}\cap F|\,|\sum_{l} e^{2\pi i\,\o_l\,y_{R}}\, \L_{I_{R}}({T^{\p_{k,(l)}^1[R,0]}}^{*}(g))(y_{R})|\\
&+\,|I_{R}\cap F|\,\sup_{l}\frac{\int_{I_{R}}|{{T^{\p_{k,(l)}^1[R]}}^{*}(g)}|}{|I_{R}|}\;.
\end{array}
\eeq
From \eqref{const1}, for an appropriately chosen $y_{R}\in I_{R}$, we further deduce
\beq \label{const12}
\begin{array}{cl}
&\|\sum_{1<k\leq k_{F}}\sum_{R\in\r_{k}^{0}}\chi_{I_{R}}{T_c^{\p_k^1[R]}}^{*}(g)\|_{L^1(F)}\\\\
&\lesssim \sum_{1<k\leq k_{F}} 2^{-k}\,\|\sum_{R\in\r_{k}^{0}}{T_c^{\p_k^{1}[R]}}^{*}(g)\|_{L^1}\\\\
&+\sum_{1<k\leq k_{F}} 2^{-k}\,\sum_{R\in\r_{k}^{0}}\sup_{l}\int_{I_{R}}|{{T^{\p_{k,(l)}^1[R]}}^{*}(g)}|\:.
\end{array}
\eeq
Using now \eqref{const12} and Lemmas \ref{treecutp1} and \ref{difcont1} we have
$\newline$

\begin{l1}\label{firstaproxp1fam}
The following holds
\beq \label{1noosc}
\begin{array}{rl}
&\|\sum_{1<k\leq k_{F}}\sum_{R\in\r_{k}^{0}}\chi_{I_{R}}{T^{\p_k^1[R]}}^{*}(g)\|_{L^1(F)}\\\\
&\lesssim\sum_{1<k\leq k_{F}} 2^{-k}\,\|\sum_{R\in\r_{k}^{0}}{T_c^{\p_k^{1}[R]}}^{*}(g)\|_{L^1}\\\\
&+\,\sum_{1<k\leq k_{F}} 2^{-k}\,\sum_{R\in\r_{k}^{0}}\sup_{l}\int_{I_{R}}|{{T^{\p_{k,(l)}^1[R]}}^{*}(g)}|\,+\,|F|\,\log(\frac{4}{|F|})\,\|g\|_{\infty}\:.

\end{array}
\eeq
 \end{l1}
$\newline$
$\newline$

\noindent\textbf{Finishing the proof of Theorem \ref{noosc}.}
$\newline$

Recalling \eqref{ukl} and \eqref{klmtreen} we set

$$\textbf{I}[n]:=\|\sum_{1<k\leq k_{F}}\sum_{R\in\r_{k}^{0}}\chi_{I_{R}}{T^{\p_k^{1,n}[R]}}^{*}(g)\|_{L^1(F)}\,,$$

and notice that a similar relation with \eqref{1noosc} will hold for this corresponding $n-$level.

Applying first Cauchy-Schwarz and then the Main Lemma to each maximal row $\p_{k,(l)}^{n}$ resulted from  $\p_l\cap\P^{k+1}_n$, we have

$$ \sum_{n}\textbf{I}[n]$$
$$\lesssim \sum_n\sum_{1<k\leq k_{F}} 2^{-k}\,\|\sum_{R\in\r_{k}^{0}}{T_c^{\p_k^{1,n}[R]}}^{*}(g)\|_{L^1}$$
$$+\,\sum_{n}\sum_{1<k\leq k_{F}} 2^{-k}\,\sum_{R\in\r_{k}^{0}}\sup_{l}\int_{I_{R}}|{{T^{\p_{k,(l)}^{1,n}[R]}}^{*}(g)}|\,+\,|F|\,\log(\frac{4}{|F|})$$
$$\lesssim \sum_n\sum_{1<k\leq k_{F}} 2^{-k}\,\sum_{R\in\r_{k}^{0}}|I_R|^{\frac{1}{2}}\,\|\sum_{l}{T_c^{\p_{k,(l)}^{1,n}[R]}}^{*}(g)\|_{L^2(I_R)}$$
$$+\,\sum_n\sum_{1<k\leq k_{F}} 2^{-k}\,\sum_{R\in\r_{k}^{0}}|I_R|^{\frac{1}{2}}\,\sup_{l}(\int_{I_{R}}|{{T^{\p_{k,(l)}^{1,n}[R]}}^{*}(g)}|^2)^{\frac{1}{2}}\,+\,|F|\,\log(\frac{4}{|F|})$$
$$\lesssim \sum_n\sum_{1<k\leq k_{F}} 2^{-k}\,|\bar {\I}_{k+1}|^{\frac{1}{2}}\,\|(\sum_{l}|{{T^{\p_{k,(l)}^{1,n}}}^{*}(g)}|^2)^{\frac{1}{2}}\|_{L^1}\,+\,|F|\,\log(\frac{4}{|F|})$$
$$\lesssim \sum_n\sum_{1<k\leq k_{F}} 2^{-\frac{n}{2}}\,2^{-k}\,
|\bar {\I}_{k+1}|\,+\,|F|\,\log(\frac{4}{|F|})\lesssim |F|\,\log \frac{4}{|F|}\:,$$
where in the second to last line we used notation \eqref{tottree}.

With this we conclude that \eqref{nooscsrep} holds.
$\newline$

\section{Reduction of Theorem \ref{ortog} to Proposition \ref{mainpropred}.}\label{thL2}

Our goal here is to show that
\beq\label{ortogs1}
\|\sum_{1<k\leq k_{F}}\sum_{R\in\r_{k}^{>0}}\chi_{I_{R}}{T^{\p_k^2[R]}}^{*}(g)\|_{L^1(F)}\lesssim |F|\,\log(\frac{4}{|F|})\,\|g\|_{\infty}\:.
\eeq

This will be the most difficult and in the same time interesting part of our paper in which the newly introduced concept of a set-resolution of the time-frequency plane at $0$-frequency will play a key role. The remainder part of the paper is entirely dedicated for proving \eqref{ortogs1}.

\subsection{An intermediate step: reduction to Proposition \ref{mainprop}.}

We start by adapting \eqref{almostconst1} to our new context:
fix $R\in\r_{k}^{>0}$ and recall the meaning of $\underline{R}$ - see Definition \ref{Rb}. Applying
the decomposition
$$\p_k^2[R]=\bigcup_{l}\p_{k,(l)}^2[R]$$
into maximal rows at distinct frequencies, we now define
\beq \label{almostconst}
{T_{c}^{\p_k^2[R]}}^{*}g(\cdot):=\sum_{l} e^{2\pi i\,\o_l\cdot}\: \L_{I_{\underline{R}}}({T^{\p_{k,(l)}^{2}[R,0]}}^{*}g)(\cdot)\:,
\eeq
where, as before, $\p_{k,(l)}^{2}[R,0]$ designates the shift of $\p_{k,(l)}^{2}[R]$ at the origin and $\o_l$ stands for the frequency corresponding to the tree $\p_{k,(l)}^2[R]$.

With the obvious modifications, the lemma below is the adaptation of the corresponding Lemma \ref{treecutp1}:
\begin{l1}\label{treecutp}
Let $R\in\r_{k}^{>0}$. Then the following holds
\beq \label{forestexpl2}
\|{T^{\p_k^2[R]}}^{*}(g)-{T_{c}^{\p_k^2[R]}}^{*}(g)\|_{L^{\infty} (I_{R})}\lesssim
\sum_{P\in\p_k^2[R]} \frac{|I_{\underline{R}}|}{|I_P|}\frac{\int_{E(P)}|g|}{|I_P|}\;.
\eeq
 \end{l1}

\begin{proof}

Adapting the definition of $T_l$ from Lemma \ref{treecutp1} to our context, we set $$T_l:={T^{\p_{k,(l)}^{2}[R,0]}}^{*}(g)-\L_{I_{\underline{R}}}({T^{\p_{k,(l)}^{2}[R,0]}}^{*}(g))\,.$$

 With this, proceeding as in the analogue proof, for fixed $l$ and $x\in I_{R}$ we deduce
\beq \label{diff}
|T_l(x)|\lesssim\sum_{{P\in\p_{k,(l)}^{2}[R,0]}\atop{I_{P^*}\supseteq I_{\underline{R}}}} \frac{|I_{\underline{R}}|}{|I_P|}\frac{\int_{E(P)}|g|}{|I_P|}\:,
\eeq
which, after summing in $l$, gives us the desired estimate.
\end{proof}

The next result is the analogue of Lemma \ref{difcont1}:

\begin{l1}\label{difcont}
With the previous notations, the following holds
\beq \label{difconts}
\eeq
$$\|\sum_{1<k\leq k_{F}}\sum_{R\in\r_{k}^{>0}}\chi_{I_{R}}{T^{\p_k^2[R]}}^{*}(g)\|_{L^1(F)}$$
$$\lesssim \|\sum_{1<k\leq k_{F}}\sum_{R\in\r_{k}^{>0}}\chi_{I_{R}}{T_c^{\p_k^2[R]}}^{*}(g)\|_{L^1(F)}\,+\,|F|\,\log(\frac{4}{|F|})\,\|g\|_{\infty}\:.$$
\end{l1}

\begin{proof}
Again, this is just an adaptation of the proof of Lemma \ref{difcont1} via Lemma \ref{treecutp}. Indeed, the analogue of \eqref{difconts1} reads now
\beq \label{difconts12}
\eeq
$$\|\sum_{1<k\leq k_{F}}\sum_{R\in\r_{k}^{>0}}\chi_{I_{R}}{T^{\p_k^2[R]}}^{*}(g)\|_{L^1(F)}$$
$$\lesssim \|\sum_{1<k\leq k_{F}}\sum_{R\in\r_{k}^{>0}}\chi_{I_{R}}{T_c^{\p_k^2[R]}}^{*}(g)\|_{L^1(F)}$$
$$+\,\sum_{k\leq k_{F}}\sum_{R\in\r_{k}^{>0}}\sum_{P\in\p_k^2[R]}
 \frac{|I_{\underline{R}}|}{|I_P|}\frac{\int_{E(P)}|g|}{|I_P|}\,|I_{R}\cap F|\:.$$

As before, letting
\beq \label{defU}
U:=\sum_{k\leq k_{F}}\sum_{R\in\r_{k}^{>0}}\sum_{P\in\p_k^2[R]}
 \frac{|I_{\underline{R}}|}{|I_P|}\frac{\int_{E(P)}|g|}{|I_P|}\,|I_{R}\cap F|\,,
\eeq
we deduce that
$$U\lesssim \|g\|_{\infty}\,\sum_{k\leq k_{F}} 2^{-k} \sum_{R\in\r_{k}^{>0}}\sum_{P\in\p_k^2[R]} \frac{|I_{\underline{R}}|\,|I_{R}|}{|I_P|^2}\,|E(P)|\,.$$

Fixing $k\in\{1,\ldots, k_{F}\}$, it remains to show that
\beq \label{cont}
U_{k}:=\sum_{R\in\r_{k}^{>0}}\sum_{P\in\p_k^2[R]}\frac{|I_{\underline{R}}|\,|I_{R}|}{|I_P|^2}\,|E(P)|
\lesssim |\tilde{\I}_{k}|\:.
\eeq
Setting $\I(\p_k^2[R]):=\{I_{P}\,|\,P\in \p_k^2[R]\}$, \eqref{cont} follows from
\beq \label{cont1}
\eeq
$$U_{k}\lesssim \sum_{R\in\r_{k}^{>0}}\sum_{I_P\in \I(\p_k^2[R])}\frac{|I_{\underline{R}}|\,|I_{R}|}{|I_P|^2}\,|I_P|$$
$$\lesssim \sum_{I\in\I_{k}}\sum_{R\in \r_{k}^{>0}[I]}|I_{R}|\,\sum_{I_P\in \I(\p_k^2[R])}\frac{|I_{\underline{R}}|}{|I_P|}\lesssim
\sum_{I\in\I_{k}}\sum_{R\in \r_{k}^{>0}[I]}|I_{R}|\lesssim |\tilde{\I}_{k}|\:.$$
\end{proof}

From the construction of the $F$-resolution of the time-frequency plane at $0-$frequency and from the definition of a generic $\p_k^2[R]$ we have that

\begin{o0}\label{osck}
Let $P\in \p_k^2[R]$. Then, the phase corresponding to $P$, referred to as $e^{2\pi i\,\a_P\,\cdot}$ with $\a_P\in \o_P$, does oscillate over the length of the interval $I_{\underline{R}}$ but is morally constant over the interval $I_{R}$.
\end{o0}

Applying this observation and invoking a similar argument with the one in \eqref{const1}-\eqref{const12} we deduce based on Lemma \ref{difcont} that

\begin{l1}\label{takef}
With the previous notations, we have that
\beq \label{takef1}
\eeq
$$\|\sum_{1<k\leq k_{F}}\sum_{R\in\r_{k}^{>0}}\chi_{I_{R}}{T^{\p_k^2[R]}}^{*}(g)\|_{L^1(F)}$$
$$\lesssim \sum_{1<k\leq k_{F}}\,\sum_{R\in\r_{k}^{>0}}\frac{|I_R\cap F|}{|I_R|}
\int_{I_R}\left|{T_c^{\p_k^2[R]}}^{*}(g)\right|\,+\,|F|\,\log(\frac{4}{|F|})\,\|g\|_{\infty}$$
\end{l1}

Fix now $1<k\leq k_{F}$ and $R\in \r_{k}^{>0}$. Recall the definitions at the end of Section \ref{tffolF}.
Now, in order to analyze the RHS of inequality \eqref{takef1}, we linearize the main term and thus, for $h,\,g\in L^{\infty}$ with $\|h\|_{\infty}=\|g\|_{\infty}=1$, we define the following expression:
\beq\label{defE}
\E(h,g):=\sum_{n\geq 1}\sum_{p\leq k_{F}}\E_{n}^{p}(h,g)\,,
\eeq
with
\beq \label{defEnm}
\E_{n}^{p}(h,g):= \sum_{1<k\leq k_{F}}\,\sum_{R\in\r_{k}^{>0}} \frac{|I_R\cap F|}{|I_R|}\,
\int_{I_R}h\,{T_c^{\bar{\p}_{k}^{2,n,[p]}[R]}}^{*}(g)\:.
\eeq

For any $R\in \r_{k}[I]$ we also define
\beq \label{almostconstt}
{T_{cc}^{\p_k^2[R]}}^{*}g(\cdot):=\sum_{l} e^{2\pi i\,\o_l\cdot}\: \L_{I}({T^{\p_{k,(l)}^{2}[R,0]}}^{*}g)(\cdot)\:.
\eeq

Depending on the context, we may choose to replace in \eqref{defEnm} the expression ${T_c^{\bar{\p}_{k}^{2,n,[p]}[R]}}^{*}(g)$ by ${T_{cc}^{\bar{\p}_{k}^{2,n,[p]}[R]}}^{*}(g)$ or simply ${T^{\bar{\p}_{k}^{2,n,[p]}[R]}}^{*}(g)$. This is fine since, the difference between any two such representations of $\E_{n}^{p}(h,g)$, even after summing in $p$, represents an admissible error of the form $\,|F|\,\log(\frac{4}{|F|})$. This last claim can be verified by following similar reasoning with the ones used in the proof of Theorem \ref{noosc} (more specifically Lemma \ref{difcont1}).

With these, we remark that \eqref{defEnm} can be then written\footnote{In these relations the symbol $\approx$ should be understood as equality up to admissible error terms that are controlled from above by $\,|F|\,\log(\frac{4}{|F|})$.} in any of the following fundamental forms depending on our convenience:
\beq \label{keyEfund}
\eeq
$$\E_{n}^{p}(h,g)=\sum_{1<k\leq k_{F}}\,\sum_{\underline{R}\in\r_{k}}\sum_{R\in\r_k^{chi}(\underline{R})}
\frac{|I_R\cap F|}{|I_R|}\,\int_{I_R}h\,{T_{c}^{\bar{\p}_{k}^{2,n,[p]}[R]}}^{*}(g)$$
$$\approx\sum_{1<k\leq k_{F}}\,\frac{|I_R\cap F|}{|I_R|}\,\,\sum_{\underline{R}\in\r_{k}}\sum_{R\in\r_k^{chi}(\underline{R})}
\int_{I_R}h\,{T_{cc}^{\bar{\p}_{k}^{2,n,[p]}[R]}}^{*}(g)\:.$$
$$\approx\sum_{1<k\leq k_{F}}\,\frac{|I_R\cap F|}{|I_R|}\,\sum_{\underline{R}\in\r_{k}}\sum_{R\in\r_k^{chi}(\underline{R})}
\int_{I_R}h\,{T^{\bar{\p}_{k}^{2,n,[p]}[R]}}^{*}(g)\:.$$

Once at this point, from Lemmas \ref{treecutp} - \ref{takef} and relation \eqref{keyEfund}, one trivially deduces that Theorem \ref{ortog} is a direct consequence of

\begin{p1}\label{mainprop} The following holds:
\beq\label{estimE}
|\E(h,g)|\lesssim |F|\,\log(\frac{4}{|F|})\,\|h\|_{\infty}\,\|g\|_{\infty}\:.
\eeq
\end{p1}

The remaining part of the paper is dedicated to proving Proposition \ref{mainprop}.

\subsection{A first reduction: small and large $n-$mass terms.}\footnote{Here we use the label $n-$mass for a tile $P$ to designate the fact that $P\in\P_n$.}

In this subsection we will split our term $\E$ in \eqref{defE} into a small and large $n-$mass component and show that the large $n$ component is under control.

Define now $n_F\in N$ to be the smallest natural number $n\in N$ such that
\beq \label{not}
2^{n}\geq k_{F}\approx\log\frac{4}{|F|}\:.
\eeq
Thus, we have that $n_F\approx \log \log \frac{12}{|F|}$.

Recalling now \eqref{keyEfund}, we define
\beq\label{defEs}
\E_{<}:=\sum_{n\leq n_F}\sum_{p\leq k_{F}}\E_{n}^{p}\,,
\eeq
\beq\label{defEl}
\E_{>}:=\sum_{n> n_F}\sum_{p\leq k_{F}}\E_{n}^{p}\,,
\eeq
and notice from \eqref{defE} that
\beq\label{defEe}
\E=\E_{<}+\E_{>}\,.
\eeq

As announced earlier, in what follows we will show that the second term is under control, that is

\begin{p1}\label{masslessprop}
With the previous notations, one has
\beq\label{estimElm}
|\E_{>}(h,g)|\lesssim |F|\,\log(\frac{4}{|F|})\,\|h\|_{\infty}\,\|g\|_{\infty}\:.
\eeq
\end{p1}

The above result is based on the next two key lemmas:

\begin{l1}\label{controlbasis}
Fix $I\in\I_k$, $r\in\N$, and $s_j\in\{U,L\}$ and take $\r_k^{s_1\ldots s_{r}}[I]$.
 With the definitions and notations from the previous sections the following holds\footnote{Here, as before, we apply the usual convention $\bar{\p}_{k}^{2,n,[p]}[\r_k^{s_1\ldots s_{r},U}[\underline{R}]]:=\bar{\p}_{k}^{2,n,[p]}[R]$ for some $R\in \r_{k,chi}^{s_1\ldots s_{r},U}[\underline{R}]$.}:
\beq\label{claimm11}
\eeq
$$\sum_{\underline{R}\in\r_k^{s_1\ldots s_{r}}[I]}
\sum_{R\in\r_k^{chi}(\underline{R})} \int_{I_R}\left|{T_{cc}^{\bar{\p}_{k}^{2,n,[p]}[R]}}^{*}(g)\right|$$
$$\lesssim \left(\sum_{\underline{R}\in \r_k^{s_1\ldots s_{r}}[I]} |I_{\underline{R}}|\right)^{\frac{1}{2}}\,
\left(\sum_{\underline{R}\in \r_k^{s_1\ldots s_{r}}[I]}\int_{I_{\underline{R}}}\left|{T_{cc}^{\bar{\p}_{k}^{2,n,[p]}[\r_k^{s_1\ldots s_{r},U}[\underline{R}]]}}^{*}(g)\right|^2\right)^{\frac{1}{2}}\:.$$
\end{l1}
\begin{proof}

We start by noticing that the LHS in \eqref{claimm11} can be rewritten as
\beq\label{VIR}
V_{\r_k^{s_1\ldots s_{r}}[I]}:=\sum_{{{s'_1=s_1,\ldots,\, s'_r=s_r,s'_{r+1}=U}\atop{s'_{u}=L\,\textrm{if}\, r+1<u\leq r'}}\atop{r'>r}}\int_{\bigcup_{R\in \r_k^{s'_1\ldots s'_{r'}}[I]}I_R}\left|{T_{cc}^{\bar{\p}_{k}^{2,n,[p]}[R]}}^{*}(g)\right|\:.
\eeq
Let us define
\beq\label{spacIR}
I_{\r_k^{s_1\ldots s_{r}}[I]}:=\bigcup_{R\in \r_k^{s_1\ldots s_{r}}[I]} I_{R}\:.
\eeq
Also notice that since for any $R,\,R'\in \r_k^{s'_1\ldots s'_{r'}}[I]$ one has $$T_{cc}^{\bar{\p}_{k}^{2,n,[p]}[R]}(g)=T_{cc}^{\bar{\p}_{k}^{2,n,[p]}[R']}(g)\,,$$
and thus one further deduces that the following is a well defined object:
$$T_{cc}^{\bar{\p}_{k}^{2,n,[p]}[\r_k^{s'_1\ldots s'_{r'}}[I]]}(g)=T_{cc}^{\bar{\p}_{k}^{2,n,[p]}[R]}(g)\:\:\:\:\textrm{for some}\:R\in \r_k^{s'_1\ldots s'_{r'}}[I]\:.$$
Notice that based on \eqref{CM}, for any $r'>r$, we have that
\beq\label{spacIRdec}
|I_{\r_k^{s_1\ldots s_{r}\ldots s_{r'}}[I]}|\lesssim \frac{1}{2^{\frac{r'-r}{2}}}\,|I_{\r_k^{s_1\ldots s_{r}}[I]}|\:.
\eeq
Fix now  $s_1\ldots s_{r}$. Throughout the entire proof when talking about $s'_1\ldots s'_{r'}$ we assume $r'>r$ and $s'_u=s_u$ if $u\leq r$, $s'_{r+1}=U$ and $s'_u=L$ if $u>r+1$.

For $w\in\N$ we define
\beq\label{spacIRdecw}
\eeq
$$I_{\r_k^{s_1\ldots s_{r'}}[I]}^w:=$$
$$\left\{x\in I_{\r_k^{s_1\ldots s_{r}}[I]}\,|\,|{T_{cc}^{\bar{\p}_{k}^{2,n,[p]}[\r_k^{s'_1\ldots s'_{r'}}[I]]}}^{*}(g)(x)|\approx 2^{w}\,\frac{\|{T_{cc}^{\bar{\p}_{k}^{2,n,[p]}[\r_k^{s'_1\ldots s'_{r'}}[I]]}}^{*}(g)\|_{L^2(I_{\r_k^{s_1\ldots s_{r}}[I]})}}{|I_{\r_k^{s_1\ldots s_{r}}[I]}|^{\frac{1}{2}}}\right\}\:.$$
Deduce now from Chebyshev inequality that
\beq\label{ChebspacIRdecw}
|I_{\r_k^{s_1\ldots s_{r'}}[I]}^w|\lesssim 2^{-2w}\,|I_{\r_k^{s_1\ldots s_{r}}[I]}|\:.
\eeq

We now estimate \eqref{VIR} as follows

$$V_{\r_k^{s_1\ldots s_{r}}[I]}\approx\sum_{{{s'_1=s_1,\ldots,\, s'_r=s_r,s'_{r+1}=U}\atop{s'_{u}=L\,\textrm{if}\, r+1<u\leq r'}}\atop{r'>r}}\sum_{w\in\N}\int_{I_{\r_k^{s_1\ldots s_{r'}}[I]}^w}\left|{T_{cc}^{\bar{\p}_{k}^{2,n,[p]}[\r_k^{s_1\ldots s_{r}}[I]]}}^{*}(g)\right|$$
$$\approx \sum_{{{s'_1=s_1,\ldots,\, s'_r=s_r,s'_{r+1}=U}\atop{s'_{u}=L\,\textrm{if}\, r+1<u\leq r'}}\atop{r'>r}}\sum_{w\in\N}
2^{w}\,\frac{|I_{\r_k^{s_1\ldots s_{r'}}[I]}^w|}{|I_{\r_k^{s_1\ldots s_{r}}[I]}|}$$
$$\times\,|I_{\r_k^{s_1\ldots s_{r}}[I]}|^{\frac{1}{2}}\,\|{T_{cc}^{\bar{\p}_{k}^{2,n,[p]}[\r_k^{s'_1\ldots s'_{r'}}[I]]}}^{*}(g)\|_{L^2(I_{\r_k^{s_1\ldots s_{r}}[I]})}\:.$$
Using now the fact that
\beq\label{dom}
\|{T_{cc}^{\bar{\p}_{k}^{2,n,[p]}[\r_k^{s'_1\ldots s'_{r'}}[I]]}}^{*}(g)\|_{L^2(I_{\r_k^{s_1\ldots s_{r}}[I]})}\lesssim\|{T_{cc}^{\bar{\p}_{k}^{2,n,[p]}[\r_k^{s_1\ldots s_{r},U}[I]]}}^{*}(g)\|_{L^2(I_{\r_k^{s_1\ldots s_{r}}[I]})}\:,
\eeq
together with \eqref{ChebspacIRdecw} and \eqref{spacIRdec} we further have
$$V_{\r_k^{s_1\ldots s_{r}}[I]}\lesssim |I_{\r_k^{s_1\ldots s_{r}}[I]}|^{\frac{1}{2}}\,\|{T_{cc}^{\bar{\p}_{k}^{2,n,[p]}[\r_k^{s_1\ldots s_{r},U}[I]]}}^{*}(g)\|_{L^2(I_{\r_k^{s_1\ldots s_{r}}[I]})}\,$$
$$\times\sum_{{{s'_1=s_1,\ldots,\, s'_r=s_r,s'_{r+1}=U}\atop{s'_{u}=L\,\textrm{if}\, r+1<u\leq r'}}\atop{r'>r}}\sum_{w\in\N}
\left(\frac{|I_{\r_k^{s_1\ldots s_{r'}}[I]}^w|}{|I_{\r_k^{s_1\ldots s_{r}}[I]}|}\right)^{\frac{1}{2}}$$
$$\lesssim |I_{\r_k^{s_1\ldots s_{r}}[I]}|^{\frac{1}{2}}\,\|{T_{cc}^{\bar{\p}_{k}^{2,n,[p]}[\r_k^{s_1\ldots s_{r},U}[I]]}}^{*}(g)\|_{L^2(I_{\r_k^{s_1\ldots s_{r}}[I]})}\,$$
$$\times\sum_{{{s'_1=s_1,\ldots,\, s'_r=s_r,s'_{r+1}=U}\atop{s'_{u}=L\,\textrm{if}\, r+1<u\leq r'}}\atop{r'>r}}\sum_{w\in\N}
\min \left\{2^{\frac{r-r'}{2}},\,2^{-w}\right\}^{\frac{1}{2}}$$
$$\lesssim |I_{\r_k^{s_1\ldots s_{r}}[I]}|^{\frac{1}{2}}\,\|{T_{cc}^{\bar{\p}_{k}^{2,n,[p]}[\r_k^{s_1\ldots s_{r},U}[I]]}}^{*}(g)\|_{L^2(I_{\r_k^{s_1\ldots s_{r}}[I]})}\,,$$
which is precisely the desired conclusion in \eqref{claimm11}.
\end{proof}

\begin{l1}\label{controlfreqbasis}
Let $I\in\I_k$ and $\r_k^{s_1\ldots s_{r}}[I]$ as before. Then, the following holds:
\beq\label{claimm21}
\eeq
$$\sum_{{{s'_1=s_1,\ldots,\, s'_r=s_r}\atop{s'_{r+1}=L}}\atop{r'>r}} \left(\sum_{\underline{R}\in \r_k^{s'_1\ldots s'_{r'}}[I]}|I_{\underline{R}}|\right)^{\frac{1}{2}}\,
\left(\sum_{\underline{R}\in \r_k^{s'_1\ldots s'_{r'}}[I]}\int_{I_{\underline{R}}}\left|{T_{cc}^{\bar{\p}_{k}^{2,n,[p]}[\r_k^{s'_1\ldots s'_{r'},U}[\underline{R}]]}}^{*}(g)\right|^2\right)^{\frac{1}{2}}$$
$$\lesssim \left(\sum_{\underline{R}\in \r_k^{s_1\ldots s_{r}}[I]}|I_{\underline{R}}|\right)^{\frac{1}{2}}\,
\left(\sum_{\underline{R}\in \r_k^{s_1\ldots s_{r}}[I]}\int_{I_{\underline{R}}}\left|{T_{cc}^{\bar{\p}_{k}^{2,n,[p]}[\r_k^{s_1\ldots s_{r},U}[\underline{R}]]}}^{*}(g)\right|^2\right)^{\frac{1}{2}}\:.$$
\end{l1}
\begin{proof}
With the notations introduced in the previous proof, our task is to control the term
\beq\label{claimm211}
\eeq
$$V:=\sum_{{{s'_1=s_1,\ldots,\, s'_r=s_r}\atop{s'_{r+1}=L}}\atop{r'>r}}
|I_{\r_k^{s'_1\ldots s'_{r'}}[I]}|^{\frac{1}{2}}\,\|{T_{cc}^{\bar{\p}_{k}^{2,n,[p]}[\r_k^{s'_1\ldots s'_{r'},U}[I]]}}^{*}(g)\|_{L^2(I_{\r_k^{s'_1\ldots s'_{r'}}[I]})}$$
Throughout this proof, we fix the sequence  $s_1\ldots s_{r}$ and whenever considering $s'_1\ldots s'_{r'}$ we assume $r'>r$ and $s'_u=s_u$ if $u\leq r$ and $s'_{r+1}=L$.

With the above notations and conventions, we have now the following important observations:
\beq\label{obs3imp}
\eeq
\begin{itemize}
\item $|I_{\r_k^{s'_1\ldots s_{r'}}[I]}|\lesssim \frac{1}{2^{\frac{r'-r}{2}}}\,|I_{\r_k^{s_1\ldots s_{r}}[I]}|\:.$

\item if $s'_1\ldots s'_{r'}$ and $s''_1\ldots s''_{r''}$ are two \textit{distinct} sequences such that $r'=r''$ then
\begin{itemize}
\item $I_{\r_k^{s'_1\ldots s_{r'}}[I]}\cap I_{\r_k^{s''_1\ldots s''_{r''}}[I]}=\emptyset$;

\item $\bar{\p}_{k}^{2,n,[p]}[\r_k^{s'_1\ldots s'_{r'},U}[I]]\cap \bar{\p}_{k}^{2,n,[p]}[\r_k^{s''_1\ldots s''_{r''},U}[I]]=\emptyset\,.$
\end{itemize}
\item given $r_0\geq r+1$ we have
\beq\label{ks1}
\bigcup_{r'=r_0}\bar{\p}_{k}^{2,n,[p]}[\r_k^{s'_1\ldots s'_{r'},U}[I]]\subseteq \bar{\p}_{k}^{2,n,[p]}[\r_k^{s_1\ldots s_{r},U}[I]]\,.
\eeq
\item same, if $r_0\geq r+1$ then, we have
\beq\label{k1}
 \sum_{{{s'_1=s_1,\ldots,\, s'_r=s_r}\atop{s'_{r+1}=L}}\atop{r'=r_0}}
|I_{\r_k^{s'_1\ldots s'_{r'}}[I]}| \lesssim |I_{\r_k^{s_1\ldots s_{r}}[I]}|\:.
\eeq
\item finally, using the previous notations and the second item above, one deduces
\beq\label{k2}
\eeq
$$\sum_{{{s'_1=s_1,\ldots,\, s'_r=s_r}\atop{s'_{r+1}=L}}\atop{r'=r_0}} \frac{\|{T_{cc}^{\bar{\p}_{k}^{2,n,[p]}[\r_k^{s'_1\ldots s'_{r'},U}[I]]}}^{*}(g)\|^2_{L^2(I_{\r_k^{s'_1\ldots s'_{r'}}[I]})}}{|I_{\r_k^{s'_1\ldots s'_{r'}}[I]}|}$$
$$\lesssim  \frac{\|{T_{cc}^{\bar{\p}_{k}^{2,n,[p]}[\r_k^{s_1\ldots s_{r},U}[I]]}}^{*}(g)\|^2_{L^2(I_{\r_k^{s_1\ldots s_{r}}[I]})}}{|I_{\r_k^{s_1\ldots s_{r}}[I]}|}\:.$$
\end{itemize}
Using now \eqref{obs3imp}, \eqref{k1}, \eqref{k2} and Cauchy-Schwarz we treat \eqref{claimm211} as follows:
\beq\label{claimm2111}
\eeq
$$V=\sum_{r_0\geq r+1}\sum_{{{s'_1=s_1,\ldots,\, s'_r=s_r}\atop{s'_{r+1}=L}}\atop{r'=r_0}}
|I_{\r_k^{s'_1\ldots s'_{r'}}[I]}|^{\frac{1}{2}}\,|I_{\r_k^{s'_1\ldots s'_{r'}}[I]}|^{\frac{1}{2}}$$
$$\times\frac{\|{T_{cc}^{\bar{\p}_{k}^{2,n,[p]}[\r_k^{s'_1\ldots s'_{r'},U}[I]]}}^{*}(g)\|_{L^2(I_{\r_k^{s'_1\ldots s'_{r'}}[I]})}}{|I_{\r_k^{s'_1\ldots s'_{r'}}[I]}|^{\frac{1}{2}}}$$
$$\lesssim \sum_{r_0\geq r+1} \frac{1}{2^{\frac{r_0-r}{4}}}\,|I_{\r_k^{s_1\ldots s_{r}}[I]}|^{\frac{1}{2}}\,
\left(\sum_{{{s'_1=s_1,\ldots,\, s'_r=s_r}\atop{s'_{r+1}=L}}\atop{r'=r_0}}|I_{\r_k^{s'_1\ldots s'_{r'}}[I]}|\right)^{\frac{1}{2}}\,$$
$$\times\left(\sum_{{{s'_1=s_1,\ldots,\, s'_r=s_r}\atop{s'_{r+1}=L}}\atop{r'=r_0}} \frac{\|{T_{cc}^{\bar{\p}_{k}^{2,n,[p]}[\r_k^{s'_1\ldots s'_{r'},U}[I]]}}^{*}(g)\|^2_{L^2(I_{\r_k^{s'_1\ldots s'_{r'}}[I]})}}{|I_{\r_k^{s'_1\ldots s'_{r'}}[I]}|}\right)^{\frac{1}{2}}$$
$$\lesssim \sum_{r_0\geq r+1} \frac{1}{2^{\frac{r_0-r}{4}}}\,|I_{\r_k^{s_1\ldots s_{r}}[I]}|^{\frac{1}{2}}\,
|I_{\r_k^{s_1\ldots s_{r}}[I]}|^{\frac{1}{2}}\,\times\,\frac{\|{T_{cc}^{\bar{\p}_{k}^{2,n,[p]}[\r_k^{s_1\ldots s_{r},U}[I]]}}^{*}(g)\|_{L^2(I_{\r_k^{s_1\ldots s_{r}}[I]})}}{|I_{\r_k^{s_1\ldots s_{r}}[I]}|^{\frac{1}{2}}}$$
$$\lesssim |I_{\r_k^{s_1\ldots s_{r}}[I]}|^{\frac{1}{2}}\,\|{T_{cc}^{\bar{\p}_{k}^{2,n,[p]}[\r_k^{s_1\ldots s_{r},U}[I]]}}^{*}(g)\|_{L^2(I_{\r_k^{s_1\ldots s_{r}}[I]})}\:.$$

\end{proof}

With these completed, we resume the proof of our Proposition \ref{masslessprop}:

First, we notice that introducing the notation\footnote{In order to be consistent with convention \eqref{conv0}, whenever we have sums involving families of the form $\r_k^{rU}[I]$ we need to allow $r$ to range through the set $\{-1,0,1,\ldots\}$ and apply the convention
\beq \label{conv1}
\r_k^{rU}[I]:=\left\{ \begin{array}{ll} \r_k^{0}[I]\:\:\textrm{if}\:\:r=-1\\
           \r_k^{1}[I]\:\:\textrm{if}\:\:r=0\:.
                        \end{array} \right.
\eeq}
\beq \label{levru}
\r_k^{rU}[I]:=\r_k^{\tiny{\underbrace{U\ldots U}_r}}[I]\:,
\eeq
based on \eqref{claimm11} and \eqref{claimm21}, we have
\beq \label{keyEfunds}
\eeq
$$|\E_{>}(h,g)|\lesssim\sum_{n> n_F}\sum_{p\leq k_{F}}\sum_{1<k\leq k_{F}}\,2^{-k}\,\sum_{\underline{R}\in\r_{k}}\sum_{R\in\r_k^{chi}(\underline{R})}
\int_{I_R}\left|{T_{cc}^{\bar{\p}_{k}^{2,n,[p]}[R]}}^{*}(g)\right|$$
$$\lesssim \sum_{n> n_F}\sum_{p\leq k_{F}}\sum_{1<k\leq k_{F}}\,2^{-k}\,
\sum_{I\in\I_k}\sum_{r}$$
$$\left(\sum_{\underline{R}\in \r_k^{r U}[I]} |I_{\underline{R}}|\right)^{\frac{1}{2}}\,
\left(\sum_{\underline{R}\in \r_k^{r U}[I]}\int_{I_{\underline{R}}}\left|{T_{cc}^{\bar{\p}_{k}^{2,n,[p]}[\r_k^{(r+1) U} [\underline{R}]]}}^{*}(g)\right|^2\right)^{\frac{1}{2}}\:.$$

Next, we record the following two key properties
\begin{itemize}
\item for any $1\leq k\leq k_{F}$ any $n\geq 0$ and any $I\in\I_k$, based on \eqref{CM} one has
\beq\label{kk11}
\sum_{\underline{R}\in {\r}_k^{(r+2) U}[I]} |I_{\underline{R}}|\leq \frac{1}{2}\sum_{\underline{R}\in {\r}_k^{r U}[I]} |I_{\underline{R}}|\:.
\eeq
\item for any $1\leq k\leq k_{F}$ any $r\not=r'$ and any $I\in\I_k$ one has
\beq\label{kk21}
\F(\r_k^{r U}[I])\cap \F(\r_k^{r' U}[I])=\emptyset\:.
\eeq
\end{itemize}

Applying now Cauchy-Schwarz and an almost-orthogonality argument one has
$$|\E_{>}(h,g)|\lesssim\sum_{n> n_F}\sum_{p\leq k_{F}}\left(\sum_{1<k\leq k_{F}}\,2^{-k}\,
\sum_{I\in\I_k}\sum_{r}\sum_{\underline{R}\in \r_k^{r U}[I]} |I_{\underline{R}}|\right)^{\frac{1}{2}}$$
$$\times\left(\sum_{1<k\leq k_{F}}\,2^{-k}\,\sum_{I\in\I_k}\sum_{r}\sum_{\underline{R}\in \r_k^{r U}[I]}\,
\int_{I_{\underline{R}}}\sum_{l}\left|{T_{cc}^{\bar{\p}_{k,(l)}^{2,n,[p]}[\r_k^{(r+1) U} [\underline{R}]]}}^{*}(g)\right|^2\right)^{\frac{1}{2}}$$

Making use of \eqref{kk11}, we have:
\beq \label{carlFol}
\sum_{I\in\I_k}\sum_{r}\sum_{\underline{R}\in \r_k^{r U}[I]}|I_{\underline{R}}|\lesssim |\bar{\I}_k|\,,
\eeq
Using now \eqref{carlFol}, Main Lemma and Corollary \ref{rowcont}, we deduce from above that
$$|\E_{>}(h,g)|\lesssim\sum_{n> n_F}\sum_{p\leq k_{F}}(|F|\,\log \frac{4}{|F|})^{\frac{1}{2}}$$
$$\times\left(\sum_{1<k\leq k_{F}}\,2^{-k}\,\sum_{I\in\I_k}\sum_{r}\sum_{\underline{R}\in \r_k^{r U}[I]}\,
\int_{I_{\underline{R}}}\sum_{l}\left|{T_{cc}^{\bar{\p}_{k,(l)}^{2,n,[p]}[\r_k^{(r+1) U} [\underline{R}]]}}^{*}(g)\right|^2\right)^{\frac{1}{2}}$$
$$\lesssim (|F|\,\log \frac{4}{|F|})^{\frac{1}{2}}\,\sum_{n> n_F}\sum_{p\leq k_{F}}
\left(\sum_{l} 2^{-2n}\,(\sum_{\p_{(l),a,b}^{n,(m),[p]}\in\p_{(l)}^{2,n,[p]}} 2^{-m}\,|I_{\p_{(l),a,b}^{n,(m),[p]}}|)\right)^{\frac{1}{2}}$$
$$\lesssim (|F|\,\log \frac{4}{|F|})^{\frac{1}{2}}\,\sum_{n> n_F}\,k_{F}\,2^{-\frac{n}{2}}\,|F|^{\frac{1}{2}}\lesssim |F|\,\log \frac{4}{|F|}\;,$$
where in the second inequality we used the convention
$$\p_{(l)}^{2,n,[p]}:=\bigcup_{{k\leq k_F}\atop{I\in\I_k}}\bigcup_{\underline{R}\in \r_k^{r U}[I]}\p_{k,(l)}^{2,n,[p]}[\r_k^{(r+1) U} [\underline{R}]]\:.$$

This proves the desired estimate \eqref{estimElm}.

It remains now to prove a similar estimate for the first term $\E_{<}$, that is

\begin{p1}\label{masslesprop}
With the previous notations, one has
\beq\label{estimEl}
|\E_{<}(h,g)|\lesssim |F|\,\log(\frac{4}{|F|})\,\|h\|_{\infty}\,\|g\|_{\infty}\:.
\eeq
\end{p1}

The proof of Proposition \ref{masslesprop} will be subject of the remaining sections.

\subsection{Decomposing $\E_{<}=\E_{Main}+\E_{Err}$. Reduction of Proposition \ref{masslesprop} to Proposition \ref{mainpropred}.}\label{redmr}

We start with the usual decomposition of the tile families under our analysis.

Fix $n\leq n_{F}$, $k\leq k_{F}$ and $p\leq k_{F}$. Recall from Section \ref{tffolF}, that
\beq \label{treedec}
\bar{\p}_{k}^{2,n,[p]}[R]=\bigcup_{l\geq 1}\bar{\p}_{k,(l)}^{2,n,[p]}[R]\:,
\eeq
with (see \eqref{klmtree})
\beq\label{klmtreess}
\bar{\p}_{k,(l)}^{2,n,[p]}[R]=\bigcup_{a,b,m}\bar{\p}_{k,(l),a,b}^{2,n,(m),[p]}[R]\:,
\eeq
such that each $\p_{k,(l),a,b}^{2,n,(m),[p]}[R]$ is as described in \eqref{layy1}; in particular, $\p_{k,(l),a,b}^{2,n,(m),[p]}[R]$ is a tree with uniform $F-$mass $\approx 2^{-m}$ and mass $\approx 2^{-n}$ and living at frequency $\o_l$.

With this said, we start a finer analysis of our families $\{\bar{\p}_{k,(l)}^{2,n,[p]}[R]\}_{R}$.

Throughout our discussion $R$, $n$ and $p$ are fixed. Recalling \eqref{uiprop} and picking $C_1\in\N$ with $C_1\leq 100$ we deduce that given any $\bar{\p}_{k,(l)}^{2,n,[p]}[R]$
there exist at most $C_1$ tress $\{\p_{k,(l),a_s,b_s}^{2,n,(m_s),[p]}[R]\}_{s=1}^{C_1}$ such that
\beq \label{redp}
\bar{\p}_{k,(l)}^{2,n,[p]}[R]=\bigcup_{s=1}^{C_1}\bar{\p}_{k,(l),a_s,b_s}^{2,n,(m_s),[p]}[R]\,.
\eeq

Given this, by splitting $\bar{\p}_{k,(l)}^{2,n,[p]}[R]$ in finitely many components (at most $C_1$) we will assume in all that follows that in fact
\beq \label{redps}
\bar{\p}_{k,(l)}^{2,n,[p]}[R]=\bar{\p}_{k,(l),a,b}^{2,n,(m),[p]}[R]\,,
\eeq
for some $a,\,b,\,m$ uniquely determined by $R$, $p$ and $l$.

Fix $\ep\in (0,\,\frac{1}{1000})$. With the convention made in \eqref{redps}, we split the family $\{\bar{\p}_{k,(l)}^{2,n,[p]}[R]\}_l$ into:
\begin{itemize}
\item the family of \textbf{$R$ heavy rows} defined by
\beq \label{heavytree}
\p_{k,he}^{2,n,[p]}[R]:=
\left\{\begin{array}{cl}
\bar{\p}_{k,(l)}^{2,n,[p]}[R]=\\
\bar{\p}_{k,(l),a,b}^{2,n,(m),[p]}[R]
\end{array}\,\bigg|\,
\frac{\|{T_{c}^{\bar{\p}_{k,(l),a,b}^{2,n,(m),[p]}[R]}}^{*}(g)\|_{L^2(I_{\underline{R}})}}
{|I_{\underline{R}}|^{\frac{1}{2}}}\geq 2^{n\ep}\,\frac{\|{T^{\bar{\p}_{(l),a,b}^{n,(m),[p]}}}^{*}(g)\|_{L^2(F)}}{(2^{-m}\,
|\tilde{I}_{\p_{(l),a,b}^{n,(m),[p]}}|)^{\frac{1}{2}}}\right\}\,.
\eeq

\item the family of \textbf{$R$ regular rows} defined by
\beq \label{regtree}
\p_{k,re}^{2,n,[p]}[R]:=
\left\{\begin{array}{cl}
\bar{\p}_{k,(l)}^{2,n,[p]}[R]=\\
\bar{\p}_{k,(l),a,b}^{2,n,(m),[p]}[R]
\end{array}\,\bigg|\,
\frac{\|{T_{c}^{\bar{\p}_{k,(l),a,b}^{2,n,(m),[p]}[R]}}^{*}(g)\|_{L^2(I_{\underline{R}})}}
{|I_{\underline{R}}|^{\frac{1}{2}}}< 2^{n\ep}\,\frac{\|{T^{\bar{\p}_{(l),a,b}^{n,(m),[p]}}}^{*}(g)\|_{L^2(F)}}{(2^{-m}\,
|\tilde{I}_{\p_{(l),a,b}^{n,(m),[p]}}|)^{\frac{1}{2}}}\right\}\,.
\eeq
\end{itemize}

\begin{o0}\label{constr}

It is worth noticing that the reference quantity in both  \eqref{heavytree} and \eqref{regtree} obeys via Corollary \ref{rowcont} the following estimate:
\beq \label{refqt}
\frac{\|{T^{\bar{\p}_{(l),a,b}^{n,(m),[p]}}}^{*}(g)\|_{L^2(F)}}{(2^{-m}\,
|\tilde{I}_{\p_{(l),a,b}^{n,(m),[p]}}|)^{\frac{1}{2}}}\lesssim 2^{-n}\:.
\eeq
In fact, at least at the heuristic level, one should think as $\lesssim$ replaced by $\approx$.
\end{o0}

Next, assume we are given $\s=(s_1,\ldots,
\,s_r)$. Our intention is to
further analyze the set $\r_{k,chi}^{\s}[\underline{R}]\equiv\r_{k,chi}^{s_1\ldots s_{r}}[\underline{R}]$ (recall \eqref{rchisb}):
\begin{itemize}

\item fix $\underline{R}\in \r_{k}$. We define the $(\s,p,n)-$\textit{height} relative to the base $\underline{R}$ as being simply
\beq \label{srich}
\H^{n,p}_{\s,\underline{R}}:=\#\{l\,|\,\bar{\p}_{k,(l)}^{2,n,[p]}[R]\in \p_{k,re}^{2,n,[p]}[R]\:\:\textrm{for some}\:R\in \r_{k,chi}^{\s}[\underline{R}]\}\,.
\eeq
The height $\H^{n,p}_{\s,\underline{R}}$ is indeed well defined since one can make a similar reasoning with the one in
Observation \ref{rchifoliation} and based on \eqref{regtree} deduce that
\beq \label{srichss}
\forall\:R,\,R'\in \r_{k,chi}^{\s}[\underline{R}]\:\:\Rightarrow\:\: \p_{k,re}^{2,n,[p]}[R]=\p_{k,re}^{2,n,[p]}[R']\,.
\eeq

\item for $w\in \N$ we \textit{define}
\beq \label{srichssw}
\r_{k}^{\s,[w]}(n,p,\underline{R})
\eeq
to be the set of all $R\in \r_{k,chi}^{\s}[\underline{R}]$
with the extra-property that
\begin{itemize}
\item if $w\geq 1$ one has
\beq \label{ws}
\|{T_{c}^{\p_{k,re}^{2,n,[p]}[R]}}^{*}(g)\|_{L^{\infty}(I_{R})}
\approx 2^{w}\,2^{-n}\,2^{n\ep}\,(\H^{n,p}_{\s,\underline{R}})^{\frac{1}{2}}\,;
\eeq
\item if $w=0$ we replace in \eqref{ws} the symbol $\approx$ by $\lesssim$.
\end{itemize}

Notice that with these definitions and notations one has
\beq \label{srichssws}
\r_{k,chi}^{\s}[\underline{R}]=\bigcup_{w\in\N}\r_{k}^{\s,[w]}(n,p,\underline{R})\:.
\eeq
\end{itemize}

We split now the operator associated with a generic tree $\p_{(l)}^{n,[p]}=\p_{(l),a,b}^{n,(m),[p]}$ into a main and error term respectively.
$\newline$

We start by defining the \underline{\textbf{error term}} of ${T^{\p_{(l)}^{n,[p]}}}^{*}(g)$ as
\beq \label{errter}
{T_{Err}^{\p_{(l)}^{n,[p]}}}^{*}(g):={T^{\p_{(l), he}^{n,[p]}}}^{*}(g)\,+\,{T^{\p_{(l), whe}^{n,[p]}}}^{*}(g)\,,
\eeq
where:

\begin{itemize}

\item the \textbf{heavy} component
\beq \label{heavyc}
\eeq
$${T^{\p_{(l), he}^{n,[p]}}}^{*}(g):=\sum_{1<k\leq k_{F}}\,\sum_{\underline{R}\in\r_{k}}\sum_{R\in\r_k^{chi}(\underline{R})}
{T^{(\p_{(l)}^{n,[p]}\cap\p_{k,he}^{2,n,[p]}[R])}}^{*}(g)\:.$$

\item the $w-$\textbf{heavy} component
\beq \label{wheavy}
\eeq
$${T^{\p_{(l), whe}^{n,[p]}}}^{*}(g):=
\sum_{1<k\leq k_{F}}\,\sum_{{\s}\atop{w\geq n\ep}}\sum_{\underline{R}\in\r_{k}}\sum_{R\in \r_{k}^{\s,[w]}(n,p,\underline{R})}{T^{(\p_{(l)}^{n,[p]}\cap\p_{k,re}^{2,n,[p]}[R])}}^{*}(g)\:.$$

\end{itemize}

We define the \underline{\textbf{main term}} of ${T^{\p_{(l)}^{n,[p]}}}^{*}(g)$ as
\beq \label{errter}
{T_{Main}^{\p_{(l)}^{n,[p]}}}^{*}(g):={T^{\p_{(l), L^1}^{n,[p]}}}^{*}(g)\,+\,{T^{\p_{(l), L^2}^{n,[p]}}}^{*}(g)\,,
\eeq
where:

\begin{itemize}
\item the \textbf{$L^1$-component} (or \textbf{low-height} component) is given by:
\beq \label{lowc}
\eeq
$${T^{\p_{(l), L^1}^{n,[p]}}}^{*}(g):=$$
$$\sum_{1<k\leq k_{F}}\,\sum_{{\s}\atop{w< n\ep}}
\sum_{{\underline{R}\in\r_{k}}\atop{\H^{n,p}_{\s,\underline{R}}\leq 2^{10 n \ep}}}\sum_{R\in \r_{k}^{\s,[w]}(n,p,\underline{R})}
{T^{(\p_{(l)}^{n,[p]}\cap\p_{k,re}^{2,n,[p]}[R])}}^{*}(g)$$

\item the \textbf{$L^2$-component} (or \textbf{high-height} component) is given by:
\beq \label{howc}
\eeq
$${T^{\p_{(l),L^2}^{n,[p]}}}^{*}(g):=$$
$$\sum_{1<k\leq k_{F}}\,\sum_{{\s}\atop{w<n\ep}}\sum_{{\underline{R}\in\r_{k}}\atop{\H^{n,p}_{\s,\underline{R}}> 2^{10 n \ep}}}\sum_{R\in \r_{k}^{\s,[w]}(n,p,\underline{R})}
{T^{(\p_{(l)}^{n,[p]}\cap\p_{k,re}^{2,n,[p]}[R])}}^{*}(g)\;.$$
\end{itemize}

Appealing to \eqref{keyEfund}, we create the corresponding analogues of the above tree components\footnote{Here it is more convenient to formulate these analogues in terms of $T_c^{\p}$ rather than $T^{\p}$ for appropriate families $\p$, where here $T_c^{\p}$ is the direct adaptation of \eqref{almostconst1} to our context.}, as follows:

We split the term
\beq \label{dec}
\E_{<}=\E_{Err}+\E_{Main}\,,
\eeq
with
\beq \label{err}
\E_{Err}:=\sum_{n\leq n_{F}}\sum_{p\leq k_{F}}\E_{n,Err}^{p}\,,
\eeq
and
\beq \label{main}
\E_{Main}:=\sum_{n\leq n_{F}}\sum_{p\leq k_{F}}\E_{n,Main}^{p}\,,
\eeq
where
\begin{itemize}
\item the \underline{\textbf{error term}} is further subdivided into
\beq \label{eql1}
\E_{n, Err}^{p}:=\E_{n,he}^{p}\,+\,\E_{n,whe}^{p}\:,
\eeq
with
\begin{itemize}
\item the \textbf{heavy} component given by
\beq \label{l1he}
\E_{n,he}^{p}(h,g):= \sum_{1<k\leq k_{F}}\,\sum_{\underline{R}\in\r_{k}}\sum_{R\in\r_k^{chi}(\underline{R})}
\frac{|I_R\cap F|}{|I_R|}\,\int_{I_R}h\,{T_c^{\p_{k,he}^{2,n,[p]}[R]}}^{*}(g)\,,
\eeq

\item the $w-$\textbf{heavy} component  given by
\beq \label{l1whe}
\E_{n,whe}^{p}:= \sum_{1<k\leq k_{F}}\,\sum_{{\s}\atop{w\geq n\ep}}\sum_{\underline{R}\in\r_{k}}
\sum_{R\in \r_{k}^{\s,[w]}(n,p,\underline{R})}\frac{|I_R\cap F|}{|I_R|}\,
\int_{I_R}h\,{T_c^{\p_{k,re}^{2,n,[p]}[R]}}^{*}(g)\,,
\eeq
\end{itemize}

\item the \underline{\textbf{main term}} is further subdivided into
\beq \label{eq}
\E_{n, Main}^{p}:= \E_{n, L^1}^{p}+\E_{n, L^2}^{p}\,,
\eeq
with:
\begin{itemize}
\item the \textbf{$L^1$-component} (or \textbf{low-height} component) given by:

\beq \label{l1slo}
\eeq
$$\E_{n, L^1}^{p}(h,g):= $$
$$\sum_{1<k\leq k_{F}}\,\sum_{{\s}\atop{w<n\ep}}\sum_{{\underline{R}\in\r_{k}}\atop{\atop{\H^{n,p}_{\s,\underline{R}}\leq 2^{10 n \ep}}}}
\sum_{R\in\r_{k}^{\s,[w]}(n,p,\underline{R})}\frac{|I_R\cap F|}{|I_R|}\,
\int_{I_R}h\,{T_c^{\p_{k,re}^{2,n,[p]}[R]}}^{*}(g)\:.$$

\item the \textbf{$L^2$-component} (or \textbf{high-height} component) given by:

\beq \label{keyE1}
\eeq
$$\E_{n,L^2}^{p}(h,g):= $$
$$\sum_{1<k\leq k_{F}}\,\sum_{{\s}\atop{w<n\ep}}\sum_{{\underline{R}\in\r_{k}}\atop{\H^{n,p}_{\s,\underline{R}}> 2^{10 n \ep}}}
\sum_{R\in\r_{k}^{\s,[w]}(n,p,\underline{R})} \frac{|I_R\cap F|}{|I_R|}\,
\int_{I_R}h\,{T_c^{\p_{k,re}^{2,n,[p]}[R]}}^{*}(g)\:.$$
\end{itemize}
\end{itemize}

\begin{p1}\label{mainpropred} With the above notations, we have
\beq\label{estimEmain}
|\E_{Main}(h,g)|\lesssim |F|\,\log(\frac{4}{|F|})\,\|h\|_{\infty}\,\|g\|_{\infty}\:.
\eeq
and
\beq\label{estimEerr}
|\E_{Err}(h,g)|\lesssim |F|\,\log(\frac{4}{|F|})\,\|h\|_{\infty}\,\|g\|_{\infty}\:.
\eeq
\end{p1}

\section{Treatment of the error term $\E_{Err}$}\label{Errterm}

Our goal in this section is to prove \eqref{estimEerr}.

Based on \eqref{err}-\eqref{l1whe}, we have
\beq \label{errdec}
\E_{Err}= \E_{he}+\E_{whe}\:,
\eeq
where
\beq \label{he}
\E_{he}:= \sum_{n\leq n_{F}}\sum_{p\leq k_{F}}\E_{n,he}^{p}\:,
\eeq
and
\beq \label{whe}
\E_{whe}:= \sum_{n\leq n_{F}}\sum_{p\leq k_{F}}\E_{n,whe}^{p}\:.
\eeq

\subsection{The term $\E_{he}$.}

We focus now our attention towards the term $\E_{n,he}^{p}$.

Recalling now \eqref{almostconst} and the convention made in \eqref{redps} we first remark the following key fact:

\beq \label{const}
\eeq
$$|{T_c^{\bar{\p}_{k,(l),a,b}^{2,n,(m),[p]}[R]}}^{*}(g)(x)|\big|_{x\in I_R}=\frac{\|{T_{c}^{\bar{\p}_{k,(l),a,b}^{2,n,(m),[p]}[R]}}^{*}(g)\|_{L^2(I_{R})}}{|I_{R}|^{\frac{1}{2}}}$$
$$=\frac{\|{T_{c}^{\bar{\p}_{k,(l),a,b}^{2,n,(m),[p]}[R]}}^{*}(g)\|_{L^2(I_{\underline{R}})}}{|I_{\underline{R}}|^{\frac{1}{2}}}=
\frac{\|{T^{\bar{\p}_{(l),a,b}^{n,(m),[p]}}}^{*}(g)\|_{L^2(I_{\underline{R}})}}{|I_{\underline{R}}|^{\frac{1}{2}}}\:.$$

Next, maintaining the assumption in \eqref{redps}, that is
\beq \label{redpss}
\bar{\p}_{k,(l)}^{2,n,[p]}[R]=\bar{\p}_{k,(l),a,b}^{2,n,(m),[p]}[R]\,,
\eeq
we notice that by applying the same techniques as in the proof of Lemma \ref{treecutp} - see \eqref{diff}, we get that for any
$x\in I_R\cap F$ one has
\beq \label{diffa}
|{T^{\bar{\p}_{(l),a,b}^{n,(m),[p]}}}^{*}(g)(x)-{T_c^{\p_{k,(l),he}^{2,n,[p]}[R]}}^{*}(g)(x)|\lesssim
\sum_{{{P\in\p_{k,(l),he}^{2,n,[p]}[R]}}\atop{I_{P^*}\supseteq I_{\underline{R}}}} \frac{|I_{\underline{R}}|}{|I_P|}\frac{\int_{E(P)}|g|}{|I_P|}\:.
\eeq

Define $I_{\p_{(l),a,b}^{n,(m),[p]}}$ as the time interval of the top of the tree $\p_{(l),a,b}^{n,(m),[p]}$ and let $\tilde{I}_{\bar{\p}_{(l),a,b}^{n,(m),[p]}}:=2\tilde{I}_{\p_{(l),a,b}^{n,(m),[p]}}$. Based on this, we notice that
$$\textrm{supp}\,{T^{\bar{\p}_{(l),a,b}^{n,(m),[p]}}}^{*}\subseteq \tilde{I}_{\bar{\p}_{(l),a,b}^{n,(m),[p]}}\,.$$

Set now
\beq \label{mnl}
A_{l,a,b}^{n,m,p}(F):=\left\{x\in \tilde{I}_{\bar{\p}_{(l),a,b}^{n,(m),[p]}}\cap F\,\big|\,|{T^{\bar{\p}_{(l),a,b}^{n,(m),[p]}}}^{*}(g)(x)|\geq 2^{n\ep-1}\,\frac{\|{T^{\bar{\p}_{(l),a,b}^{n,(m),[p]}}}^{*}(g)\|_{L^2(F)}}{(2^{-m}\,
|\tilde{I}_{\p_{(l),a,b}^{n,(m),[p]}}|)^{\frac{1}{2}}}\right\}\,.
\eeq

Recalling definition \eqref{heavytree} and based on \eqref{const}-\eqref{mnl}, we now have
\beq \label{rewH11}
\eeq
$$|\E_{n,he}^{p}(h,g)|\lesssim\sum_{1<k\leq k_{F}}\,\sum_{\underline{R}\in\r_{k}}\sum_{R\in\r_k^{chi}(\underline{R})}
\frac{|I_R\cap F|}{|I_R|}\,\int_{I_R}\left|{T_c^{\p_{k,he}^{2,n,[p]}[R]}}^{*}(g)\right|$$
$$\lesssim\sum_{1<k\leq k_{F}}\,\sum_{\underline{R}\in\r_{k}}\sum_{R\in\r_k^{chi}(\underline{R})}
\sum_l\int_{I_R\cap F}\left|{T_c^{\p_{k,(l),he}^{2,n,[p]}[R]}}^{*}(g)\right|$$

$$\lesssim\sum_{l}\sum_{1<k\leq k_{F}}\,\sum_{\underline{R}\in\r_{k}}\sum_{R\in\r_k^{chi}(\underline{R})}
\,|I_R\cap F|\,\sum_{{{P\in\p_{k,(l),he}^{2,n,[p]}[R]}}\atop{I_{P^*}\supseteq I_{\underline{R}}}} \frac{|I_{\underline{R}}|}{|I_P|}\frac{\int_{E(P)}|g|}{|I_P|}$$
$$+\,\sum_{l}\int_{A_{l,a,b}^{n,m,p}(F)}|{T^{\bar{\p}_{(l),a,b}^{n,(m),[p]}}}^{*}g|$$
$$=: U^{n,[p]} \,+\,\sum_{l}\int_{A_{l,a,b}^{n,m,p}(F)}|{T^{\bar{\p}_{(l),a,b}^{n,(m),[p]}}}^{*}g|\,.$$

Now from Chebyshev we deduce that
\beq \label{sizemnl}
|A_{l,a,b}^{n,m,p}(F)|\lesssim 2^{-2n\ep}\,2^{-m}\,
|\tilde{I}_{\p_{(l),a,b}^{n,(m),[p]}}|\:.
\eeq
Inserting now \eqref{sizemnl} in \eqref{rewH11} and using Cauchy-Schwarz and Corollary \ref{rowcont}, we deduce
\beq \label{rewH12}
\eeq
$$|\E_{n,he}^{p}(h,g)|$$
$$\lesssim U^{n,[p]}\,+\, \sum_{l}\sum_{\p_{(l),a,b}^{n, (m),[p]}\in \p_{(l),*}^{n,[p]}}|A_{l,a,b}^{n,m,p}(F)|^{\frac{1}{2}}\,(\int_{F}|{T^{\bar{\p}_{(l),a,b}^{n,(m),[p]}}}^{*}g|^2)^{\frac{1}{2}}$$
$$\lesssim U^{n,[p]}\,+\,2^{-n\ep}\,\sum_{l}
\sum_{\p_{(l),a,b}^{n, (m),[p]}\in \p_{(l),*}^{n,[p]}}\,2^{-n}\,2^{-m}\,
|I_{\p_{(l),a,b}^{n,(m),[p]}}|\:.$$
Finally, recalling the proof of Lemma \ref{difcont} that relies on the estimate of the term defined in \eqref{defU}, we conclude
\beq \label{rewHc}
\eeq
$$|\E_{he}(h,g)|\lesssim$$
$$ U\,+\,
\sum_{n\leq n_{F}}\sum_{p\leq k_{F}}
2^{-n\ep}\,\sum_{l}
\sum_{\p_{(l),a,b}^{n, (m),[p]}\in \p_{(l),*}^{n,[p]}}\,2^{-n}\,2^{-m}\,
|I_{\p_{(l),a,b}^{n,(m),[p]}}|$$
$$\lesssim |F|\,\log\frac{4}{|F|}\,.$$

\subsection{The term $\E_{whe}$.}
We pass to analyzing the term $\E_{n,whe}^{m}$.

Departing from definition \eqref{l1whe}, and making use of \eqref{ws}, we have
\beq \label{rewHw}
\eeq

$$|\E_{n,whe}^{p}(h,g)|\lesssim\sum_{1<k\leq k_{F}}\,2^{-k}\,\sum_{{\s}\atop{w\geq n\ep}}\sum_{\underline{R}\in\r_{k}}
\sum_{R\in \r_{k}^{\s,[w]}(n,p,\underline{R})}
\int_{I_R}\left|{T_c^{\p_{k,re}^{2,n,[p]}[R]}}^{*}(g)\right|$$

$$\lesssim \sum_{1<k\leq k_{F}}\,2^{-k}\,\sum_{{\s}\atop{w\geq n\ep}}\sum_{\underline{R}\in\r_{k}}
\sum_{R\in \r_{k}^{\s,[w]}(n,p,\underline{R})}\,|I_{R}|\,
2^{w}\,2^{-n}\,2^{n\ep}\,(\H^{n,p}_{\s,\underline{R}})^{\frac{1}{2}}\:.$$

We stop here for introducing some useful notation.

First of all, for a fixed $I\in\I_k$ and $\underline{R}\in \r_k^{s_1\ldots s_{r}}[I]$, we notice that any children  $R\in\r_k^{chi}(\underline{R})$ must have the form $R\in \r_{k,chi}^{s'_1\ldots s'_{r'}}[\underline{R}]$ with $r'>r$ such that  $s'_j=s_j$ if $j\leq r$, $s'_{r+1}=U$ and, if $r'>r+1$ then $s'_j=L$ if $r+1<j\leq r'$.

Now, if $\s=(s_1,\ldots,\,s_r)$ represents the standard notation for a given sequence, then we will write
$\underline{R}\in \r_k^{s_1\ldots s_{r}}[I]$ as $\underline{R}\in \r_k^{\s}[I]$ and $R\in\r_k^{chi}(\underline{R})$ with
$R\in \r_{k,chi}^{s'_1\ldots s'_{r'}}[\underline{R}]$ as $R\in \r_{k,chi}^{\s, U, (r'-r-1)L}[\underline{R}]$.
This notation will become very useful when distinguishing among the generations of the children within $\r_k^{chi}(\underline{R})$.

Now notice that given any two $R,\,R'\in \r_{k,chi}^{\s,U,rL,[w]}(n,p,\underline{R})$ one has
$$\p_{k,re}^{2,n,[p]}[R']=\p_{k,re}^{2,n,[p]}[R]\,,$$
and thus the following is well defined:
$$\p_{k,re}^{2,n,[p]}[ \r_{k,chi}^{\s,U,rL,[w]}(n,p,\underline{R})]:= \p_{k,re}^{2,n,[p]}[R]\:\:\textrm{for some}\:R\in \r_{k,chi}^{\s,U,rL,[w]}(n,p,\underline{R})\:.$$
Also, we set
\beq \label{rre}
\tilde{\I}_{\r_{k}^{\s,r,[w]}(n,p,\underline{R})}:=\bigcup_{R\in \r_{k,chi}^{\s,U,rL,[w]}(n,p,\underline{R})} I_{R}\:,
\eeq
and for any $R\in \r_{k,chi}^{\s,U,rL,[w]}(n,p,\underline{R})$ based on \eqref{srich} and \eqref{srichss} it makes sense to define
\beq \label{rreem}
\H_{\r_{k}^{\s,r,[w]}(n,p,\underline{R})}:=\H^{n,p}_{\s,\underline{R}}\:.
\eeq

Then, \eqref{rewHw} can be rewritten as
\beq \label{rewHww1}
\eeq
$$|\E_{n,whe}^{p}(h,g)|\lesssim$$
$$\sum_{1<k\leq k_{F}}\,2^{-k}\,\sum_{I\in\I_{k}}\sum_{\s}\sum_{\underline{R}\in\r_{k}^{\s}[I]} \sum_{{r\geq 0}\atop {w\geq n\ep}} |\tilde{\I}_{\r_{k}^{\s,r,[w]}(n,p,\underline{R})}|\,
2^{w}\,2^{-n}\,2^{n\ep}\,(\H_{\r_{k}^{\s,r,[w]}(n,p,\underline{R})})^{\frac{1}{2}}\:.$$

Now, corroborating the following elements:
\begin{itemize}
\item definition of the tiles inside $\p_{k,re}^{2,n,[p]}[R]$ and their frequency position relative to $\underline{R}$;

\item the lacunary structure of the frequencies of our tiles;

\item relations \eqref{regtree} and  \eqref{ws};
\end{itemize}
we can apply Zygmund's inequality (see \cite{LaDo} and also Section 4 in \cite{lvKony1}), to deduce the following key John-Nirenberg type condition:

there exists $c>0$ absolute constant such that
\beq \label{JN}
|\tilde{\I}_{\r_{k}^{\s,r,[w]}(n,p,\underline{R})}|\lesssim e^{-c 2^{2w}}\,|I_{\underline{R}}|\:.
\eeq

We now have

\begin{l1}\label{controlbasisnew}
Fix $I\in\I_k$, and $\underline{R}\in\r_k^{\s}[I]$.

Then, with the above notations, for any $0<\ep\leq \frac{1}{2}$ the following holds:
\beq\label{clb}
\eeq
$$\sum_{r\geq 0}\,|\tilde{\I}_{\r_{k}^{\s,r,[w]}(n,p,\underline{R})}|\,(\H_{\r_{k}^{\s,r,[w]}(n,p,\underline{R})})^{\frac{1}{2}}
\lesssim_{\ep} e^{-c 2^{2w}}\,|I_{\underline{R}}|\,(\H_{\r_{k}^{\s,0}(n,p,\underline{R})})^{\frac{1}{2}+\ep}\:,$$
where above we denoted with
$$\H_{\r_{k}^{\s,0}(n,p,\underline{R})}:=\#\{l\,|\,\bar{\p}_{k,(l)}^{2,n,[p]}[R]\in \p_{k,re}^{2,n,[p]}[R]\:\:\textrm{for some}\:R\in \r_{k,chi}^{\s,U,0}(n,p,\underline{R})\}\,,$$
and here $c>0$ is the same as in \eqref{JN}.
\end{l1}
\begin{proof}
The above statement and its proof here can be regarded as a refinement/specialization of Lemma \ref{controlbasis}.

We start by noticing that given any  $r'\geq r$ one has
$$\p_{k,re}^{2,n,[p]}[ \r_{k,chi}^{\s,U,r'L,[w]}(n,p,\underline{R})]\subseteq\p_{k,re}^{2,n,[p]}[ \r_{k,chi}^{\s,U,rL,[w]}(n,p,\underline{R})]\,,$$
which in particular implies that
\beq\label{clbsa}
\H_{\r_{k}^{\s,r',[w]}(n,p,\underline{R})}\leq \H_{\r_{k}^{\s,r,[w]}(n,p,\underline{R})}\leq \H_{\r_{k}^{\s,0}(n,p,\underline{R})}\:.
\eeq
Next, based on \eqref{clbsa}, we observe that in order to prove \eqref{clb} it is enough to show that

\beq\label{clb1}
\eeq
$$\sum_{\H_{\r_{k}^{\s,r,[w]}(n,p,\underline{R})}> \frac{1}{2}\,\H_{\r_{k}^{\s,0}(n,p,\underline{R})}}\,|\tilde{\I}_{\r_{k}^{\s,r,[w]}(n,p,\underline{R})}|\,
(\H_{\r_{k}^{\s,r,[w]}(n,p,\underline{R})})^{\frac{1}{2}}$$
$$\lesssim_{\ep} e^{-c 2^{2w}}\,|I_{\underline{R}}|\,(\H_{\r_{k}^{\s,0}(n,p,\underline{R})})^{\frac{1}{2}+\ep}\:.$$

Indeed, once \eqref{clb1} is proved, one can iterate the same argument in order to obtain for $k\in\N$ the general estimate
\beq\label{clb2}
\eeq
$$\sum_{{\frac{1}{2^{k+1}}\,\H_{\r_{k}^{\s,0}(n,p,\underline{R})}<\H_{\r_{k}^{\s,r,[w]}(n,p,\underline{R})}\leq \frac{1}{2^k}\,\H_{\r_{k}^{\s,0}(n,p,\underline{R})}}\atop{\frac{\H_{\r_{k}^{\s,0}(n,p,\underline{R})}}{2^{k+1}}}\geq 2^{2w}}\,|\tilde{\I}_{\r_{k}^{\s,r,[w]}(n,p,\underline{R})}|\,(\H_{\r_{k}^{\s,r,[w]}(n,p,\underline{R})})^{\frac{1}{2}}$$
$$\lesssim_{\ep} e^{-c 2^{2w}}\,|I_{\underline{R}}|\,(\frac{\H_{\r_{k}^{\s,0}(n,p,\underline{R})}}{2^k})^{\frac{1}{2}+\ep}\:,$$
which summing up over $k$ gives the desired estimate \eqref{clb}.

Thus, it remains now to prove \eqref{clb1}.

We next observe that wlog we may suppose that
\beq\label{clb3}
r\not=r'\:\:\:\Leftrightarrow\:\:\:\H_{\r_{k}^{\s,r',[w]}(n,p,\underline{R})}\not= \H_{\r_{k}^{\s,r,[w]}(n,p,\underline{R})}\:.
\eeq
Indeed, otherwise we can regroup the terms in the LHS of \eqref{clb1} and notice that letting $\tilde{\I}_{\r_{k}^{\s,[w]}(n,p,\underline{R})}^{\a}$ be the union of the spacial intervals $\{\tilde{\I}_{\r_{k}^{\s,r,[w]}(n,p,\underline{R})}\}$ having the property that $\H_{\r_{k}^{\s,r,[w]}(n,p,\underline{R})}=\a$ for some fixed $\a\in\N$, one still has the analogue of \eqref{JN}, that is
\beq \label{JNcas}
|\tilde{\I}_{\r_{k}^{\s,[w]}(n,p,\underline{R})}^{\a}|\lesssim e^{-c 2^{2w}}\,|I_{\underline{R}}|\:.
\eeq

With these done, we now proceed as follows: fix $\ep\in (0,\frac{1}{2}]$ as appearing in our hypothesis and assume wlog that $(\H_{\r_{k}^{\s,0}(n,p,\underline{R})})^{\ep},\,(\H_{\r_{k}^{\s,0}(n,p,\underline{R})})^{1-\ep}\in\N$.
For $i\in\{0,\ldots, \frac{1}{2}\,(\H_{\r_{k}^{\s,0}(n,p,\underline{R})})^{\ep}-1\}$ let us set
$$A_i:=\H_{\r_{k}^{\s,0}(n,p,\underline{R})}-\left[i\,(\H_{\r_{k}^{\s,0}(n,p,\underline{R})})^{1-\ep},\,
(i+1)\,(\H_{\r_{k}^{\s,0}(n,p,\underline{R})})^{1-\ep}\right)$$ and choose
$r_i$ to be the smallest $r$ for which $\H_{\r_{k}^{\s,r,[w]}(n,p,\underline{R})}\in A_i$.

Let now
$$F_{0}=F_{\r_{k}^{\s,0}(n,p,\underline{R})}:=$$
$$\{l\,|\,\bar{\p}_{k,(l)}^{2,n,[p]}[R]\in \p_{k,re}^{2,n,[p]}[R]\:\:\textrm{for some}\:R\in \r_{k,chi}^{\s,U,0}(n,p,\underline{R})\}\,,$$ and similarly, for $p>0$, set
$$F_{r}=F_{\r_{k}^{\s,r,[w]}(n,p,\underline{R})}:=$$
$$\{l\,|\,\bar{\p}_{k,(l)}^{2,n,[p]}[R]\in \p_{k,re}^{2,n,[p]}[R]\:\:\textrm{for some}\:R\in \r_{k}^{\s,r,[w]}(n,p,\underline{R})\}\,.$$
Deduce that
$$F_{r+1}\subsetneq F_r\:\:\textrm{and that}\:\: \#F_r=\H_{\r_{k}^{\s,r,[w]}(n,p,\underline{R})}\:.$$

Observe now that \eqref{clb1} follows if for any $i\in\{0,\ldots, \frac{1}{2}\,(\H_{\r_{k}^{\s,0}(n,p,\underline{R})})^{\ep}-1\}$ we show that

\beq\label{clb13}
\sum_{r=r_{i}}^{r_{i+1}-1}\,|\tilde{\I}_{\r_{k}^{\s,r,[w]}(n,p,\underline{R})}|
\lesssim e^{-c 2^{2w}}\,|I_{\underline{R}}|\:.
\eeq

We now rephrase \eqref{ws} in the form
\beq \label{wsref}
\frac{1}{2^{-n}\,2^{n\ep}\,(\#F_r)^{\frac{1}{2}}}\,\|\sum_{l\in F_r} {T_{c}^{\bar{\p}_{k,(l)}^{2,n,[p]}[\r_{k}^{\s,r,[w]}(n,p,\underline{R})]}}^{*}(g)\|_{L^{\infty}(\tilde{\I}_{\r_{k}^{\s,r,[w]}(n,p,\underline{R})})}
\approx 2^{w}\,.
\eeq

\begin{obs}\label{const} It is worth noticing that for any $R\in \r_{k,chi}^{\s,U,rL,[w]}(n,p,\underline{R})$ with $w>>1$ we have that the term $\sum_{l\in F_r} {T_{c}^{\bar{\p}_{k,(l)}^{2,n,[p]}[\r_{k}^{\s,r,[w]}(n,p,\underline{R})]}}^{*}(g)$ is morally constant on $I_R$, in the sense that
\beq \label{almconst}
\begin{array}{cl}
\sup_{x\in I_{R}}|\sum_{l\in F_r} {T_{c}^{\bar{\p}_{k,(l)}^{2,n,[p]}[\r_{k}^{\s,r,[w]}(n,p,\underline{R})]}}^{*}(g)(x)|\\\\
\approx \inf_{x\in I_{R}}|\sum_{l\in F_r} {T_{c}^{\bar{\p}_{k,(l)}^{2,n,[p]}[\r_{k}^{\s,r,[w]}(n,p,\underline{R})]}}^{*}(g)(x)|\approx 2^{w}\,2^{-n}\,2^{n\ep}\,(\#F_r)^{\frac{1}{2}}\:.
\end{array}
\eeq
Indeed, this follows from a similar reasoning with the one used to prove heuristic \ref{heur} (see in particular \eqref{const1}), together with the estimate \eqref{wsref}.
\end{obs}

As a consequence of the above observation, applying Zygmund's inequality, we notice that for any $r_{i}\leq r< r_{i+1}$ and for each $R\in \r_{k,chi}^{\s,U,rL,[w]}(n,p,\underline{R})$ (remark from \eqref{rre} that $ I_{R}\subseteq \tilde{\I}_{\r_{k}^{\s,r,[w]}(n,p,\underline{R})}$) we have
\beq \label{wsreff}
\eeq
$$\ln \frac{|I_R|}{|I_{\underline{R}}|}\lesssim$$
$$-c\max\left\{\frac{\|\sum_{l\in F_{r_{i+1}}} {T_{c}^{\bar{\p}_{k,(l)}^{2,n,[p]}[\r_{k}^{\s,r,[w]}(n,p,\underline{R})]}}^{*}(g)\|_{L^{\infty}(I_R)}}
{2^{-n}\,2^{n\ep}\,(\#F_{r_{i+1}})^{\frac{1}{2}}},\right.$$
$$\,\left.\frac{\|\sum_{l\in F_{r}\setminus F_{r_{i+1}}} {T_{c}^{\bar{\p}_{k,(l)}^{2,n,[p]}[\r_{k}^{\s,r,[w]}(n,p,\underline{R})]}}^{*}(g)\|_{L^{\infty}(I_R)}}
{2^{-n}\,2^{n\ep}\,(\#(F_{r}\setminus F_{r_{i+1}}))^{\frac{1}{2}}} \right\}^2$$

Set now $$Y_r^{i}(R):=\frac{\|\sum_{l\in F_{r_{i+1}}} {T_{c}^{\bar{\p}_{k,(l)}^{2,n,[p]}[\r_{k}^{\s,r,[w]}(n,p,\underline{R})]}}^{*}(g)\|_{L^{\infty}(I_R)}}
{2^{-n}\,2^{n\ep}\,(\#F_{r_{i+1}})^{\frac{1}{2}}}\:,$$
and using now \eqref{wsref} and triangle inequality we rewrite \eqref{wsreff} as
\beq \label{wsrefff}
\eeq
$$\ln \frac{|I_R|}{|I_{\underline{R}}|}\lesssim$$
$$-c\max\left\{Y_r^{i}(R),\,2^{w}\,(\frac{\#F_r}{\#(F_{r}\setminus F_{r_{i+1}})})^{\frac{1}{2}}- Y_r^{i}(R)\,(\frac{\#F_{r_{i+1}}}{\#(F_{r}\setminus F_{r_{i+1}})})^{\frac{1}{2}}\right\}^2$$
Define now the following quantities:
$$\r_{k,chi,ba}^{\s,U,rL,[w]}(n,p,\underline{R}):=\{R\in \r_{k,chi}^{\s,U,rL,[w]}(n,p,\underline{R})\,|\,Y_r^{i}(R)\approx 2^{w}\}\,,$$
$$\r_{k,chi,to}^{\s,U,rL,[w]}(n,p,\underline{R}):=\r_{k,chi}^{\s,U,rL,[w]}(n,p,\underline{R})\setminus \r_{k,chi,ba}^{\s,U,rL,[w]}(n,p,\underline{R})\,,$$
$$\tilde{\I}_{\r_{k,ba}^{\s,r,[w]}(n,p,\underline{R})}:=\bigcup_{R\in \r_{k,chi,ba}^{\s,U,rL,[w]}(n,p,\underline{R})} I_{R}$$
and
$$\tilde{\I}_{\r_{k,to}^{\s,r,[w]}(n,p,\underline{R})}:=\bigcup_{R\in \r_{k,chi,to}^{\s,U,rL,[w]}(n,p,\underline{R})} I_{R}\:.$$
The key two observations are the following:
\beq\label{keyob1}
\sum_{r=r_{i}}^{r_{i+1}-1}\,|\tilde{\I}_{\r_{k,ba}^{\s,r,[w]}(n,p,\underline{R})}|
\lesssim e^{-c 2^{2w}}\,|I_{\underline{R}}|\:,
\eeq
and
\beq\label{keyob2}
|\tilde{\I}_{\r_{k,to}^{\s,r,[w]}(n,p,\underline{R})}|
\lesssim e^{-c 2^{2w}\,\frac{\#F_r}{\#(F_{r}\setminus F_{r_{i+1}})}}\,|I_{\underline{R}}|\:.
\eeq
With these, setting $L:=\H_{\r_{k}^{\s,0}(n,p,\underline{R})})$ we now have that
\beq\label{clb14}
\eeq
$$\sum_{r=r_{i}}^{r_{i+1}-1}\,|\tilde{\I}_{\r_{k}^{\s,r,[w]}(n,p,\underline{R})}|=$$
$$\sum_{r=r_{i}}^{r_{i+1}-1}\,|\tilde{\I}_{\r_{k,ba}^{\s,r,[w]}(n,p,\underline{R})}|+
\sum_{r=r_{i}}^{r_{i+1}-1}\,|\tilde{\I}_{\r_{k,to}^{\s,r,[w]}(n,p,\underline{R})}|$$
$$\lesssim  e^{-c 2^{2w}}\,|I_{\underline{R}}|\, (1+\sum_{r=r_{i}}^{r_{i+1}-1}
e^{-c\frac{\#F_r}{\#(F_{r}\setminus F_{r_{i+1}})}})$$
$$\lesssim_{\ep}  e^{-c 2^{2w}}\,|I_{\underline{R}}|\, (1+\sum_{l=1}^{L^{1-\ep}}(\frac{l}{L})^{\frac{1}{\ep}})$$
$$\lesssim_{\ep} e^{-c 2^{2w}}\,|I_{\underline{R}}|\:.$$
\end{proof}

Returning now to estimate \eqref{rewHww1} and using \eqref{clb} we have

$$|\E_{n,whe}^{p}(h,g)|$$
$$\lesssim
\sum_{1<k\leq k_{F}}\,2^{-k}\,\sum_{I\in\I_{k}}\sum_{\s}\sum_{\underline{R}\in\r_{k}^{\s}[I]} \sum_{w\geq n\ep}
e^{-c 2^{2w}}\,2^{w}\,2^{n\ep}\,2^{-n}\,\H_{\r_{k}^{\s,0}(n,p,\underline{R})}\,|I_{\underline{R}}|$$
$$\lesssim  2^{-n\ep}\,\sum_{l}
\sum_{\p_{(l),a,b}^{n, (m),[p]}\in \p_{(l),*}^{n,[p]}}\,2^{-n}\,2^{-m}\,
|\tilde{I}_{\p_{(l),a,b}^{n,(m),[p]}}|\:.$$
Thus
\beq \label{rewheconc}
\eeq
$$|\E_{whe}(h,g)|=\left |\sum_{n\leq n_{F}}\sum_{p\leq k_{F}}\E_{n,whe}^{p}(h,g)\right |$$
$$\lesssim\sum_{n\leq n_{F}}\sum_{p\leq k_{F}}2^{-n\ep}\,
\sum_{l}
\sum_{\p_{(l),a,b}^{n, (m),[p]}\in \p_{(l),*}^{n,[p]}}\,2^{-n}\,2^{-m}\,
|\tilde{I}_{\p_{(l),a,b}^{n,(m),[p]}}|\lesssim |F|\,\log\frac{4}{|F|}\,.$$

\section{Treatment of the main term $\E_{Main}$}\label{mainthm}

\subsection{Treatment of the $L^2-$term $\E_{L^2}:=\sum_{n\leq n_{F}}\sum_{p\leq k_{F}}\E_{n,L^2}^{p}$}
$\newline$

We start by analyzing the term $\E_{n,L^2}^{p}$.

Making now use of \eqref{ws} and \eqref{keyE1} we deduce that
\beq \label{keyE2}
\eeq
$$|\E_{n,L^2}^{p}(h,g)|$$
$$\lesssim\sum_{1<k\leq k_{F}}\,\sum_{{\s}\atop{w< n\ep}}\sum_{{\underline{R}\in\r_{k}}\atop{\H^{n,p}_{\s,\underline{R}}> 2^{10 n \ep}}}
\sum_{R\in\r_{k}^{\s,[w]}(n,p,\underline{R})}\frac{|I_R\cap F|}{|I_R|}\,
\int_{I_R}\left|{T_c^{\p_{k,re}^{2,n,[p]}[R]}}^{*}(g)\right|$$
$$\lesssim\sum_{1<k\leq k_{F}}\,\sum_{{\s}\atop{w< n\ep}}\sum_{{\underline{R}\in\r_{k}}\atop{\H^{n,p}_{\s,\underline{R}}> 2^{10 n \ep}}}
\sum_{R\in\r_{k}^{\s,[w]}(n,p,\underline{R})}|I_R\cap F|\,
\|{T_{c}^{\p_{k,re}^{2,n,[p]}[R]}}^{*}(g)\|_{L^{\infty}(I_{R})}$$
$$\lesssim\sum_{1<k\leq k_{F}}\,\sum_{\s}\sum_{{\underline{R}\in\r_{k}}\atop{\H^{n,p}_{\s,\underline{R}}> 2^{10 n \ep}}}
\sum_{R\in\r_{k}^{\s}(n,p,\underline{R})} 2^{2n\ep}\,(\H^{n,p}_{\s,\underline{R}})^{\frac{1}{2}}\,2^{-n}\,|I_{R}\cap F|$$
$$\lesssim\sum_{1<k\leq k_{F}}\,\sum_{\s}\sum_{{\underline{R}\in\r_{k}}\atop{\H^{n,p}_{\s,\underline{R}}> 2^{10 n \ep}}}
\sum_{R\in\r_{k}^{\s}(n,p,\underline{R})} 2^{-3n\ep}\,\H^{n,p}_{\s,\underline{R}}\,2^{-n}\,|I_{R}\cap F|$$
$$\lesssim_{Fubini} 2^{-3n\ep}\sum_{l}
\sum_{\p_{(l),a,b}^{n, (m),[p]}\in \p_{(l),*}^{n,[p]}}\,2^{-n}\,
|\tilde{I}_{\bar{\p}_{(l),a,b}^{n,(m),[p]}}\cap F|\:.$$

If we now set \beq \label{el2}
\E_{n,L^2}:=\sum_{p\leq k_{F}}\E_{n,L^2}^{p}\,,
\eeq
from above, we deduce that
$$|\E_{n,L^2}(h,g)|\lesssim $$
$$2^{-3n\ep}\sum_{p\leq k_{F}}\sum_{l}
\sum_{\p_{(l),a,b}^{n, (m),[p]}\in \p_{(l),*}^{n,[p]}}\,2^{-n}\,
|\tilde{I}_{\p_{(l),a,b}^{n,(m),[p]}}\cap F|\lesssim 2^{-3n\ep}\,|F|\,\log\frac{4}{|F|}\:,$$
which implies
\beq \label{ell2}
|\E_{L^2}(h,g)|=\left|\sum_{n\leq n_{F}}|\E_{n,L^2}(h,g)\right|\lesssim |F|\,\log\frac{4}{|F|}\,.
\eeq

\subsection{Treatment of the $L^1-$term $\E_{L^1}:=\sum_{n\leq n_{F}}\sum_{p\leq k_{F}}\E_{n,L^1}^{p}$}
$\newline$

This entire section is dedicated to the analysis of the term $\E_{L^1}$. Along the way, we will develop a methodology that will treat simultaneously families of tiles with distinct mass parameters.
\medskip

Indeed, unlike the previous cases, we will not be able - for the generic case - to obtain direct decay in the mass parameter $n$ for the term $\E_{n,L^1}^{p}$;  instead we will need to play a subtle game between the mass and $F$-mass scales of the maximal trees involved in our decomposition. Our approach will be based on two new ideas:
\begin{itemize}
\item the tree foliation of the time-frequency plane introduced in Section \ref{foltrees1};

\item the ($E$-)mass versus $F$-mass dichotomy defined in Section \ref{EF}.
\end{itemize}

\subsubsection{Identifying the main component of $\E_{L^1}$}\label{rowrefinementt}

We start by recalling the main term analyzed in this section:
\beq \label{l1slo11}
\eeq
$$\E_{L^1}(h,g)=\sum_{n\leq n_{F}}\sum_{p\leq k_{F}}\E_{n,L^1}^{p}(h,g)$$
$$=\sum_{{n\leq n_{F}}\atop{p\leq k_{F}}}\sum_{1<k\leq k_{F}}\sum_{{\s}\atop{w< n\ep}}\sum_{{\underline{R}\in\r_{k}}\atop{\H^{n,p}_{\s,\underline{R}}\leq 2^{10 n \ep}}}
\sum_{R\in\r_{k}^{\s,[w]}(n,p,\underline{R})}\frac{|I_R\cap F|}{|I_R|}\,
\int_{I_R}h\,{T_c^{\p_{k,re}^{2,n,[p]}[R]}}^{*}(g)$$
$$\lesssim\sum_{{n\leq n_{F}}\atop{p\leq k_{F}}}\sum_{1<k\leq k_{F}}\,2^{-k}\,\sum_{{\s}\atop{w< n\ep}}\sum_{{\underline{R}\in\r_{k}}\atop{\H^{n,p}_{\s,\underline{R}}\leq 2^{10 n \ep}}}
\sum_{R\in\r_{k}^{\s,[w]}(n,p,\underline{R})}
\int_{I_R}h\,{T_c^{\p_{k,re}^{2,n,[p]}[R]}}^{*}(g)\:.$$

\noindent  We introduce the following

\begin{d0}\label{lp} Let $1\leq p \leq k_F$. We say that $\p_{(l),a,b}^{n,(m)}=\p_{(l),a,b}^{n,(m),[p]}$ is an \textbf{$(L^1,\,F)$-saturated tree} iff
\beq \label{llow}
\sum_{1<k\leq k_{F}}\,\sum_{{\s}\atop{w< n\ep}}\sum_{{\underline{R}\in\r_{k}}\atop{\H^{n,p}_{\s,\underline{R}}\leq 2^{10 n \ep}}}
\sum_{R\in\r_{k}^{\s,[w]}(n,p,\underline{R})} 2^{-k}\,|I_{\bar{\p}_{k,(l),a,b,re}^{2,n,(m),[p]}[R]}^{R}|\geq \frac{2^{-m}}{2^{2n\ep}}\,|\tilde{I}_{P_{(l),a,b}^{n,(m),[p]}}|\:,
\eeq
where, as expected, here
$$\bar{\p}_{k,(l), a,b, re}^{2,n,(m),[p]}[R]:=\p_{k,(l),re}^{2,n}[R]\cap \bar{\p}_{(l),a,b}^{n,(m),[p]}\:\:\textrm{and}\:\:
I_{\bar{\p}_{k,(l),a,b,re}^{2,n,(m),[p]}[R]}^{R}:=I_R\cap \tilde{I}_{\bar{\p}_{(l),a,b}^{n,(m),[p]}}\:.$$

We let $\F_{sat}$ be the collection of all $\p_{(l),a,b}^{n,(m),[p]}$ with $n\leq n_{F},\,p\leq k_{F},\,a,\,b\geq 1$ which are $(L^1,\,F)$-saturated trees.
\end{d0}

With these done, we decompose
\beq \label{l1slo112}
\E_{L^1}(h,g)=\E_{L^1}^{sat}(h,g)+\E_{L^1}^{nsat}(h,g)\,,
\eeq
where
\beq \label{l1slo1121}
\eeq
$$\E_{L^1}^{sat}(h,g):=
\sum_{{n\leq n_{F}}\atop{p\leq k_{F}}}\sum_{1<k\leq k_{F}}\sum_{{\s}\atop{w< n\ep}}\sum_{{\underline{R}\in\r_{k}}\atop{\H^{n,p}_{\s,\underline{R}}\leq 2^{10 n \ep}}}
\sum_{R\in\r_{k}^{\s,[w]}(n,p,\underline{R})}\frac{|I_R\cap F|}{|I_R|}$$
$$\times\sum_{{\bar{\p}_{k,(l),a,b}^{2,n,(m),[p]}[R]\in\p_{k,re}^{2,n,[p]}[R]}
\atop{\p_{(l),a,b}^{n,(m),[p]}\in \F_{sat}}}
\int_{I_R} h\,{T_{c}^{\bar{\p}_{k,(l),a,b}^{2,n,(m),[p]}[R]}}^{*}(g)$$
with the obvious correspondence for the term $\E_{L^1}^{nsat}(h,g)$.

\begin{claim}\label{llp} The term  $\E_{L^1}^{nsat}(h,g)$ in \eqref{l1slo112} is an error term, that is
\beq \label{ertr}
|\E_{L^1}^{nsat}(h,g)|\lesssim |F|\,\log\frac{4}{|F|}\,\|f\|_{\infty}\,\|g\|_{\infty}.
\eeq
\end{claim}

\begin{proof}

Using now \eqref{regtree}, \eqref{refqt} and \eqref{const} we have
$$|\E_{L^1}^{nsat}(h,g)|=
|\sum_{{n\leq n_{F}}\atop{p\leq k_{F}}}\sum_{1<k\leq k_{F}}\sum_{{\s}\atop{w< n\ep}}\sum_{{\underline{R}\in\r_{k}}\atop{\H^{n,p}_{\s,\underline{R}}\leq 2^{10 n \ep}}}
\sum_{R\in\r_{k}^{\s,[w]}(n,p,\underline{R})}\frac{|I_R\cap F|}{|I_R|}$$
$$\times\sum_{{\bar{\p}_{k,(l),a,b}^{2,n,(m),[p]}[R]\in\p_{k,re}^{2,n,[p]}[R]}
\atop{\p_{(l),a,b}^{n,(m),[p]}\notin \F_{sat}}}
\int_{I_R} h\,{T_{c}^{\bar{\p}_{k,(l),a,b}^{2,n,(m),[p]}[R]}}^{*}(g)|$$
$$\lesssim\sum_{{n\leq n_{F}}\atop{p\leq k_{F}}}\sum_{1<k\leq k_{F}}\sum_{{\s}\atop{w< n\ep}}\sum_{{\underline{R}\in\r_{k}}\atop{\H^{n,p}_{\s,\underline{R}}\leq 2^{10 n \ep}}}
\sum_{R\in\r_{k}^{\s,[w]}(n,p,\underline{R})}\frac{|I_R\cap F|}{|I_R|}$$
$$\times\sum_{{\bar{\p}_{k,(l),a,b}^{2,n,(m),[p]}[R]\in\p_{k,re}^{2,n,[p]}[R]}
\atop{\p_{(l),a,b}^{n,(m),[p]}\notin \F_{sat}}}
|I_R|\,\frac{\|{T_{c}^{\bar{\p}_{k,(l),a,b}^{2,n,(m),[p]}[R]}}^{*}(g)\|_{L^2(I_R)}}{|I_R|^{\frac{1}{2}}}$$
$$\lesssim \sum_{{n\leq n_{F}}\atop{p\leq k_{F}}}\sum_{1<k\leq k_{F}}\sum_{{\s}\atop{w< n\ep}}\sum_{{\underline{R}\in\r_{k}}\atop{\H^{n,p}_{\s,\underline{R}}\leq 2^{10 n \ep}}}\sum_{R\in\r_{k}^{\s,[w]}(n,p,\underline{R})}$$
$$\times\sum_{{\bar{\p}_{k,(l),a,b}^{2,n,(m),[p]}[R]\in\p_{k,re}^{2,n,[p]}[R]}
\atop{\p_{(l),a,b}^{n,(m),[p]}\notin \F_{sat}}} 2^{n\ep}\,2^{-n}\,2^{-k}\,|I_{\bar{\p}_{k,(l),a,b,re}^{2,n,(m),[p]}[R]}^{R}|$$
$$\lesssim \sum_{{n\leq n_{F}}\atop{p\leq k_{F}}}\sum_{\p_{(l),a,b}^{n,(m),[p]}}
2^{-n\ep}\,2^{-n}\,2^{-m}\,|\tilde{I}_{\p_{(l),a,b}^{n,(m),[p]}}|\lesssim |F|\,\log\frac{4}{|F|}\,.$$
\end{proof}

We are now left with the treatment of the saturated term $\E_{L^1}^{sat}$. In this situation we can no longer use almost orthogonality techniques among various collections of trees. Thus, our very first step evolves as follows:
within each $0$ frequency tile $R$ we use triangle inequality - relative to $\|\cdot\|_1$ - among distinct frequency trees and then for each frequency $l$ we apply Cauchy-Schwarz and Corollary \ref{rowcont}. That is

\beq \label{lopp0}
\eeq
$$|\E_{L^1}^{sat}(h,g)|\lesssim \sum_{{n\leq n_{F}}\atop{p\leq k_{F}}}\sum_{1<k\leq k_{F}}\sum_{{\s}\atop{w< n\ep}}\sum_{{\underline{R}\in\r_{k}}\atop{\H^{n,p}_{\s,\underline{R}}\leq 2^{10 n \ep}}}
\sum_{R\in\r_{k}^{\s,[w]}(n,p,\underline{R})}|I_R\cap F|$$
$$\times\sum_{{\bar{\p}_{k,(l),a,b}^{2,n,(m),[p]}[R]\in\p_{k,re}^{2,n,[p]}[R]}
\atop{\p_{(l),a,b}^{n,(m),[p]}\in \F_{sat}}}
\int_{I_R\cap F} |{T_{c}^{\bar{\p}_{k,(l),a,b}^{2,n,(m),[p]}[R]}}^{*}(g)|$$
$$\lesssim \sum_{n\leq n_{F}}\,\sum_{p\leq k_{F}}\,\sum_{\p_{(l),a,b}^{n,(m),[p]}\in \F_{sat}} |F|^{\frac{1}{2}}\,\|{T^{\bar{\p}_{(l),a,b}^{n,(m),[p]}}}^{*}(g)\|_{L^(F)}$$
$$\lesssim \sum_{n\leq n_{F}}\,\sum_{p\leq k_{F}}\,\sum_{\p_{(l),a,b}^{n,(m),[p]}\in \F_{sat}} 2^{-n}\,2^{-m}\,|\tilde{I}_{P_{(l),a,b}^{n,(m),[p]}}|\:.$$

In what follows, for notational simplicity we will ignore the error term $\E_{L^1}^{nsat}(h,g)$ and thus wlog consider we have the following estimate

\beq \label{lopp}
|\E_{L^1}(h,g)|\lesssim \sum_{n\leq n_{F}}\,\sum_{p\leq k_{F}}\,\sum_{\p_{(l),a,b}^{n,(m),[p]}\in \F_{sat}} 2^{-n}\,2^{-m}\,|\tilde{I}_{P_{(l),a,b}^{n,(m),[p]}}|\:.
\eeq

\subsubsection{Tree foliation of the time-frequency plane depending on the mass and $F$-mass parameters}\label{foltrees1}

In this section, in parallel with the process described in Section \ref{foltree}, we develop an algorithm that organizes the family of tiles involved in $\E_{L^1}$ into maximal trees having uniform mass and $F-$mass parameters according to a ``descending foliation pattern" that focuses directly on the spacial support of these trees. This contrasts with the analysis performed in Section \ref{foltree} that was centered on the spacial support of the adjoint operators.

Recall first the decomposition made in Section \ref{foltree}, specifically relations \eqref{lndectree} - \eqref{Fmasdectree1n}.

Define now $\textbf{$\T\T$}$ as the set of all tops $P_{(l),a,b}^{n,(m),[p]}$ of the trees $\p_{(l),a,b}^{n,(m),[p]}\in \F_{sat}$.

Next, we let $\textbf{$\T\T$}(l)$ be the collection of all tops inside $\textbf{$\T\T$}$ which live at the frequency $l.$ Decompose this set into maximal layers  $\{\textbf{$\T\T$}(l,s)\}_{s\geq 1}$ as follows:
\begin{itemize}
\item for $s=1$ the set $\textbf{$\T\T$}(l,1)$ defines the collection of maximal, mutually disjoint top-trees at frequency $l$;

\item then, we remove the collection $\textbf{$\T\T$}(l,1)$ from $\textbf{$\T\T$}(l)$ and repeat the above selection procedure completing thus  the second level (layer) $\textbf{$\T\T$}(l,2)$.

\item this process will end up in finitely many steps; in fact one can say even more: there are at most $n_{F}+k_{F}$ such layers.
\end{itemize}

\begin{nt}
From now on, we redenote $P\in \textbf{$\T\T$}(l,s)$ as
\beq \label{plchar}
P_{(l),a}^{n,(m),\{s\}}\:,
\eeq
iff $P$ is the top of a maximal tree $\P_{(l),a}^{n,(m),\{s\}}$ with the following characteristics:
\begin{itemize}
\item lives at the frequency $\o_{l}$;
\item has uniform mass $\approx2^{-n}$;
\item has uniform $F-$mass $\approx2^{-m}$.
\end{itemize}
\end{nt}

Further on, we will apply the following convention: for $P_{(l),a}^{n,(m),\{s\}}\in \textbf{$\T\T$}(l,s)$ we define
\beq \label{keyrestr}
\textbf{$\T\T$}(l,s+1)[P_{(l),a}^{n,(m),\{s\}}]:=
\{P_{(l),a'}^{n',(m'),\{s+1\}}\in\textbf{$\T\T$}(l,s+1)\,|\,P_{(l),a'}^{n',(m'),\{s+1\}}< P_{(l),a}^{n,(m),\{s\}}\}\:.
\eeq

 We now have the following important

\begin{o0}\label{keytrdec}
Recalling Definitions \ref{Fmass} and \ref{mass}, and with the above notations, the following holds
\beq \label{key2}
\textrm{if}\:\:P_{(l),a'}^{n',(m'),\{s+1\}}\in\textbf{$\T\T$}(l,s+1)[P_{(l),a}^{n,(m),\{s\}}]\:\:\textrm{then}\footnote{The itemization below is not mutually exclusive.}
\eeq
\begin{itemize}
\item  either $m'=m-1$;
\item  or $n'=n-1$.
\end{itemize}
Also, remark that always $m-1\leq m'\leq m$ while $n'$ is not always/necessarily related to $n$; indeed, we have that if $m'=m$ then $n'=n-1$, while if $m'=m-1$ the value of $n'$ can be anything relative to $n$.
\end{o0}
\medskip

\subsubsection{$(E-)$mass versus $F-$mass dichotomy}\label{EF}
In what follows we introduce a partition of the set $\textbf{$\T\T$}(l)$ into two components
\beq \label{keyll}
\textbf{$\T\T$}(l)=\textbf{$\T\T^{E}$}(l)\cup \textbf{$\T\T^{F}$}(l)\;,
\eeq
with
\begin{itemize}
\item $\textbf{$\T\T^{E}$}(l)$ the\footnote{In the remaining part of the section we prefer to stress the fact that the standard mass introduced in Definition \ref{mass} quantifies the behavior of the measurable function $N(\cdot)$ in \eqref{carlac11} and is thus directly related with the properties of the set $E=\{ E(P)\}_{P\in\P_{sep}}$.} $(E-)$\textbf{mass} component;

\item $\textbf{$\T\T^{F}$}(l)$ the $F-$\textbf{mass} component.
\end{itemize}

This will be done according to the following \underline{selection algorithm}:

- let  $P_{(l),a}^{n,(m),\{s\}}\in \textbf{$\T\T$}(l,s)$ for some $s\in\N$; now

\begin{itemize}

\item \underline{\textsf{if}}
\beq \label{keyll1}
\sum_{P_{(l),a'}^{n',(m-1),\{s+1\}}\in\textbf{$\T\T$}(l,s+1)[\P_{(l),a}^{n,(m),\{s\}}]} 2^{-(m-1)}\,|\tilde{I}_{P_{(l),a'}^{n',(m-1),\{s+1\}}}|\leq \frac{7}{8}\,2^{-m}\,|\tilde{I}_{P_{(l),a}^{n,(m),\{s\}}}|\,
\eeq
\textsf{then} we declare  $P_{(l),a}^{n,(m),\{s\}}\in \textbf{$\T\T^{E}$}(l)$.

\item \underline{\textsf{otherwise}} we select $P_{(l),a}^{n,(m),\{s\}}\in \textbf{$\T\T^{F}$}(l)$.
\end{itemize}

Now, we let
$$\textbf{$\T\T^{E}$}:=\bigcup_{l}\textbf{$\T\T^{E}$}(l)\:\:\:\textrm{and}\:\:\:
\textbf{$\T\T^{F}$}:=\bigcup_{l}\textbf{$\T\T^{F}$}(l)\;.$$

Let us set now

\beq \label{lopp1}
\E_{L^1}^{E}:=\sum_{l}\sum_{n\leq n_{F}}\,\sum_{s\leq k_{F}}\,\sum_{P_{(l),a}^{n,(m),\{s\}}\in \textbf{$\T\T^{E}(l)$}} 2^{-n}\,2^{-m}\,|\tilde{I}_{P_{(l),a}^{n,(m),\{s\}}}|\:,
\eeq

\beq \label{lopp2}
\E_{L^1}^{F}:=\sum_{l}\sum_{n\leq n_{F}}\,\sum_{s\leq k_{F}}\,\sum_{P_{(l),a}^{n,(m),\{s\}}\in \textbf{$\T\T^{F}(l)$}} 2^{-n}\,2^{-m}\,|\tilde{I}_{P_{(l),a}^{n,(m),\{s\}}}|\:,
\eeq
and deduce from \eqref{lopp1}, \eqref{lopp2} and \eqref{lopp} that

\beq \label{lopp3}
|\E_{L^1}(h,g)|\lesssim \E_{L^1}^{E}\,+\,\E_{L^1}^{F}\:.
\eeq

\subsubsection{Treatment of the $E-$mass component}

We pass now to estimating the term $\E_{L^1}^{E}$.

We start by setting
\beq \label{not}
\textbf{$\T\T^{E,n}(l)$}:=\{P_{(l),a}^{n',(m),\{s\}}\in \textbf{$\T\T^{E}$}\,|\,n'=n\} \:.
\eeq
In a similar spirit with the row selection presented in the second half of Section \ref{rowrefinementt}, we perform a layer type decomposition of the set \textbf{$\T\T^{E,n}(l)$} as follows:
\beq \label{rowselecE}
\eeq
\begin{itemize}
\item we define \textbf{$\T\T^{E,n}(l)(1)$} to be the collection of maximal disjoint tops within the set \textbf{$\T\T^{E,n}(l)$};
\item next, we erase from \textbf{$\T\T^{E,n}(l)$} the set \textbf{$\T\T^{E,n}(l)(1)$} and repeat the previous procedure of selecting the collection of maximal disjoint tops into a new set \textbf{$\T\T^{E,n}(l)(2)$}.
\item iterate this procedure until the first time when the new initialized set \textbf{$\T\T^{E,n}(l)$} is void and then stop. This way we construct the successive layers $\{\textbf{$\T\T^{E,n}(l)(r)$}\}_{r\geq1}$.
\end{itemize}

Now first, we observe that based on the Carleson type condition \eqref{keyll1}, we deduce the following key relation:

for any $l$, $n\leq n_{F}$ and $P\in \textbf{$\T\T^{E,n}(l)$}$ we have

\beq \label{keynn}
\sum_{{P'< P}\atop{P'\in \textbf{$\T\T^{E,n}(l)$}}} \A_F(P')\,|\tilde{I}_{P'}| \leq \A_F(P)\,|\tilde{I}_{P}| \:,
\eeq
Combining now \eqref{lopp1} and \eqref{keynn} we have
\beq\label{lopp11}
\eeq
$$\E_{L^1}^{E}\approx \sum_{l}\sum_{n\leq n_{F}}\sum_{s\leq k_{F}}\,\sum_{r}\sum_{P_{(l),a}^{n,(m),\{s\}}\in \textbf{$\T\T^{E,n}(l)(r)$}} 2^{-n}\, 2^{-m}\,|\tilde{I}_{P_{(l),a}^{n,(m),\{s\}}}|$$
$$\lesssim \sum_{l}\sum_{n\leq n_{F}}\sum_{s\leq k_{F}}\sum_{P_{(l),a}^{n,(m),\{s\}}\in \textbf{$\T\T^{E,n}(l)(1)$}}  2^{-n}\, 2^{-m}\,|\tilde{I}_{P_{(l),a}^{n,(m),\{s\}}}|\:,$$
which further, by using Definition \ref{lp}, becomes
$$\lesssim \sum_{n\leq n_{F}}\sum_{p\leq k_{F}}\sum_{1<k\leq k_{F}}\sum_{I\in\I_k}\sum_{l}$$
$$\sum_{{\s}\atop{w< n\ep}}\sum_{{\underline{R}\in\r_{k}}\atop{\H^{n,p}_{\s,\underline{R}}\leq 2^{10 n \ep}}}
\sum_{R\in\r_{k}^{\s,[w]}(n,p,\underline{R})} \sum_{P_{(l),a,b}^{n,(m),[p]}\in \textbf{$\T\T^{E,n}(l)(1)$}} 2^{2n\ep}\,2^{-n}\,2^{-k}\,|I_{{\bar{\p}_{k,l,a,b,re}^{n,(m),[p],2}[R]}}^{R}|$$
$$\lesssim \sum_{1<k\leq k_{F}}\sum_{I\in\I_k}
\sum_{r\geq 0}\sum_{\underline{R}\in\r_{k}^{rU}[I]}
\sum_{R\in\r_{k}^{(r+1)U}(\underline{R})}
(\sum_{n\leq n_{F}} 2^{2n\ep}\,2^{-n}\,2^{10 n \ep})\,2^{-k}\,|I_{R}|\:,$$
where in the last inequality we made key use of the fact that in the second line above, the summation is restricted to $R\in\r_{k}^{\s,[w]}(n,p,\underline{R})$ with $\H^{n,p}_{\s,\underline{R}}\leq 2^{10 n \ep}$ and thus, for any such given $R$, we have that
\beq \label{srichsh}
\eeq
$$\#\{l\,|\,\bar{\p}_{k,l,a,b,re}^{n,(m),[p],2}[R]\,|\,P_{(l),a,b}^{n,(m),[p]}\in \textbf{$\T\T^{E,n}(l)(1)$}\}$$
$$\leq \#\{l\,|\,\bar{\p}_{k,(l)}^{2,n,[1]}[R]\in \p_{k,re}^{2,n,[1]}[R]\}=\H^{n,1}_{\s,\underline{R}}\leq 2^{10 n \ep}\,.$$

Finally, using \eqref{kk11} and \eqref{carlFol}, we conclude that
\beq\label{loppconc}
\E_{L^1}^{E}\lesssim \sum_{1<k\leq k_{F}} 2^{-k}\,\sum_{I\in \I_k}|I|\lesssim k_{F}\,|F|\:.
\eeq

\subsubsection{Treatment of the $F-$mass component}

We move now our attention towards the term $\E_{L^1}^{F}$.

In this context, we notice that if $\P_{(l),a}^{n,(m),\{s\}}\in \textbf{$\T\T^{F}$}(l)$ for some $s\in\N$ then \eqref{keyll1}
fails fact that can be re-written in a first instance as
\beq \label{keyll12}
\sum_{P_{(l),a'}^{n',(m-1),\{s+1\}}\in\textbf{$\T\T$}(l,s+1)[P_{(l),a}^{n,(m),\{s\}}]} 2^{-(m-1)}\,|\tilde{I}_{P_{(l),a'}^{n',(m-1),\{s+1\}}}|> \frac{7}{8}\,2^{-m}\,|\tilde{I}_{P_{(l),a}^{n,(m),\{s\}}}|\,,
\eeq
which thus further implies the key relation
\beq \label{keyll13}
\sum_{P_{(l),a'}^{n',(m), \{s+1\}}\in\textbf{$\T\T$}(l,s+1)[P_{(l),a}^{n,(m),\{s\}}]} 2^{-m}\,|\tilde{I}_{P_{(l),a'}^{n',(m),\{s+1\}}}|\leq \frac{1}{8}\,2^{-m}\,|\tilde{I}_{P_{(l),a}^{n,(m),\{s\}}}|\,.
\eeq
At this point, rephrasing Observation \ref{keytrdec} in our context, we notice that
\beq \label{keyll14}
\textrm{if}\:\p_{(l),a'}^{n',(m),\{s+1\}}\in\textbf{$\T\T$}(l,s+1)[\p_{(l),a}^{n,(m),\{s\}}] \:\:\textrm{then}\:\:n'= n-1\,.
\eeq
Consequently, from \eqref{keyll13} and \eqref{keyll14} we deduce that
\beq \label{keyll15}
\eeq
$$\sum_{\p_{(l),a'}^{n',(m),\{s+1\}}\in\textbf{$\T\T$}(l,s+1)[\p_{(l),a}^{n,(m),\{s\}}]} A(P_{(l),a'}^{n',(m),\{s+1\}})\,\A_{F}(P_{(l),a'}^{n',(m),\{s+1\}})\,|\tilde{I}_{P_{(l),a'}^{n',(m),\{s+1\}}}|$$
$$\leq \frac{1}{2} A(P_{(l),a}^{n,(m),\{s\}})\,\A_{F}(P_{(l),a}^{n,(m),\{s\}})\,|\tilde{I}_{P_{(l),a}^{n,(m),\{s\}}}|\,.$$
In symmetry with \eqref{not} we let
\beq \label{not1F}
\textbf{$\T\T^{F,[m]}(l)$}:=\{P_{(l),a}^{n,(m'),\{s\}}\in \textbf{$\T\T^{F}$}\,|\,m'=m\} \:.
\eeq
Also, applying an analogue procedure with the one presented in \eqref{rowselecE}, we decompose each set \textbf{$\T\T^{F,[m]}(l)$} into successive (finitely many) layers of maximal rows
\beq \label{not1FF}
\{\textbf{$\T\T^{F,[m]}(l)(r)$}\}_{r\geq 1}\:.
\eeq

Now \eqref{keyll15} implies the fundamental analogue of \eqref{keyn}:

for any $l$, $m\leq k_{F}$ and $P\in \textbf{$\T\T^{F,[m]}(l)$}$ we have
\beq \label{keyn}
\sum_{{P'< P}\atop{P'\in \textbf{$\T\T^{F,[m]}(l)$}}} A(P')\,\A_{F}(P')\,|\tilde{I}_{P'}| \leq  A(P)\,\A_{F}(P')\,|\tilde{I}_{P}| \:,
\eeq

With all these done, applying \eqref{keyn}, we have
\beq\label{loppconcF}
\eeq
$$\E_{L^1}^{F}\lesssim $$
$$\sum_{l}\sum_{r\leq n_{F}}\sum_{n\leq n_{F}}\sum_{m\leq k_{F}}\sum_{\p_{(l),a}^{n,(m),\{s\}}\in \textbf{$\T\T^{F,[m]}(l)(r)$}} A(P_{(l),a}^{n,(m),\{s\}})\,\A_{F}(P_{(l),a}^{n,(m),\{s\}})\,|\tilde{I}_{P_{(l),a}^{n,(m),\{s\}}}|$$
$$\lesssim \sum_{l}\sum_{n\leq n_{F}}\sum_{m\leq k_{F}}\sum_{\p_{(l),a}^{n,(m),\{s\}}\in \textbf{$\T\T^{F,[m]}(l)(1)$}} A(P_{(l),a}^{n,(m),\{s\}})\,\A_{F}(P_{(l),a}^{n,(m),\{s\}})\,|\tilde{I}_{P_{(l),a}^{n,(m),\{s\}}}|\:.$$
$$\lesssim \sum_{m\leq k_{F}}\sum_{l}\sum_{P\in \textbf{$\T\T^{F,[m]}(l)(1)$}} 2^{-m}\,A(P)\,|I_P|\:.$$
$$\lesssim \sum_{m\leq k_{F}}\sum_{l}\sum_{P\in \textbf{$\T\T^{F,[m]}(l)(1)$}} 2^{-m}\,|E(P)|\:,$$
where in the last line we used the special definition of the mass introduced in our current paper - i.e. Definition \ref{mass}, and the following key observation:
\beq\label{obsm1}
\textrm{if}\:P\in \textbf{$\T\T^{F,[m]}(l)(1)$}\:\textrm{then}\:A(P)=\frac{|E(P)|}{|I_P|}\,.
\eeq
Indeed, if  $P\in \textbf{$\T\T^{F,[m]}(l)(1)$}$ with $P\in\P_n$ then from Definition \ref{mass} we know that there must exist $P'\in P^{m}$ with $P\leq P'$ such that $A(P)=\frac{|E(P')|}{|I'|}\in(2^{-n-1}, 2^{-n}]$. However, from the construction of our trees described in the first half of Section \ref{foltree}, $P$ must be the top of a \emph{maximal} tree of a given $F-$mass and mass parameters. Thus since both $P,\,P'\in \P^m_n$ and $P\leq P'$ one must have $P=P'$.

We have now a last fundamental observation to record here:

Assume that we are given a (dyadic) interval $I\subset[0,1]$ and $\p=\p(I)\in\P$ any family of tiles obeying

\begin{itemize}
\item no two tiles in $\p$ are comparable under $``\leq"$;

\item for any $P\in\p$ one has that $I_{P}\subseteq I$.
\end{itemize}

Then, we have the following fundamental Carleson-packing type condition:

\beq\label{JNcon}
\sum_{P\in\p(I)} |E(P)|\leq |I|\:.
\eeq

Implementing now \eqref{JNcon} in \eqref{loppconcF} we conclude that
\beq\label{loppconcF111}
\sum_{l}\sum_{P\in \textbf{$\T\T^{F,[m]}(l)(1)$}} |E(P)|\lesssim |\tilde{I}_{m}|\,
\eeq
and hence
\beq\label{loppconcF1}
\eeq
$$\E_{L^1}^{F}\lesssim \sum_{m\leq k_{F}}\sum_{l}\sum_{P\in \textbf{$\T\T^{F,[m]}(l)(1)$}} 2^{-m}\,|E(P)|$$
$$\lesssim \sum_{m\leq k_{F}} 2^{-m}\,|\tilde{I}_{m}|\lesssim k_{F}\,|F|\:.$$

\section{Final remarks}\label{FR}

1) The techniques developed in this paper can easily be extended to the Walsh case. Indeed, recall first quickly several simple facts:

Fix $n\in\N$ and let
\beq\label{2base}
n=\sum_{i=0}^{\infty}\ep_{i}\,2^{i}\:\:\:\textrm{with}\:\:\:\ep_{i}\in\{0,\,1\}\,,
\eeq
be it's dyadic decomposition. We define the Walsh system $\{w_n\}_{n\in\N}$ as:
\beq\label{defw}
\eeq
\begin{itemize}
\item if $x\in\R\setminus[0,1)$ then $w_n(x)=0$ for any $n\in\N$;

\item if $x\in[0,1)$ and $n=0$ then we set $w_n(x)=1$;

\item if $x\in[0,1)$ and $n\geq 1$ obeys \eqref{2base} then we let\footnote{Notice that in \eqref{2base} only finitely many $\ep_{i}$'s are nonzero.}
$$w_{n}(x):=\prod_{i=0}^{\infty}\left( \textrm{sgn} (\sin 2^{i+1}\, \pi\,x)\right)^{\ep_i}\;.$$
\end{itemize}
Define the $n^{th}$ partial Walsh-sum as
\beq\label{Walshid}
W_{n}f(x):=\sum_{k=0}^{n} <f,\,w_k>\,w_k(x)\:.
\eeq
Now as a direct consequence of the methods developed in this paper, we have:

\begin{cor}\label{WW} Let $\{n_j\}_j$ be a lacunary sequence. Define the (lacunary) Walsh-Carleson operator as
\beq\label{WalshC}
C_{W}^{\{n_j\}_j}f(x):=\sup_{j\in\N}|W_{n_j}f(x)|\;.
\eeq
Then we have that
\beq\label{WalshC1}
C_{W}^{\{n_j\}_j}:\:L\log L\:\rightarrow\:L^1\:.
\eeq
Moreover, within the class of \textit{all} lacunary sequences this result is sharp.
\end{cor}

The fact that \eqref{WalshC1} holds can be proved by following the same steps as in the Fourier case. The second claim, that is that the $(L\log L,\,L^1)$ bound is sharp within the class of all lacunary sequences can be easily proved as follows:

Choose $\{n_j\}_j$ such that $n_j=2^j-1$. Then, from the properties of the Walsh system, we have
\beq\label{alg4}
W_{n_j}(f)(x)=<f(\cdot),\,\prod_{i=0}^{j-1} \left(r_{0}(x)\,r_{0}(\cdot)+r_{2^i}(x)\,r_{2^i}(\cdot)\right)>\;,
\eeq
where here $r_{2^i}(x):=\textrm{sgn}(\sin 2^{i+1} \,\pi\,x)$ stands for the $i^{th}$ Rademacher function.

Consequently, we deduce
$$C_{W}^{\{n_j\}_j}f(x)\approx M f(x)\,,$$
and choosing now $f(x)$ within the sequence $\{f_n\}_{n\geq 1}$ given by $f_n(x):=\min\{2^n,\,\frac{1}{x}\}$ on $[0,1]$ one immediately notices that
$$\|C_{W}^{\{n_j\}_j}f_n\|_1\approx n^2\approx \|f_n\|_{L\log L}\:.$$

2) Our main theorem is a first result in the time-frequency literature involving $(L\log L, L^1)$ bounds for a frequency modulated maximal operator of Carleson type that incorporates both the behavior of the Hardy-Littlewood maximal operator and that of the Hilbert transform. This result is possible due to the following key two features:
\beq\label{keys}
\eeq
\begin{itemize}
\item the new concept of the set-resolution of the time-frequency plane (at a fix frequency) which is further based on the concept of the time-frequency regularization of a set;

\item the $E-$mass versus $F-$mass dichotomy (see Section \ref{EF}) that brings a selection of the trees in two families according to
\eqref{keyll}, \eqref{keyll1} and \eqref{keyll13} respectively.
\end{itemize}

The first item is responsible with matching the structural properties of the set $F$ and of the lacunary Carleson operator. The second item realizes an antithesis/complementarity in terms of the impact of the sets $E$ and $F$ on the tiles: if we isolate all the trees at a given frequency then the tops of the trees obey one of the two Carleson packing conditions
corresponding to
\begin{itemize}
\item \eqref{keynn} for the $E-$mass component;

\item \eqref{keyn} for the $F-$mass component.
\end{itemize}
This dichotomy is finally the one that allows the control over the summation in the terms \eqref{lopp1} and \eqref{lopp2} having as the final output the desired $L\log L$ bound.

3) Let us briefly comment on Conjecture 2. Recall that in this case we ask wether or not the full Carleson operator $C$ maps $L\log L$ into $L^1$. Obviously in this situation the much increased complexity is given by the fact that the set of frequencies range over arbitrary sets of integers. Thus isolating structures within this set of frequencies relative to the structural properties of the set $F$ becomes a very challenging task. However, there are genuine hopes
about approaching this problem that reside on elaborating on itemization \eqref{keys}: the first item there can be transformed into a modulation invariant concept that we like to call it ``the set-resolution of the time-frequency plane". This seems the right tool to be used since in this general instance no frequency plays a central role. The second item described in \eqref{keys} is applied to each single fixed frequency and thus can be transferred without modifications to our general setting. One of the main difficulties, relies now on the fact that there are no simple analogues of \eqref{lac} below (see the Appendix). This is the critical point where we expect that techniques from additive combinatorics will play a major role. We hope to address this topic into a future work.

\section{Appendix\\Interaction between additive combinatorics and time-frequency analysis via
TFR}\label{Zygmund}

This section should be regarded as providing some motivation and further heuristics for the key concept introduced Section \ref{resol}. Our aim is to explain how the newly introduced concept of the time-frequency regularization of a set exploits the additive combinatoric properties of the set relative to lacunary Fourier series.

Indeed, we have the following

\begin{prop}\label{setresapp}
Assume we are given the following:
\begin{itemize}
\item a natural numbers $N\in\N$;

\item a measurable set $F\subseteq \TT$ with $|F|\approx 2^{-N}$;

\item a lacunary sequence $\{n_j\}_{j\in\N}\subset\N$;

\item a sequence $\{a_j\}_{j\in\N}\subset \C$ such that $|a_j|\approx 1$ for any $j\in\N$.
\end{itemize}

Then given any $M\in\N$ one has
\beq\label{sharpZ}
\int_{F}|\sum_{j=1}^{M} a_j\, e^{2\pi\,i\,n_j\,x}|\,dx\lesssim |F|\,\min\{N^{\frac{1}{2}},\,M^{\frac{1}{2}}\}\,M^{\frac{1}{2}}\:.
\eeq
Moreover inequality \eqref{sharpZ} is sharp in terms of parameters $M,\,N$.
\end{prop}

Assuming for the moment the above proposition we have the following immediate consequence:

\begin{cor}\label{Zyg} [\textsf{Zygmund's inequality}]

Given a lacunary sequence $\{n_j\}_{j\in\N}\subset\N$ and a sequence $\{a_j\}_{j\in\N}\in l^2(\N)$ one has
\beq\label{Z}
\|\sum_{j=1}^{\infty} a_j\, e^{2\pi\,i\,n_j\,\cdot}\|_{\exp(L^2)}\lesssim \|\{a_j\}_{j}\|_{l^2(\N)}\:.
\eeq
\end{cor}

\begin{proof}

Inspecting the \emph{proof} of our Proposition \ref{setresapp} we deduce that the following stronger relation holds:
\beq\label{sharpZmod}
\int_{F}|\sum_{j=1}^{\infty} a_j\, e^{2\pi\,i\,n_j\,x}|\,dx\lesssim |F|\,(\log\frac{4}{|F|})^{\frac{1}{2}}\,
(\sum_{j=1}^{\infty}|a_j|^2)^{\frac{1}{2}}\:.
\eeq
Exploiting now the duality between the Orlicz (Lorentz) spaces $L (\log L)^{\frac{1}{2}}$ and $\exp(L^2)$ we conclude that \eqref{Z} holds.
\end{proof}

\textbf{Proof of Proposition \ref{setresapp}}

We first notice that the case $M\leq N$ is trivial since in this situation \eqref{sharpZ} reduces simply to $L^1-L^{\infty}$ H\"older inequality. The fact that the estimate is sharp can be easily verified in the instance $F$ interval, say $F=[0,2^{-N}]$ and $n_j=2^{j}$ for $j\in\{1,\ldots,M\}$ with $M\leq N$.

Thus, we remain with the following statement to prove: for any $M>N$ one has
\beq\label{sharpZ1}
\int_{F}|\sum_{j=1}^{M} a_j\, e^{2\pi\,i\,n_j\,x}|\,dx\lesssim |F|\, N\,(\frac{M}{N})^{\frac{1}{2}}\:.
\eeq

Apply now the \textbf{TFR} algorithm for the set $F$ and recall the properties listed in Section \ref{propTFR}.

Set now the set of frequencies
\beq\label{defFfre}
\F:=\{n_j\}_{j=1}^{M}\,,
\eeq
where here $\{n_j\}_{j=1}^{M}$ is as given in our hypothesis.

With these done, we have
\beq\label{f1}
\eeq
$$S_{F}:=\int_{F}|\sum_{j=1}^{M} a_j\, e^{2\pi\,i\,n_j\,x}|\,dx\leq$$
$$\sum_{k=1}^{k_F}\sum_{I\in\I_k}\sum_{n}\sum_{{s_j\in\{L,\,U\}}\atop{1\leq j\leq n}}\sum_{\underline{R}\in\r_k^{s_1\ldots s_{n}}[I]} \sum_{R\in\r_k^{chi}(\underline{R})}\int_{F}|\sum_{j\in \f[R]} a_j\, e^{2\pi\,i\,n_j\,x}|\;.$$

\begin{obs}\label{lac} [\textsf{Relevance of lacunarity}]
Here we want to record the two key relations where we encode the information that $\F$ consists of lacunary frequencies:
\begin{itemize}
\item using similar reasonings with those from the proofs of Lemmas \ref{treecutp1} and \ref{difcont1}, one has for each $R\in\r_k[I]$
\beq\label{KL1}
\int_{F}|\sum_{j\in \f[R]} a_j\, e^{2\pi\,i\,n_j\,x}|\lesssim 2^{-k}\,\int_{I_{R}}|\sum_{j\in \f[R]} a_j\, e^{2\pi\,i\,n_j\,x}|\;.
\eeq
\item with the previous notations, one has
\beq\label{lac}
\int_{I_{\underline{R}}}|\sum_{j\in \f[R]} a_j\, e^{2\pi\,i\,n_j\,x}|^2\approx |I_{\underline{R}}|\,\sum_{j\in \f[R]} |a_j|^2\:.
\eeq
\end{itemize}
\end{obs}

Using now \eqref{f1} and \eqref{KL1} one deduces
\beq\label{f2}
S_{F}\lesssim\sum_{k=1}^{k_F}\sum_{I\in\I_k}\sum_{n}\sum_{{s_j\in\{L,\,U\}}\atop{1\leq j\leq n}}
\sum_{\underline{R}\in\r_k^{s_1\ldots s_{n}}[I]} \sum_{R\in\r_k^{chi}(\underline{R})}  2^{-k}\,\int_{I_{R}}|\sum_{j\in \f[R]} a_j\, e^{2\pi\,i\,n_j\,x}|\:.
\eeq

The following result is a subcase of the more involved Lemma \ref{controlbasis} whose proof was provided in Section \ref{thL2}:

\begin{claim}\label{c1} Fix $I\in\I_k$, $n\in\N$, $s_j\in\{U,L\}$ and take $\r_k^{s_1\ldots s_{n}}[I]$.
 With the definitions and notations from the previous section the following holds:
\beq\label{claimm}
\eeq
$$\sum_{\underline{R}\in\r_k^{s_1\ldots s_{n}}[I]} \sum_{R\in\r_k^{chi}(\underline{R})}
 \int_{I_{R}}|\sum_{j\in \f[R]} a_j\, e^{2\pi\,i\,n_j\,x}|$$
$$\lesssim (\sum_{\underline{R}\in \r_k^{s_1\ldots s_{n}}[I]} |I_{\underline{R}}|)\,(\sum_{j\in \F(\r_k^{s_1,\ldots s_{n},U}[I])}|a_j|^2)^{\frac{1}{2}}\:.$$
\end{claim}

Then we immediately deduce that

\beq\label{f3}
\eeq
$$\sum_{k=1}^{k_F}\sum_{I\in\I_k}\sum_{n}\sum_{{s_j\in\{L,\,U\}}\atop{1\leq j\leq n}}
\sum_{\underline{R}\in\r_k^{s_1\ldots s_{n}}[I]} \sum_{R\in\r_k^{chi}(\underline{R})}  2^{-k}\,\int_{I_{R}}|\sum_{j\in \f[R]} a_j\, e^{2\pi\,i\,n_j\,x}|$$
$$\lesssim \sum_{k=1}^{k_F} 2^{-k}\,\sum_{I\in\I_k}\sum_{n}\sum_{{s_j\in\{L,\,U\}}\atop{1\leq j\leq n}}
(\sum_{\underline{R}\in \r_k^{s_1\ldots s_{n}}[I]}\,|I_{\underline{R}}|)\,(\sum_{j\in \f[\r_k^{s_1\ldots s_{n},U}[I]]} |a_j|^2)^{\frac{1}{2}}$$

The next result is a particular instance of Lemma \ref{controlfreqbasis} whose proof was also provided in Section \ref{thL2}:

\begin{claim}\label{c2} Let $I\in\I_k$ and $\r_k^{s_1\ldots s_{n}}[I]$ as before. Then, the following holds:
\beq\label{claimm2}
\eeq
$$\sum_{{{s'_1=s_1\ldots s'_n=s_n}\atop{s'_{n+1}=L}}\atop{n'>n}}(\sum_{\underline{R}\in \r_k^{s'_1\ldots s'_{n'}}[I]} |I_{\underline{R}}|)\,(\sum_{j\in \F(\r_k^{s'_1,\ldots s'_{n'},U}[I])}|a_j|^2)^{\frac{1}{2}}$$
$$\lesssim (\sum_{\underline{R}\in \r_k^{s_1\ldots s_{r}}[I]} |I_{\underline{R}}|)\,(\sum_{j\in \F(\r_k^{s_1,\ldots s_{n},U}[I])}|a_j|^2)^{\frac{1}{2}}\:.$$
\end{claim}

Implementing \eqref{claimm2} in \eqref{f3} and putting this together with \eqref{f2}, we deduce that
\beq\label{f4}
S_{F}\lesssim\sum_{k=1}^{k_F}2^{-k}\,\sum_{I\in\I_k}\sum_{n}
(\sum_{\underline{R}\in \r_k^{n U}[I]} |I_{\underline{R}}|)\,(\sum_{j\in \F(\r_k^{(n+1)U}[I])}|a_j|^2)^{\frac{1}{2}}\:,
\eeq
where here we recall the notation \eqref{levru}.

We will use now the following two observations

\begin{itemize}
\item for any $1\leq k\leq k_{F}$ any\footnote{Recall the convention made in \eqref{conv1}.} $n\geq -1$ and any $I\in\I_k$, based on \eqref{CM} one has
\beq\label{kk1}
\sum_{R\in \r_k^{(n+2) U}[I]} |I_R|\leq \frac{1}{4}\sum_{R\in \r_k^{n U}[I]} |I_R|\:.
\eeq
\item for any $1\leq k\leq k_{F}$ any $n\not=n'$ and any $I\in\I_k$ one has
\beq\label{kk2}
\F(\r_k^{n U}[I])\cap \F(\r_k^{n' U}[I])=\emptyset\:.
\eeq
\end{itemize}

From the above itemization we deduce that it is enough to estimate $S_{F}^{n}$ for $n=-1$, where
\beq\label{fnn0}
S_{F}^{n}:=\sum_{k=1}^{k_F}2^{-k}\,\sum_{I\in\I_k}
(\sum_{R\in \r_k^{n U}[I]} |I_R|)\,(\sum_{j\in \F(\r_k^{(n+1) U}[I])}|a_j|^2)^{\frac{1}{2}}\:,
\eeq
since from \eqref{kk1} we will get a geometric decay in $n$.

Now, using Cauchy-Schwarz, we deduce that
\beq\label{LAC}
\eeq
$$S_{F}^{0}\lesssim \sum_{k=1}^{k_F}\sum_{I\in\I_k}
2^{-k}\,|I|\,(\sum_{j\in \F(\r_k^{1}[I])}|a_j|^2)^{\frac{1}{2}}$$
$$\lesssim (\sum_{k=1}^{k_F}\sum_{I\in\I_k} 2^{-k}\,|I| )^{\frac{1}{2}}\,(\sum_{k=1}^{k_F}
\sum_{I\in\I_k}|I\cap F|\sum_{j\in \F(\r_k^{1}[I])}|a_j|^2)^{\frac{1}{2}}$$
$$\lesssim (\sum_{k=1}^{k_F}\sum_{I\in\I_k}\,|I\cap F| )^{\frac{1}{2}}\,(\int_{F} \sum_{k=1}^{k_F}
\sum_{I\in\I_k}\chi_{I}\sum_{j\in \F(\r_k^{1}[I])}|a_j|^2)^{\frac{1}{2}}$$
$$\lesssim k_F^{\frac{1}{2}}\,|F|\,M^{\frac{1}{2}}\:.$$

This ends the proof of inequality \eqref{sharpZ}.

It remains now to show that \eqref{sharpZ1} is sharp.

We start with a general comment and continue with few remarks: in understanding the best possible upper bound for our inequality \eqref{sharpZ} the fundamental role is played by the interaction between the \textit{structure of the set} $F$ and the corresponding \textit{structure of the frequencies}  $\{n_j\}_j$. Indeed, let us notice the following:
\begin{itemize}
\item if $F$ is an \textit{interval}, then for any lacunary sequence $\{n_j\}_j$ relation \eqref{sharpZ} can be upgraded to:
\beq\label{sharpZz}
\int_{F}|\sum_{j=1}^{M} a_j\, e^{2\pi\,i\,n_j\,x}|\,dx\lesssim |F|\,(N\,+\,M^{\frac{1}{2}})\:.
\eeq
\item same holds if $F$ consists of a union of equidistant same-size intervals - this is what we call the arithmetic progression case.

\item in symmetry with this, if we fix $\{n_j=2^j\}_{j=1}^{M}$ then, depending on the structure of $F$, the best upper bound in \eqref{sharpZ} is
    \begin{itemize}
    \item $|F|\,(N\,+\,M^{\frac{1}{2}})$ if $F$ an interval or an arithmetic progression with low entropy (the number of equidistant same-size intervals that form $F$ is much smaller than $\min\{M, N\}$)

    \item $|F|\,M^{\frac{1}{2}}$ - which is the best possible bound - if $F$ is an arithmetic progression with high entropy (larger than $2^N$)
    \item $|F|\,\min\{N^{\frac{1}{2}},\,M^{\frac{1}{2}}\}\,M^{\frac{1}{2}}$ if $F$ has the structure of a suitable Cantor set as described a bit later below  - see \eqref{defF} (this corresponds to the generalized arithmetic progression case).
    \end{itemize}
\end{itemize}

From the above remarks it becomes transparent that in order to find an example for which the RHS of \eqref{sharpZ1} is sharp one needs to search for a set $F$ with a Cantor set structure, that is for which $F$ has the structure of a generalized arithmetic progression. Recalling that we only discuss the case $M>N$, wlog we can assume that $M=2^s\,N$ for some large $s\in\N$.

We stop here briefly to introduce some more notations: assuming that $J\subseteq[0,1]$ dyadic interval we set $\J_s(J):=\{I\subseteq J\,|\,I\:\textrm{dyadic},\:|I|=|J|\,2^{-s}\}$. Also, if $\A=\bigcup\A_j$ is a collection of dyadic intervals then we set  $\A^{rt}:=\bigcup\A_j^{rt}$ where $\A_j^{rt}$ is the right child of $\A$.

With this, the construction of $F$ is as follows:
\beq\label{defF}
\eeq
\begin{itemize}
\item Stage $1$ - we let $\F_0:=J_s([0,1])$;

\item Stage $2$ - we define $\F_1:=\bigcup_{I\in \F_0^{rt}}J_s(I)$;

\item we continue inductively to construct $\F_2$, $\F_3$ and so on until we reach

Stage $N$ - define $\F_N:=\bigcup_{I\in \F_{N-1}^{rt}}J_s(I)$;

\item finally, we simply set
$$F:=\F_{N}\:.$$
\end{itemize}
Choosing now $\{n_j=2^j\}_{j=1}^{M}$ one can simply check now that relations \eqref{fnn0} - \eqref{LAC} become approximate equalities\footnote{Up to harmless constants the bound from the above is of the same nature with the bound from below.}.

This ends our proof.

\end{document}